\documentclass[final,leqno]{siamltex}

\usepackage[cmex10]{amsmath}
\usepackage{amssymb} 

\usepackage{algorithmic}

\usepackage{array}

\usepackage{url}

\DeclareMathOperator*{\argmin}{arg\,min}
\DeclareMathOperator*{\mymax}{max}

\def\mplus{\mathrel{%
  \ooalign{\raise.29ex\hbox{$\scriptscriptstyle\mathbf{+}$}\cr}}}
\usepackage{color}
\usepackage{multirow}

\newcolumntype{x}[1]{%
>{\centering\hspace{0pt}}p{#1}}%

\usepackage{comment}
\newtheorem{remark}[theorem]{remark} 
\newtheorem{assumption}[theorem]{assumption} 

\usepackage{graphicx}
\usepackage{caption}
\usepackage{subcaption}
 \usepackage{float}

\begin{document}
\title{A Preconditioner for A Primal-Dual Newton Conjugate Gradients Method for Compressed Sensing Problems\\ \vspace{0.5cm}\textnormal{D\MakeLowercase{ecember} 30, 2014, \MakeLowercase{revised in} J\MakeLowercase{uly} 2, 2015}}

\author{Ioannis~Dassios\thanks{I. Dassios is with the School of Mathematics and Maxwell Institute, The University of Edinburgh,
Peter Guthrie Tait Road, Edinburgh EH9 3FD, United Kingdom e-mail: idassios@ed.ac.uk. I. Dassios is supported by EPSRC Grant EP/I017127/1} \and
        Kimon~Fountoulakis\thanks{K. Fountoulakis is with the School of Mathematics and Maxwell Institute, The University of Edinburgh,
Peter Guthrie Tait Road, Edinburgh EH9 3FD, United Kingdom e-mail: K.Fountoulakis@sms.ed.ac.uk.} \and
        Jacek~Gondzio\thanks{J. Gondzio is with the School of Mathematics and Maxwell Institute, The University of Edinburgh,
Peter Guthrie Tait Road, Edinburgh EH9 3FD, United Kingdom e-mail: J.Gondzio@ed.ac.uk.}
}

\markboth{Journal of \LaTeX\ Class Files,~Vol.~11, No.~4, December~2012}%
{Shell \MakeLowercase{\textit{et al.}}: Bare Demo of IEEEtran.cls for Journals}

\maketitle

\begin{abstract}
In this paper 
we are concerned with the solution of Compressed Sensing (CS) problems where the signals to be recovered are sparse in coherent and redundant dictionaries. 
We extend the primal-dual Newton Conjugate Gradients (pdNCG) method in \cite{ctpdnewton} for CS problems. We provide an inexpensive and provably effective preconditioning technique for linear systems
using pdNCG. 
Numerical results are presented on CS problems which demonstrate the performance of pdNCG with the proposed preconditioner compared to state-of-the-art existing solvers.
\end{abstract}


\section{Introduction}

CS is concerned with recovering a signal $\tilde x\in\mathbb{R}^{n}$ by observing a linear combination of the signal
$$
\tilde b = A \tilde x,
$$
where $A\in \mathbb{R}^{m\times n}$ is an under-determined linear operator with $m < n$
and $\tilde b \in \mathbb{R}^{m}$ are the observed measurements. 
Although this system has infinitely many solutions, reconstruction of $\tilde x$ is possible due to its assumed properties.
In particular, $\tilde x$ is assumed to have a sparse image through a coherent and redundant dictionary 
$W\in E^{n\times l}$, where $E=$ $\mathbb{R}$ or $\mathbb{C}$ and $n \le l$.
More precisely, $W^* \tilde x$, is sparse, i.e. it has only few non-zero components,
where the star superscript denotes the conjugate transpose.
If $W^* \tilde{x}$ is sparse, under certain conditions on matrices $A$ and $W$ (discussed in Subsection \ref{subsec:wrip}) 
the optimal solution of the linear problem
$$
\mbox{minimize} \ \|W^* x\|_1, \quad \mbox{subject to:} \quad Ax=\tilde b
$$
is $\tilde x$, where $\|\cdot\|_1$ is the $\ell_1$-norm. 

Frequently measurements $\tilde b$ might be contaminated with noise, i.e. one measures $b = \tilde b + e$ instead, 
where $e$ is a vector of noise, usually modelled as Gaussian with zero-mean and bounded Euclidean norm.
In addition, in realistic applications, $W^*\tilde x$ might not be exactly sparse, but its mass might be concentrated only on few of its components, while the rest
are rapidly decaying. 
In this case, (again under certain conditions on matrices $A$ and $W$) the optimal solution of the following problem
\begin{equation}\label{prob1}
\mbox{minimize} \  f_{c}(x) :=  c\|W^* x\|_1 + \frac{1}{2}\|Ax-b\|^2_2,
\end{equation}
is proved to be a good approximation to $\tilde x$. In \eqref{prob1}, $c$ is an a-priori chosen positive scalar and $\|\cdot\|_2$ is the Euclidean norm.

\subsection{Brief description of CS applications}\label{subsec:apps} 

An example of $W$ being redundant and coherent with orthonormal rows is the curvelet frame where an image is assumed to have an approximately sparse representation \cite{apps3}. 
Moreover, for radar and sonar systems it is frequent that Gabor frames are used
in order to reconstruct pulse trains from CS measurements \cite{apps4}.
For more applications a small survey is given in \cite{l1analysis}.
Isotropic Total-Variation (iTV) is another application of CS, which exploits the fact that digital images frequently have slowly varying pixels, except along edges. This property implies that digital images 
with respect to the discrete nabla operator, i.e. local differences of pixels, are approximately sparse.
For iTV applications, matrix $W\in \mathbb{C}^{n \times n}$ is square, complex and rank-deficient with $rank (W)=n-1$.
An alternative to iTV is $\ell_1$-analysis, where matrix $W$ is a Haar wavelet transform. However, a more pleasant to the eye reconstruction result is obtained by solving the iTV problem compared to the $\ell_1$-analysis problem,
see \cite{tvrobust}.


\subsection{Conditions and properties of CS matrices}\label{subsec:wrip}

There has been an extensive amount of literature studying conditions and properties of matrices $A$ and $W$ which guarantee recoverability of a good approximation of $\tilde x$ by solving problem \eqref{prob1}.
For a thorough analysis we refer the reader to \cite{csimplications,l1analysis,tvrobust}.
The previously cited papers use
a version of the well-known \textit{Restricted Isometry Property} (RIP) \cite{l1analysis}, which is repeated below.

\begin{definition}\label{def:1}
The restricted isometry constant of a matrix $A\in \mathbb{R}^{m\times n}$ adapted to $W\in {E}^{n\times l}$ is defined as the smallest $\delta_q$ such that
$$
(1-\delta_q)\|Wz\|_2^2 \le \|AWz\|_2^2 \le (1+\delta_q)\|Wz\|_2^2
$$
for all at most $q$-sparse $z\in {E}^{l}$, where $E=\mathbb{R} \mbox{ or } \mathbb{C}$.
\end{definition}

For the rest of the paper we will refer to Definition \ref{def:1} as W-RIP.
It is proved in Theorem $1.4$ in \cite{l1analysis} that if $W\in E^{n\times l}$ has orthonormal rows with $n \le l$ and if $A$, $W$ satisfy the W-RIP with $\delta_{2q} \le 8.0e$-$2$, then the solution $x_c$ obtained 
by solving problem \eqref{prob1}
satisfies
\begin{equation}\label{eq:160}
\|x_c - \tilde x\|_2 = C_0 \|e\|_2  + C_1 \frac{\|W^*x_c - (W^*\tilde x)_q\|_1}{\sqrt{q}},
\end{equation}
where $(W^*\tilde x)_q$ is the best $q$-sparse approximation of $W^*\tilde x$, $C_0$ and $C_1$ are small constants and only depend on $\delta_{2q}$. 
It is clear that $W^* \tilde x$ must have $l-q$ rapidly decaying components, in order for 
$\|x_c - \tilde x\|_2$ to be small and the reconstruction to be successful. 

iTV is a special case of $\ell_1$-analysis where matrix $W$ does not have orthonormal rows, hence, result \eqref{eq:160} does not hold. For iTV
there are no conditions on $\delta_{2q}$ such that a good reconstruction is assured. However, there exist results which directly impose restrictions on the number of measurements $m$, see Theorems $2$, $5$ and $6$ in \cite{tvrobust}.
Briefly, in these theorems it is mentioned that if $m \ge q \log(n)$ linear measurements are acquired for which matrices $A$ and $W$ satisfy the W-RIP for some $\delta_q < 1$, then, similar reconstruction guarantees 
as in \eqref{eq:160} are obtained for iTV.
Based on the previously mentioned results regarding reconstruction guarantees it is natural to assume that for iTV a similar condition applies, i.e. $\delta_{2q} < 1/2$. Hence, we make the following assumption.
\begin{assumption}\label{assum:1}
The number of nonzero components of $W^*x_c$, denoted by $q$, and the dimensions $l$, $m$, $n$ are such that matrices $A$ and $W$ satisfy W-RIP for some $\delta_{2q} < 1/2$. 
\end{assumption}
This assumption will be used in the spectral analysis of our preconditioner in Section \ref{sec:cs}.

Another property of matrix $A$ is the near orthogonality of its rows.
Indeed many applications in CS use matrices $A$ that satisfy
\begin{equation}\label{bd8}
\| AA^\intercal - I_m \|_2 \le \delta,
\end{equation}
with a small constant $\delta\ge 0$.
Finally, through the paper we will make use of the following assumption  
\begin{equation}\label{bd54}
\mbox{Ker}(W^*)\cap \mbox{Ker}(A) = \{0\},
\end{equation}
This is commonly used assumption in the literature, see for example \cite{fadili}.

\subsection{Contribution}
In \cite{ctpdnewton}, Chan, Golub and Mulet, proposed a primal-dual Newton Conjugate Gradients method for image denoising and deblurring problems.
In this paper we modify their method and adapt it for CS problems. There are two major contributions.

First, we propose an inexpensive preconditioner for fast solution of systems using pdNCG when applied to CS problems with coherent and redundant dictionaries. The proposed preconditioner 
is a generalization of the preconditioner in \cite{fountoulakisGondzio} for CS problems with incoherent dictionaries.
We analyze the limiting behaviour of our preconditioner and prove that the eigenvalues of the preconditioned matrices are clustered around one.
This is an essential property that guarantees that only a few iterations of CG will be needed to approximately solve the linear systems. Moreover, 
we provide computational evidence that the preconditioner works well not only close to the solution (as predicted by its spectral analysis) but also 
in earlier iterations of pdNCG.

Second, we demonstrate that despite being a second-order method, pdNCG can be more efficient than specialized first-order methods for CS problems of our interest,
even on large-scale instances.
This performance is observed in several numerical experiments presented in this paper. 
We believe that the reason for this is that pdNCG, as a second-order method, captures the curvature of the problems, which results in sufficient decrease in the number of iterations compared 
to first-order methods. This advantage comes with the computational cost of having to solve a linear system at every iteration. 
However, inexact solution of the linear systems using CG combined with the proposed
efficient preconditioner crucially reduces the computational costs per iteration.

\subsection{Format of the paper and notation}
The paper is organized as follows. 
In Section \ref{sec:smooth}, problem \eqref{prob1} is replaced by a smooth approximation; the $\ell_1$-norm is approximated by the pseudo-Huber function.
Derivation of pseudo-Huber function is discussed and its derivatives are calculated.
In Section \ref{section:optcond}, a primal-dual reformulation of the approximation to problem \eqref{prob1} and its optimality conditions are obtained.
In Section \ref{sec:pdNCG}, pdNCG is presented. For convergence analysis of pdNCG method the reader is referred to \cite{ctpdnewton,2ndpaperstrongly}.
In Section \ref{sec:cs}, a preconditioning technique is described for controlling the spectrum of matrices in systems which arise. 
In Section \ref{sec:cont}, a continuation framework for pdNCG is described.
In Section \ref{secNumExp}, numerical experiments are discussed that present the efficiency of pdNCG. Finally, in Section \ref{sec:concl}, conclusions are made. 

Throughout the paper, $\|\cdot\|_1$ is the $\ell_1$-norm, $\|\cdot\|_2$ is the Euclidean norm, $\|\cdot\|_\infty$ the infinity norm and $|\cdot|$ is the absolute value. The functions $Re(\cdot)$ and $Im(\cdot)$
take a complex input and return its real and imaginary part, respectively. For simplification of notation, occasionally we will use $Re(\cdot)$ and $Im(\cdot)$ without the parenthesis. 
Furthermore, $diag(\cdot)$ denotes the function which takes as input a vector and outputs a diagonal square matrix with the vector in
the main diagonal. Finally, the super index $c$ denotes the complementarity set, i.e. $\mathcal{B}^c$
is the complement set of $\mathcal{B}$.

\section{Regularization by pseudo-Huber}\label{sec:smooth}
In pdNCG \cite{ctpdnewton} the non-differentiability of the $\ell_1$-norm is treated by applying smoothing. 
In particular, the $\ell_1$-norm is replaced with the pseudo-Huber function \cite{Hartley2004}
\begin{equation}\label{bd6}
 \psi_\mu(W^* x) := \sum_{i=1}^l (({\mu^2+{|W_i^* x|^2}})^{\frac{1}{2}} - \mu),
\end{equation}
where $W_i$ is the $i^{th}$ column of matrix $W\in E^{n\times l}$ and $\mu$ controls the quality of approximation, i.e. for $\mu\to 0$, $\psi_\mu(x)$ tends to the $\ell_1$-norm.
The original problem \eqref{prob1} is approximated by
\begin{equation}\label{prob2}
\mbox{minimize} \  f_c^\mu(x) :=  c\psi_\mu(W^* x) + \frac{1}{2}\|Ax-b\|^2_2.
\end{equation}

%

\subsection{Derivation of pseudo-Huber function}
The pseudo-Huber function \eqref{bd6} can be derived in a few simple steps. 
First, we re-write function $\|W^*x\|_1$ in its dual form
\begin{equation}\label{dualpH}
\|W^*x\|_1 = \sup_{g\in\mathbb{C}^l, \|g\|_\infty \le 1} Re(\bar{g}^*W^*)x,
\end{equation}
where $g$ are dual variables. 
The pseudo-Huber function is obtained by regularizing the previous dual form
\begin{equation}
\psi_\mu(W^*x) = \sup_{g\in\mathbb{C}^l, \|g\|_\infty \le 1} Re(\bar{g}^*W^*)x + \sum_{i=1}^l \left( \mu (1 - |g_i|^2)^{\frac{1}{2}} - \mu \right),
\end{equation}
where $g_i$ is the $i^{th}$ component of vector $g$.

An approach which consists of smoothing a non-smooth function through its dual form is known as 
Moreau's proximal smoothing technique \cite{moreauprox}. Another way to smooth function $\|W^*x\|_1$ is to regularize its dual form with a strongly convex quadratic function $\mu/2\|g\|^2_2$.
Such an approach provides a smooth approximation of $\|W^*x\|_1$ which is known as Huber function and it has been used
in \cite{IEEEhowto:Nesta}. 
Generalizations of the Moreau proximal smoothing 
technique can be found in \cite{nesterovsmoothing} and \cite{beckteboullesmoth}. 

\subsection{Derivatives of pseudo-Huber function}
The gradient of pseudo-Huber function $\psi_{\mu}(W^* x)$ in \eqref{bd6} is given by
\begin{equation*}\label{bd58}
\nabla \psi_{\mu}(W^* x) = Re(W D W^*)x,
\end{equation*}
where $D := diag(D_1, D_2, \cdots, D_l)$ with
\begin{equation}\label{bd61}
D_i := (\mu^2+|y_i|^2)^{-\frac{1}{2}} \quad \forall i=1,2,\cdots,l,
\end{equation}
and $y =[y_1,y_2,\cdots,y_l]^\intercal := W^*x$.
The gradient of function $f_c^\mu(x)$ in \eqref{prob2} is
\begin{equation*}\label{bd101}
\nabla f_c^\mu(x) = c\nabla \psi_\mu(W^*x) + A^\intercal(Ax - b).
\end{equation*}
The Hessian matrix of $\psi_\mu(x)$ is
\begin{equation}\label{nabla2psi}
\nabla^2\psi_{\mu}(W^*x) :=\frac{1}{4}(W \hat{Y} W^* + \bar W \hat{ {Y}} \bar W^* + W \tilde{Y} \bar W^* + \bar W \tilde{\bar Y} W^*),
\end{equation}
where the bar symbol denotes the complex conjugate, 
$\hat Y :=diag\left[\hat{Y}_1,\hat{Y}_2,..., \hat{Y}_l\right]$, 
$\tilde{Y} :=diag\left[\tilde{Y}_1,\tilde{Y}_2,..., \tilde{Y}_l\right]$ and
\begin{equation}\label{bd62}
\hat{Y}_i:={\mu^2}D_i^{3}+D_i, \quad 
\tilde{Y}_i:=-{y_i^2}D_i^{3},\quad i=1,2,...,l,
\end{equation}
Moreover, the Hessian matrix of $f_c^\mu(x)$ is
\begin{equation}\label{eq131}
\nabla^2 f_c^\mu(x) = c\nabla^2 \psi_\mu(W^*x) + A^\intercal A.
\end{equation}

\section{Primal-dual formulation and optimality conditions}\label{section:optcond}
In \cite{ctnewtonold} the authors solved iTV problems for square and full-rank matrices $A$ which were inexpensively diagonalizable, i.e. image deblurring or denoising.  
More precisely, in the previous cited paper the authors tackled iTV problems using a Newton-CG method to solve problem \eqref{prob2}. 
They observed that close to the points of non-smoothness of the $\ell_1$-norm, the smooth pseudo-Huber function \eqref{bd6}
exhibited an ill-conditioning behaviour. This results in two major drawbacks of the application of Newton-CG. First, the linear algebra is challenging. Second,
the region of convergence of Newton-CG is substantially shrunk. To deal with these problems they proposed to 
incorporate Newton-CG inside a continuation procedure on the parameters $c$ and $\mu$. Although they showed that continuation did improve the global
convergence properties of Newton-CG it was later discussed in \cite{ctpdnewton} (for the same iTV problems) that continuation was difficult to control (especially for small $\mu$) and Newton-CG was not always convergent in reasonable CPU time. 

In \cite{ctpdnewton}, Chan, Golub and Mulet 
provided numerical evidence that the behaviour of a Newton-CG method can be made significantly more robust even for small values of $\mu$. This is achieved by simply solving a primal-dual reformulation of \eqref{prob2},
which is given below
\begin{equation}\label{primaldualform}
\mbox{minimize} \  \sup_{g\in\mathbb{C}^l, \|g\|_\infty \le 1} c Re(\bar{g}^*W^*)x + c \sum_{i=1}^l \left(\mu (1 - |g_i|^2)^{1/2} - \mu\right) + \frac{1}{2}\|Ax-b\|^2_2.
\end{equation}
The reason that Newton-CG method is more robust when applied on problem \eqref{primaldualform} than on problem \eqref{prob2} is hidden in the linearization of the optimality conditions 
of the two problems.

\subsection{Optimality conditions}
The optimality conditions of problem \eqref{prob2} are 
\begin{equation}\label{optcondprob2}
\nabla \psi_\mu(W^*x) + A^\intercal (Ax-b) = cRe(W D W^*)x +A^\intercal (Ax-b) = 0.
\end{equation}
The first-order optimality conditions of the primal-dual problem \eqref{primaldualform} are 
\begin{equation}\label{optcondprob2_2}
\begin{aligned}
c Re(W\bar{g}) +A^\intercal (Ax-b) = 0, \\ 
D^{-1}\bar{g} = W^* x.
\end{aligned}
\end{equation}
Notice for conditions \eqref{optcondprob2_2} that the constraint $\|g\|_\infty \le 1$ in \eqref{primaldualform} is redundant since any $x$ and $g$ that satisfy \eqref{optcondprob2_2} 
also satisfy this constraint. Hence, the constraint has been dropped. Conditions \eqref{optcondprob2_2} are obtained from \eqref{optcondprob2} by simply setting $\bar{g} = DW^*x$. Hence, their only difference is the inversion 
of matrix $D$. However, this small difference affects crucially the performance of Newton-CG.

The reason behind this is that the linearization of the second equation in \eqref{optcondprob2_2}, i.e. $D^{-1}\bar{g} = W^* x$, is of much better quality than the linearization of $\nabla \psi_\mu(W^*x)$ for $\mu \approx 0$
and $W^* x \approx 0$. To see why this is true, observe
that for small $\mu$ and $W^* x\approx 0$, the gradient $\nabla \psi_\mu(W^*x)$ becomes close to singular and its linearization is expected to be inaccurate. On the other hand, $D^{-1}\bar{g} = W^* x$
as a function of $W^*x$ is not singular for $\mu\approx 0$ and $W^*x \approx 0$, hence, its linearization is expected to be more accurate. 
We refer the reader to Section $3$ of \cite{ctpdnewton} for empirical justification.

\section{Primal-dual Newton conjugate gradients method}\label{sec:pdNCG}
In this section we present details of pdNCG method \cite{ctpdnewton}.

\subsection{The method}

First, we convert optimality conditions \eqref{optcondprob2_2} to the real case. 
This is done by splitting matrix $W = ReW + \sqrt{-1}Im W$ and the dual variables $g = g_{re} + \sqrt{-1} g_{im}$ into their real and imaginary parts. 
We do this in order to obtain optimality conditions which are differentiable in the classical sense of real analysis. 
This allows a straightforward application of pdNCG method \cite{ctpdnewton}. The real optimality conditions of the primal-dual problem \eqref{primaldualform}
are given below
\begin{equation}\label{bd53}
\begin{aligned}
c(ReWg_{re}+ImWg_{im}) +A^\intercal (Ax-b) = 0, \\ 
D^{-1}g_{re} = ReW^\intercal x, \quad D^{-1}g_{im} = ImW^\intercal x.
\end{aligned}
\end{equation}

At every iteration of pdNCG the primal-dual directions are calculated by approximate solving the following linearization of the equality constraints in \eqref{bd53} 
\begin{equation}\label{bd56}
\begin{aligned}
B\Delta x & =  -\nabla f_c^\mu(x) \\ 
\Delta g_{re}   & = D(I-B_1)ReW^\intercal \Delta x + DB_2ImW^\intercal \Delta x- g_{re} + DReW^\intercal x\\
\Delta g_{im}   & = D(I-B_4)ImW^\intercal \Delta x + DB_3ReW^\intercal \Delta x- g_{im} + DImW^\intercal x
\end{aligned}
\end{equation}
where 
\begin{equation}\label{bd55}
B := c\tilde{B} + A^\intercal A,
\end{equation} 
\begin{align*}
\tilde{B}  & := ReWD(I-B_1)ReW^\intercal + ImWD(I-B_4)ImW^\intercal + ReW DB_2 ImW^\intercal  \\ 
              &  \quad \ + ImW B_3D ReW^\intercal,
\end{align*}
and $B_i, i=1,2,3,4$ are diagonal matrices with components
\begin{eqnarray*}
&[B_1]_{ii} := D_{i}[g_{re}]_iReW_i^\intercal x,&  \quad [B_2]_{ii} := D_{i}[g_{re}]_iImW_i^\intercal x, \\
&[B_3]_{ii} := D_{i}[g_{im}]_iReW_i^\intercal x,& \quad  [B_4]_{ii} := D_{i}[g_{im}]_iImW_i^\intercal x.
\end{eqnarray*}
\begin{remark}\label{rem:1}
Matrix $B$ in \eqref{bd55} is positive definite (and invertible) if $\|g_{re} + \sqrt{-1}g_{im}\|_\infty \le 1$ and condition \eqref{bd54} are satisfied.
The former condition will be maintained through all iterations of pdNCG.
\end{remark}

It is straightforward to show the claim in Remark \ref{rem:1} for the case of $W$ being a real matrix. For the case of complex $W$ we refer
the reader to a similar claim which is made in \cite{ctpdnewton}, page $1970$.
Although matrix $B$ is positive definite under the conditions stated in Remark \ref{rem:1}, it is not symmetric, except in the case that $W$ is real where all imaginary parts are dropped.
Therefore in the case 
of complex matrix $W$, preconditioned CG (PCG) cannot be employed to approximately solve \eqref{bd56}.
To avoid the problem of non-symmetric matrix $B$ the authors in \cite{ctpdnewton} suggested to ignore the non-symmetric part in matrix $B$ and employ
CG to solve \eqref{bd56}. This idea is based on the following remark. 
\begin{remark}\label{rem:2}
The symmetric part of $B$ tends to the symmetric second-order derivative of $f_c^\mu(x)$ as pdNCG converges 
(see Section $5$ in \cite{ctpdnewton}). 
\end{remark}

Hence, system \eqref{bd56} is replaced with 
\begin{equation}\label{eq108}
\begin{aligned}
\hat{B}\Delta x  &  =  -\nabla f_c^\mu(x) \\
\Delta g_{re}   & = D(I-B_1)ReW^\intercal \Delta x + DB_2ImW^\intercal \Delta x- g_{re} + DReW^\intercal x\\
\Delta g_{im}   & = D(I-B_4)ImW^\intercal \Delta x + DB_3ReW^\intercal \Delta x- g_{im} + DImW^\intercal x
\end{aligned}
\end{equation}
where
\begin{equation}\label{eq109}
\hat{B} := c\,\mbox{sym}(\tilde{B}) + A^\intercal A
\end{equation} 
and $\mbox{sym}(\tilde{B}):= 1/2(\tilde{B} + \tilde{B}^\intercal)$ is the symmetric part of $\tilde{B}$. Moreover, PCG is terminated when 
\begin{equation}\label{bd59}
\|\hat{B}\Delta x + \nabla f_c^\mu(x)\|_2 \le \eta \|\nabla f_c^\mu(x)\|_2,
\end{equation}
is satisfied for $\eta \in [0,1)$. 
Then the iterate $g=g_{re} + \Delta g_{re} + \sqrt{-1}(g_{im} + \Delta g_{im})$ is orthogonally projected on the box $\{x:\left\|x\right\|_\infty\leq 1\}$.
The projection operator for complex arguments is applied component-wise and it is defined as
$
v :=  P_{\|\cdot\|_\infty\le1}(u) = \mbox{min}({1}/{|u|},1)\odot u
$, where $\odot$ denotes the component-wise multiplication.
In the last step, line-search is employed for the primal $\Delta x$ direction in order to guarantee that the objective value $f_c^\mu(x)$ is monotonically decreasing, see Section $5$ of \cite{ctpdnewton}.
The pseudo-code of pdNCG is presented in Figure \ref{fig:2}.

\begin{figure}
\begin{algorithmic}[1]
\vspace{0.1cm}
\STATE \textbf{Input:} $\tau_1\in(0,1)$, $\tau_2\in(0,1/2)$, $x^0$, $g_{re}^0$ and $g_{im}^0$, where $\|g_{re}^0 + \sqrt{-1}g_{im}^0\|_\infty \le 1$.
\STATE \textbf{Loop:} For $k=1,2,...$, until termination criteria are met. \vspace{0.1cm}
\STATE \hspace{0.5cm} Calculate $\Delta x^k$, $\Delta g_{re}^k$ and $\Delta g_{im}^k$ by solving approximately the system \eqref{eq108}, \\ \vspace{0.1cm}
              \hspace{0.5cm} until \eqref{bd59} is satisfied for some $\eta\in[0,1)$. \\ \vspace{0.1cm}
\STATE \hspace{0.5cm} Set 
$
\tilde{g}_{re}^{k+1}  :=  g_{re}^k + \Delta g_{re}^k 
$
,
$
\tilde{g}_{im}^{k+1}   :=  g_{im}^k + \Delta g_{im}^k
$
and calculate
\begin{equation*}
\bar{g}^{k+1}  :=  P_{\|\cdot\|_\infty\le1}(\tilde{g}_{re}^{k+1}  + \sqrt{-1}\tilde{g}_{im}^{k+1} ),
\end{equation*}
\hspace{0.5cm} where $P_{\|\cdot\|_\infty\le1}(\cdot)$ is the orthogonal projection on the $\ell_\infty$ ball. \\ \vspace{0.1cm} 
\hspace{0.5cm} Then set
${g}_{re}^{k+1} := Re \bar{g}^{k+1} $ and $\quad {g}_{im}^{k+1}:=Im\bar{g}^{k+1} $.
\vspace{0.1cm}
\STATE \hspace{0.5cm} Find the least integer $j\ge0$ such that
\begin{equation*}
f_c^\mu(x^k + \tau_1^j \Delta x^k) \le f_c^\mu(x^k) + \tau_2\tau_1^j (\nabla f_c^\mu(x^k))^\intercal \Delta x^k
\end{equation*}
 \hspace{0.5cm} and set $\alpha := \tau_1^j$.
\STATE \hspace{0.5cm}  Set $x^{k+1} := x^k + \alpha \Delta x^k$. 
\end{algorithmic}
\caption{Algorithm primal-dual Newton Conjugate Gradients}
\label{fig:2}
\end{figure}

\section{Preconditioning} \label{sec:cs}

Practical computational efficiency of pdNCG applied to system \eqref{eq108} depends on spectral properties of matrix $\hat{B}$ in \eqref{eq109}. Those can be improved by a suitable preconditioning. In this section we introduce a new preconditioner for $\hat{B}$ and discuss the limiting behaviour of the spectrum of preconditioned $\hat{B}$.

First,  we give an intuitive analysis on the construction of 
the proposed preconditioner.
In Remark \ref{rem:3} it is mentioned that the distance $\omega$ of the two solutions $x_{c}:=\argmin f_c(x)$ and $x_{c,\mu}:=\argmin f_c^\mu (x)$ can be arbitrarily small for sufficiently small values of $\mu$.
Moreover, according to Assumption \ref{assum:1}, $W^*x_c$ is $q$ sparse. Therefore, Remark \ref{rem:3} implies that $W^*x_{c,\mu}$ is approximately $q$ sparse
with nearly zero components of $\mathcal{O}(\omega)$. A consequence of the previous 
statement is that the components of $W^*x_{c,\mu} $ split into the following disjoint sets
\begin{equation*}
\begin{aligned}
&\mathcal{B}  := \{i \in \{1,2,\cdots,l\}\ | \ |W^*_ix_{c,\mu} |\gg \mathcal{O}(\omega)\}, 	&   |\mathcal{B}|=q=|\mbox{supp}(W^*x_c)|,\\
& & \\
&\mathcal{B}^c  := \{i \in \{1,2,\cdots,l\}\ | \ |W^*_ix_{c,\mu} |\approx \mathcal{O}(\omega)\}, &    |\mathcal{B}^c|=l - q.
\end{aligned}
\end{equation*}
The behaviour of $W^*x_{c,\mu}$ has a crucial influence on matrix $\nabla^2 \psi_\mu(W^*x_{c,\mu})$ in \eqref{nabla2psi}.
Notice that the components of the diagonal matrix $D$, defined in \eqref{bd61} as part of $\nabla^2 \psi_\mu(W^*x_{c,\mu})$, split into two disjoint sets. 
In particular, $q$ components are non-zeros much less than $\mathcal{O}({1}/{\omega})$, while the majority, $l-q$, of its components are of $\mathcal{O}({1}/{\omega})$, 
\begin{equation}\label{eq:2}
D_i \ll \mathcal{O}(\frac{1}{\omega}) \quad \forall i\in\mathcal{B}  \quad \mbox{and} \quad D_i = \mathcal{O}(\frac{1}{\omega}) \quad \forall i\in\mathcal{B}^c.
\end{equation}
Hence, for points close to $x_{c,\mu}$ and small $\mu$, matrix $\nabla^2 f_c^\mu(x)$ in \eqref{eq131} consists of a dominant matrix $c\nabla^2 \psi_\mu(x) $ and
of matrix $A^\intercal A$ with moderate largest eigenvalue. The previous argument for $A^\intercal A$ is due to \eqref{bd8}. Observe that $\lambda_{max}(A^\intercal A)=\lambda_{max}(A A^\intercal)$,
hence, if $\delta$ in \eqref{bd8} is not a very large constant, then $\lambda_{max}(A^\intercal A) \le 1+\delta $.
According to Remark \ref{rem:2}, the symmetric matrix $\mbox{sym}(\tilde{B})$ in \eqref{eq131} tends to matrix $\nabla^2 \psi_\mu(x)$ as $x \to x_{c,\mu}$. 
Therefore, matrix $\mbox{sym}(\tilde{B})$ is the dominant matrix in $\hat{B}$.
For this reason, in the proposed preconditioning technique, matrix $A^\intercal A$ in \eqref{eq131} is replaced by a scaled identity $\rho I_n$, $\rho>0$, while 
the dominant matrix $\mbox{sym}(\tilde{B})$ is maintained.
Based on these observations we propose the following preconditioner
\begin{equation}\label{bd10}
\tilde{N} := c\, \mbox{sym}(\tilde{B}) + \rho I_n.
\end{equation}

In order to capture the approximate separability of the diagonal components of matrix $D$ 
for points close to $x_{c,\mu}$, when $\mu$ is sufficiently small,
we will work with approximate guess of $\mathcal{B}$ and $\mathcal{B}^c$. 
For this reason, we introduce the positive constant $\nu$,
such that
$$
\# (D_i < \nu) = \sigma.
$$
Here $\sigma$ might be different from the sparsity of $W^*x_c$. Furthermore, according to the above definition we have the sets
\begin{equation}\label{eq:1}
\mathcal{B}_\nu := \{i \in \{1,2,\cdots,l\}\ | \ D_i < \nu \} \quad \mbox{and} \quad \mathcal{B}_\nu^c :=  \{1,2,\cdots,l\} \backslash \mathcal{B}_\nu,
\end{equation}
with $|\mathcal{B}_\nu|=\sigma$ and $|\mathcal{B}_\nu^c|=l-\sigma$. 
This notation is being used in the following theorem, in which
we analyze the behaviour of the spectral properties of preconditioned $\nabla^2 f_c^{\mu}(x)$, with preconditioner
$N:= c\nabla^2 \psi_\mu(W^*x) + \rho I_n$.
However, according to Remark \ref{rem:2} matrices $\hat{B}$ and $\tilde{N}$ tend to $\nabla^2 f_c^\mu(x)$ and $N$, respectively, as $x\to x_{c,\mu}$.
Therefore, the following theorem is useful for
the analysis of the limiting behaviour of the spectrum of preconditioned $\hat{B}$.
\begin{theorem}\label{thm:3}
Let $\nu$ be any positive constant and $\#(D_i < \nu) = \sigma$ at a point $x$, where $D$ is defined in \eqref{bd61}. 
Let 
$$
\nabla^2 f_c^\mu(x) = c\nabla^2 \psi_\mu(W^*x) + A^\intercal A \quad \mbox{and} \quad N:= c\nabla^2 \psi_\mu(W^*x) + \rho I_n.
$$
Additionally, let $A$ and $W$ satisfy W-RIP
with some constant $\delta_{\sigma} < 1/2$ and let $A$ satisfy \eqref{bd8} for some constant $\delta\ge0$.\\
If the eigenvectors of $N^{-\frac{1}{2}}\nabla^2 f_c^\mu(x) N^{-\frac{1}{2}} $ do not belong in $\mbox{Ker}(W^*_{\mathcal{B}_\nu^c})$ and $\rho\in[\delta_{\sigma},1/2]$, then the eigenvalues of $N^{-1}\nabla^2 f_c^\mu(x)$ satisfy
$$
|\lambda -1| \le \frac{1}{2}\frac{{\chi + 1} + (5\chi^2 - 2\chi + 1)^{\frac{1}{2}}}{ c\mu^2\nu^3\lambda_{min}(\mbox{Re}(W_{\mathcal{B}_\nu^c}W^*_{\mathcal{B}_\nu^c}))  + \rho},
$$
where $\lambda\in\mbox{spec}(N^{-1}\nabla^2 f_c^\mu(x) )$, $\lambda_{min}(\mbox{Re}(W_{\mathcal{B}_\nu^c}W^*_{\mathcal{B}_\nu^c}))$ is the minimum nonzero eigenvalue of $\mbox{Re}(W_{\mathcal{B}_\nu^c}W^*_{\mathcal{B}_\nu^c})$
and  $\chi:= 1+\delta - \rho$. \\
If the eigenvectors of $N^{-\frac{1}{2}}\nabla^2 f_c^\mu(x) N^{-\frac{1}{2}} $ belong in $\mbox{Ker}(W^*_{\mathcal{B}_\nu^c})$, then
$$
|\lambda -1| \le \frac{1}{2}\frac{{\chi + 1} + (5\chi^2 - 2\chi + 1)^{\frac{1}{2}}}{\rho}.
$$
\end{theorem}
\begin{proof}
We analyze the spectrum of matrix $N^{-\frac{1}{2}}\nabla^2 f_c^\mu(x) N^{-\frac{1}{2}}$ instead, because it has the same eigenvalues
as matrix $N^{-1} \nabla^2 f_c^\mu(x)$. We have that
\begin{align*}
N^{-\frac{1}{2}}\nabla^2 f_c^\mu(x) N^{-\frac{1}{2}} & =  N^{-\frac{1}{2}}(c\nabla^2\psi_{\mu}(x) + A^\intercal A ) N^{-\frac{1}{2}} \\  
                                                                     & =  N^{-\frac{1}{2}}(c\nabla^2\psi_{\mu}(x) + A^\intercal A  + \rho I_n - \rho I_n) N^{-\frac{1}{2}} \\
                                                                     & =  N^{-\frac{1}{2}}(c\nabla^2\psi_{\mu}(x) + \rho I_n ) N^{-\frac{1}{2}} +  N^{-\frac{1}{2}} A^\intercal A N^{-\frac{1}{2}} - \rho N^{-1}\\
                                                                     & = I_n  + N^{-\frac{1}{2}} A^\intercal A N^{-\frac{1}{2}} - \rho N^{-1}
\end{align*}
Let $u$ be an eigenvector of $N^{-\frac{1}{2}}\nabla^2 f_c^\mu(x) N^{-\frac{1}{2}} $ with $\|u\|_2=1$ and $\lambda$ the corresponding eigenvalue, then 
\begin{align}\label{eq:501}
(I_n + N^{-\frac{1}{2}} A^\intercal A N^{-\frac{1}{2}} - \rho N^{-1})u & = \lambda u& \Longleftrightarrow \nonumber \\
(N + N^{\frac{1}{2}} A^\intercal A N^{-\frac{1}{2}} - \rho I_n )u          & = \lambda N u& \Longrightarrow \nonumber \\
u^\intercal N^{\frac{1}{2}} (A^\intercal A - \rho I_n  )N^{-\frac{1}{2}}  u          & = (\lambda -1)u^\intercal N u& \Longrightarrow \nonumber \\
|u^\intercal N^{\frac{1}{2}} (A^\intercal A   - \rho I_n  )N^{-\frac{1}{2}}u|          & = |\lambda -1|u^\intercal N u.&
\end{align}
First, we find an upper bound for $|u^\intercal N^{\frac{1}{2}} (A^\intercal A   - \rho I_n  )N^{-\frac{1}{2}}u|$. Matrices $N^{\frac{1}{2}} (A^\intercal A   - \rho I_n  )N^{-\frac{1}{2}}$
and $A^\intercal A   - \rho I_n$ have the same eigenvalues. Therefore, 
\begin{align*}
|u^\intercal N^{\frac{1}{2}} (A^\intercal A   - \rho I_n  )N^{-\frac{1}{2}}u| \le \lambda_{max}^{+}(A^\intercal A - \rho I_n)
\end{align*}
where $\lambda_{max}^{+}(\cdot)$ is the largest eigenvalue of the input matrix in absolute value. Thus,
\begin{align*}
|u^\intercal N^{\frac{1}{2}} (A^\intercal A   - \rho I_n  )N^{-\frac{1}{2}}u|  &\le \mymax_{\|v\|_2^2\le 1} |v^\intercal(A^\intercal A - \rho I_n) v | \\
          												  & =  \mymax_{\|Pv\|_2^2 + \|Qv\|_2^2\le 1} |(Pv + Qv)^\intercal(A^\intercal A - \rho I_n) (Pv + Qv) |, 
\end{align*}
where $P$ is the projection matrix to the column space of $W_{\mathcal{B}_\nu}$ and $Q=I_n - P$. Using triangular inequality we get
\begin{align*}
|u^\intercal N^{\frac{1}{2}} (A^\intercal A   - \rho I_n  )N^{-\frac{1}{2}}u|  & \le  \mymax_{\|Pv\|_2^2 + \|Qv\|_2^2\le 1} \big(|(Pv)^\intercal (A^\intercal A - \rho I_n)Pv|  \nonumber \\
													   & + |(Qv)^\intercal (A^\intercal A - \rho I_n)Qv| 
 													      + 2 |(Pv)^\intercal (A^\intercal A - \rho I_n)Qv|\big).
\end{align*}
Let us denote by $\hat v$ the solution of this maximization problem and set $\|P \hat v\|_2^2=\alpha$ and $\|Q \hat v\|_2^2=1-\alpha$, where $\alpha \in [0,1]$, then
\begin{align}\label{eq:152}
|u^\intercal N^{\frac{1}{2}} (A^\intercal A   - \rho I_n  )N^{-\frac{1}{2}}u|  & \le \big(|(P \hat v)^\intercal (A^\intercal A - \rho I_n)P \hat v|  \nonumber \\
													   & + |(Q \hat v)^\intercal (A^\intercal A - \rho I_n)Q \hat v| 
 													      + 2 |(P \hat v)^\intercal (A^\intercal A - \rho I_n)Q \hat v|\big).
\end{align}
Since $P \hat v$ belongs to the column space of $W_{\mathcal{B}_\nu}$ and $|\mathcal{B}_\nu|=\sigma$, from W-RIP with $\delta_\sigma < 1/2$ we have that
\begin{align*}
\|P\hat v\|_2^2 (1-\delta_\sigma) & \le \|AP \hat v\|_2^2 & \Longrightarrow \\
\|P\hat v\|_2^2 (1-\rho) &\le \|AP\hat v\|_2^2 & \Longleftrightarrow \\
\|P\hat v\|_2^2 (1-2\rho) &\le \|AP\hat v\|_2^2 - \rho \|P\hat v\|_2^2.  & 
\end{align*}
Since $\rho\in [\delta_\sigma, 1/2]$ we have that 
$
\rho \|P\hat v\|_2^2 \le \|AP\hat v\|_2^2,
$
which implies that if the eigenvector corresponding to an eigenvalue of matrix $A^\intercal A$ belongs to the column space of $W_{\mathcal{B}_\nu}$, then
the eigenvalue cannot be smaller than $\rho$.
Hence,
$$
 |(P\hat v)^\intercal (A^\intercal A - \rho I_n)P\hat v| \le |(P\hat v)^* (A^\intercal A - \rho I_n)P\hat v| = (P\hat v)^* (A^\intercal A - \rho I_n)P\hat v.
$$
 Moreover, from W-RIP with $\delta_\sigma<1/2$ and $\rho\in [\delta_\sigma, 1/2]$, we also have that $(P\hat v)^* (A^\intercal A - \rho I_n)P\hat v \le \|P\hat v\|_2^2$. 
 Thus,
  \begin{equation}\label{eq:1000}
 |(P\hat v)^\intercal (A^\intercal A - \rho I_n)P\hat v| \le \alpha.
 \end{equation}
 From property \eqref{bd8} and $\lambda_{max}(A^\intercal A)=\lambda_{max}(A A^\intercal)$, we have that 
 $\lambda_{max}(A^\intercal A - \rho I_n) \le 1+\delta - \rho$. Finally, using the Cauchy-Schwarz inequality, we get that
 \begin{equation}\label{eq:1001}
|(Q\hat v)^\intercal (A^\intercal A - \rho I_n)Q\hat v| \le (1+\delta -\rho)(1-\alpha)
\end{equation}
and
 \begin{equation}\label{eq:1002}
 |(P \hat v)^\intercal (A^\intercal A - \rho I_n)Q\hat v|\le (1+\delta - \rho)\sqrt{\alpha(1-\alpha)}.
 \end{equation}
Using \eqref{eq:1000}, \eqref{eq:1001} and \eqref{eq:1002} in \eqref{eq:152} we have that 
\begin{align} \label{eq:153}
|u^\intercal N^{\frac{1}{2}} (A^\intercal A   - \rho I_n  )N^{-\frac{1}{2}}u|  
													   & \le \alpha + (1+\delta - \rho)(1-\alpha) + 2(1+\delta - \rho)\sqrt{\alpha(1-\alpha)}. 
\end{align}
Set 
$\chi:= 1+\delta - \rho$,
it is easy to check that in the interval $\alpha \in [0,1]$ the right hand side of \eqref{eq:153} has a maximum at one of the four candidate points
$$
\alpha_1 = 0, \quad \alpha_2 = 1, \quad \alpha_{3,4} = \frac{1}{2}(1 \pm \left(\frac{(\chi - 1)^2}{5\chi^2 - 2\chi + 1}\right)^{1/2}),
$$
where $\alpha_3$ is for plus and $\alpha_4$ is for minus. The corresponding function values are
$$
\chi, \quad 1, \quad \frac{\chi + 1}{2} + \frac{1}{2}\frac{3\chi^2 + 2\chi -1}{(5\chi^2 - 2\chi + 1)^{{1}/{2}}}, \quad  \frac{\chi + 1}{2} + \frac{1}{2}(5\chi^2 - 2\chi +1 )^{{1}/{2}},
$$
respectively.
Hence, the maximum among these four values is given for $\alpha_4$.
Thus, \eqref{eq:153} is upper bounded by
\begin{equation}\label{eq:154}
|u^\intercal N^{\frac{1}{2}} (A^\intercal A   - \rho I_n  )N^{-\frac{1}{2}}u|  \le \frac{\chi + 1}{2} + \frac{1}{2}(5\chi^2 - 2\chi + 1)^{\frac{1}{2}}.
\end{equation}
We now find a lower bound for $u^\intercal N u$. 
Using the definition of $D$ in \eqref{bd61}, matrix $\hat Y$ in \eqref{bd62} is rewritten as
$
\hat Y_i = (2\mu^2 + |y_i|^2) D_{i}^3\   \forall i=1,2,\cdots,l.
$
Thus $\nabla^2 \psi_\mu (x)$ in \eqref{nabla2psi} is rewritten as
\begin{align}\label{eq120}
\nabla^2\psi_{\mu}(W^*x) & =\frac{1}{4}[(W \tilde{D}^3 W^* + \bar W \tilde{D}^3 \bar W^* + W \tilde{Y} \bar W^* + \bar W \tilde{\bar Y} W^*) \\ 
                                         & + 2\mu^2 (W D^3 W^* + \bar W D^3 \bar W^*)],\nonumber
\end{align}
where $\tilde{D}_{i} =|y_i|^2 D_{i}^3$ $\forall i=1,2,\cdots,l$.
Observe, that matrix $\nabla^2\psi_{\mu}(W^*x)$ consists of two matrices $W \tilde{D}^3 W^* + \bar W \tilde{D}^3 \bar W^* + W \tilde{Y} \bar W^* + \bar W \tilde{\bar Y} W^*$ and 
$2\mu^2 (W D^3 W^* + \bar W D^3 \bar W^*)$ which are positive semi-definite. 
Using \eqref{eq120} and the previous statement 
we get that 
\begin{align*}
u^\intercal N u & = u^\intercal (c\nabla^2\psi_\mu(W^*x) + \rho I_n) u \\
                        & = \frac{c}{4}u^\intercal(W \hat{Y} W^* + \bar W \hat{ {Y}} \bar W^* + W \tilde{Y} \bar W^* + \bar W \tilde{\bar Y} W^* )u +\rho\\
                        & \ge \frac{c\mu^2}{2}u^\intercal (WD^3W^* + \bar{W}D^3\bar{W}^*)u +\rho.
\end{align*}
Furthermore, using the splitting of matrix $D$ \eqref{eq:1}, the last inequality is equivalent to
\begin{align*}
u^\intercal N u & = \frac{c\mu^2}{2}u^\intercal (W_{\mathcal{B}_\nu}D^3_{\mathcal{B}_\nu} W^*_{\mathcal{B}_\nu} + W_{\mathcal{B}_\nu^c}D^3_{\mathcal{B}_\nu^c} W^*_{\mathcal{B}_\nu^c} + 
                        \bar{W}_{\mathcal{B}_\nu}D^3_{\mathcal{B}_\nu}\bar{W}^*_{\mathcal{B}_\nu} + \bar{W}_{\mathcal{B}_\nu^c}D^3_{\mathcal{B}_\nu^c}\bar{W}^*_{\mathcal{B}_\nu^c})u   +\rho  \\
                        & \ge \frac{c\mu^2}{2}u^\intercal(W_{\mathcal{B}_\nu^c}D^3_{\mathcal{B}_\nu^c} W^*_{\mathcal{B}_\nu^c}  + \bar{W}_{\mathcal{B}_\nu^c}D^3_{\mathcal{B}_\nu^c}\bar{W}^*_{\mathcal{B}_\nu^c})u  + \rho. 
\end{align*}
Using the defition of $\mathcal{B}_\nu^c$ \eqref{eq:1} in the last inequality, the quantity $u^\intercal N u $ is further lower bounded by
\begin{equation}\label{eq:150}
                     u^\intercal N u    \ge \frac{c\mu^2\nu^3}{2}u^\intercal(W_{\mathcal{B}_\nu^c}W^*_{\mathcal{B}_\nu^c}  + \bar{W}_{\mathcal{B}_\nu^c}\bar{W}^*_{\mathcal{B}_\nu^c})u+\rho.
\end{equation}
If $u\notin \mbox{Ker}(W^*_{\mathcal{B}_\nu^c})$, then from \eqref{eq:150} we get
\begin{equation} \label{eq:151}
u^\intercal N u \ge  c\mu^2\nu^3\lambda_{min}(\mbox{Re}(W_{\mathcal{B}_\nu^c}W^*_{\mathcal{B}_\nu^c}))  + \rho. 
\end{equation}
Hence, combining \eqref{eq:501}, \eqref{eq:154} and \eqref{eq:151} we conclude that 
\begin{align*}
|\lambda-1|  \le \frac{1}{2}\frac{{\chi + 1} + (5\chi^2 - 2\chi + 1)^{\frac{1}{2}}}{ c\mu^2\nu^3\lambda_{min}(\mbox{Re}(W_{\mathcal{B}_\nu^c}W^*_{\mathcal{B}_\nu^c}))  + \rho} 
\end{align*}
If $u\in \mbox{Ker}(W^*_{\mathcal{B}_\nu^c})$, then from \eqref{eq:150} we have that $u^\intercal N u \ge \rho$, hence 
$$
|\lambda -1| \le \frac{1}{2}\frac{{\chi + 1} + (5\chi^2 - 2\chi + 1)^{\frac{1}{2}}}{\rho}.
$$
\end{proof}

Let us now draw some conclusions from Theorem \ref{thm:3}. In order for the eigenvalues of $N^{-1}\nabla^2 f_c^\mu(x)$ to be around one, it is required that the degree of freedom $\nu$ is chosen such that $\nu = \mathcal{O}(1/\mu)$
and $\mu$ is small.
For such $\nu$, the cardinality $\sigma$ of the set $\mathcal{B}_\nu$ must be small enough such that 
matrices $A$ and $W$ satisfy W-RIP with constant $\delta_{\sigma}< 1/2$; otherwise the assumptions of Theorem \ref{thm:3} will not be satisfied. 
This is possible if the pdNCG iterates are close to the optimal solution $x_{c,\mu}$ and $\mu$ is sufficiently small. In particular, for  sufficiently small $\mu$, 
from Remark \ref{rem:3} we have that $x_{c,\mu} \approx x_c$ and $\sigma \approx q$. 
According to Assumption \ref{assum:1} for the $q$-sparse $x_c$, W-RIP is satisfied for $\delta_{2q} < 1/2 \Longrightarrow \delta_q < 1/2$.
Hence, for points close to $x_{c,\mu}$ and small $\mu$ we expect that $\delta_\sigma < 1/2$.
Therefore, the result in Theorem \ref{thm:3} captures only the limiting behaviour of preconditioned $\nabla^2 f_c^\mu(x)$ as $x\to x_{c,\mu}$.
Moreover, according to Remark \ref{rem:2}, Theorem \ref{thm:3}
implies that at the limit the eigenvalues of $\tilde{N}^{-1}\hat{B}$ are also clustered around one.

We now comment on the second result of Theorem \ref{thm:3}, when the eigenvectors of $N^{-\frac{1}{2}}\nabla^2 f_c^\mu(x) N^{-\frac{1}{2}} $ belong in $\mbox{Ker}(W^*_{\mathcal{B}_\nu^c})$.
In this case, according to Theorem \ref{thm:3} the preconditioner removes the disadvantageous dependence of the spectrum of $\nabla^2 f_c^\mu(x)$ on the smoothing parameter $\mu$.
However, there is no guarantee that the eigenvalues of $N^{-1}\nabla^2 f_c^\mu(x)$ are clustered around one, regardless of the distance from the optimal solution $x_{c,\mu}$. 
Again, because of Remark \ref{rem:2} we expect that the spectrum 
of $\tilde{N}^{-1}\hat{B}$ at the limit will have a similar behaviour. 

The scenario of limiting behaviour of the preconditioner is pessimistic.
Let $\tilde{\sigma}$ be the minimum sparsity level such that matrices $A$ and $W$ are W-RIP with $\delta_{\tilde{\sigma}} < 1/2$. Then,
according to the uniform property of W-RIP (i.e. it holds for all at most $\tilde{\sigma}$-sparse vectors), 
the preconditioner will start to be effective even if the iterates $W^*x^k$ are approximately sparse with $\tilde{\sigma}$ dominant non-zero components.
Numerical evidence is provided in Figure \ref{figSpec} to confirm this claim.
In Figure \ref{figSpec} the spectra $\lambda(\hat{B})$ and $\lambda(\tilde{N}^{-1}\hat{B})$ are displayed for a sequence of systems which arise 
when an iTV problem is solved. For this iTV problem we set matrix $A$ to be a partial $2D$ DCT, $n=2^{10}$, $m=n/4$, $c=2.29e$-$2$
and $\rho=5.0e$-$1$. 
For the first experiment, which corresponds to Figures \ref{fig5_1_a} and \ref{fig5_1_b} the smoothing parameter has been set to $\mu=1.0e$-$3$.
For the second experiment, which corresponds to Figures \ref{fig5_1_c} and \ref{fig5_1_d} we set $\mu=1.0e$-$5$. 
Observe in Figures \ref{fig5_1_b} and \ref{fig5_1_d} that the eigenvalues of matrix $\tilde{N}^{-1}\hat{B}$ are \textit{nicely clustered around one}.
On the other hand in Figures \ref{fig5_1_a} and \ref{fig5_1_c} the eigenvalues of matrix $\hat{B}$
have large variations and there are many small eigenvalues close to zero.
Notice that the preconditioner was effective not only at optimality as it was predicted by theory, but through all iterations of pdNCG. This is because starting from the zero solution 
the iterates $W^*x^k$ were maintained approximately sparse $\forall k$. 

Additionally, in Figure \ref{figSpec2} we show the number of CG/PCG iterations and time required for the unpreconditioned 
and the preconditioned cases of the same experiment.
Observe in Figures \ref{fig5_1_e} ($\mu=1.0e$-$3$) and \ref{fig5_1_f} ($\mu=1.0e$-$5$) that 
the number of PCG iterations are much less than the number of CG iterations. 
Surprisingly, the number of CG iterations required for the experiment
with $\mu=1.0e$-$3$ were more than the number of iterations for CG for the experiment with $\mu=1.0e$-$5$. Although, matrix $\hat{B}$ has worse condition number in the latter case, see the values of the vertical axis in Figures \ref{fig5_1_a} and \ref{fig5_1_c}.
We believe that this is because of the slightly better clustering of the eigenvalues of matrix $\hat{B}$ for the experiments with $\mu=1.0e$-$5$; see Figures \ref{fig5_1_a} and \ref{fig5_1_c}.
Finally, PCG was faster than CG in terms of required time for convergence, see Figures \ref{fig5_1_g} and \ref{fig5_1_h}.
\begin{figure}
\centering
	\begin{subfigure}[b]{0.48\textwidth}
		\includegraphics[width=\textwidth]{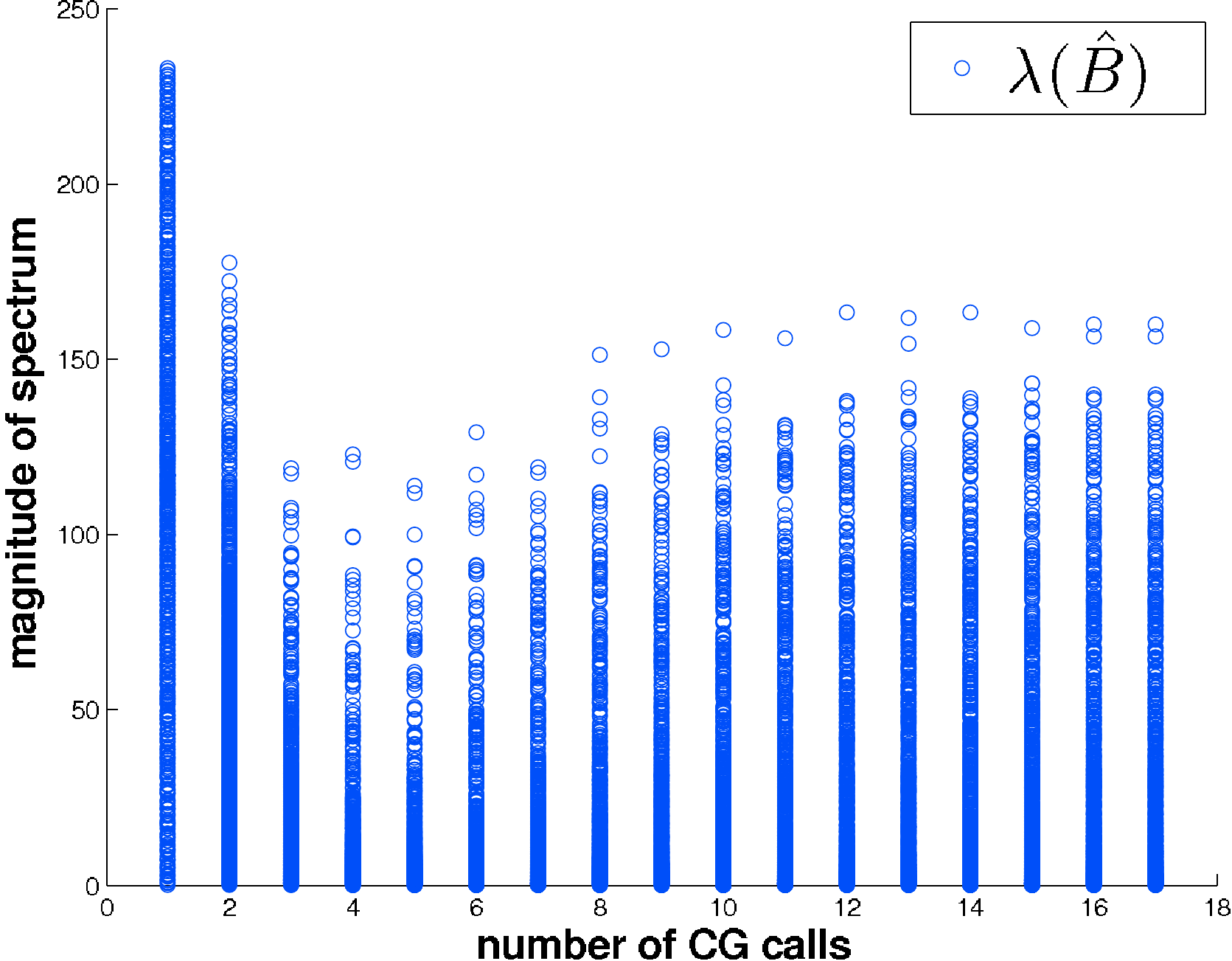}		
		\caption{Unpreconditioned}
		\label{fig5_1_a}%
         \end{subfigure}
	\begin{subfigure}[b]{0.48\textwidth}
		\includegraphics[width=\textwidth]{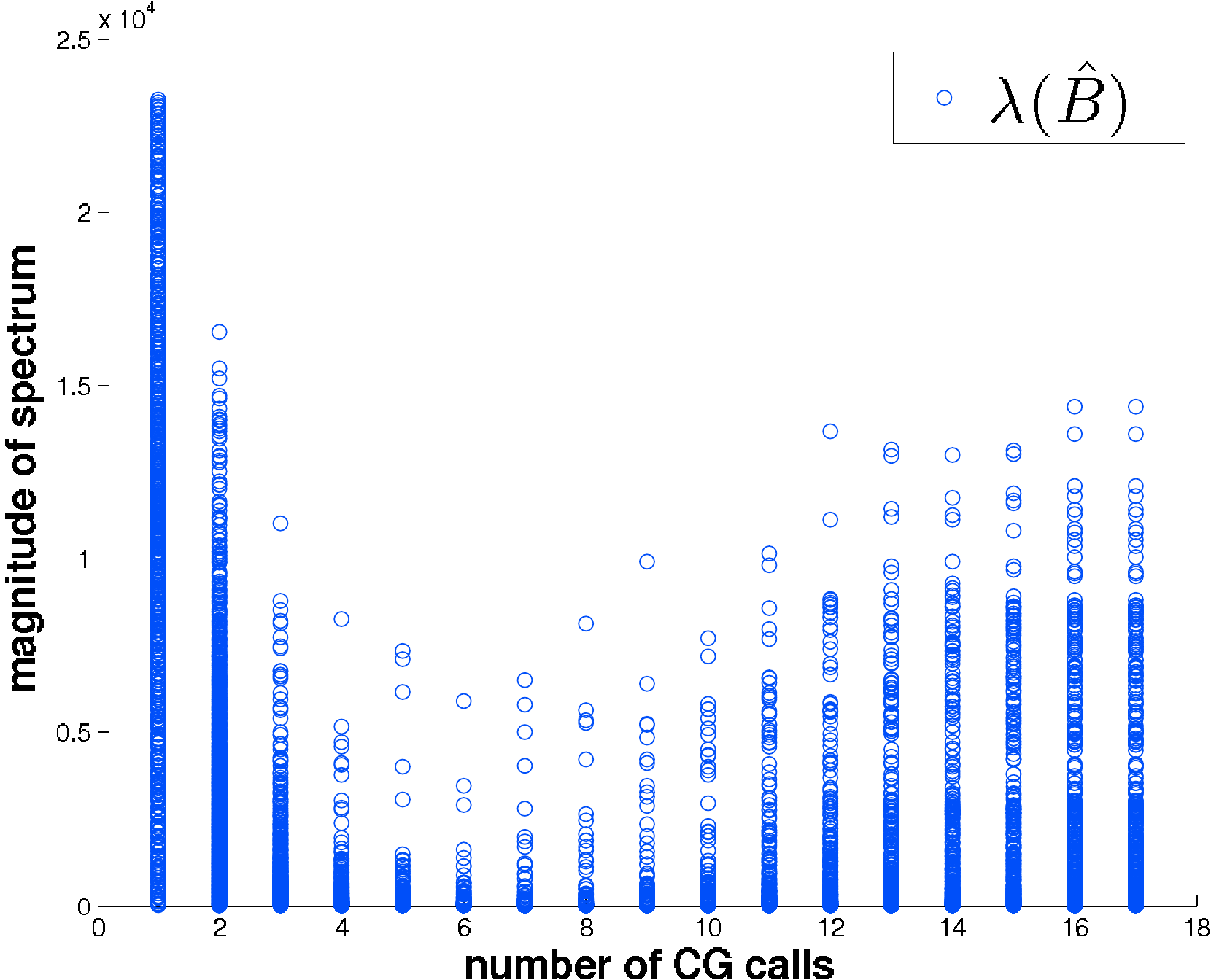}		
		\caption{Unpreconditioned}
		\label{fig5_1_c}%
         \end{subfigure}
	\begin{subfigure}[b]{0.48\textwidth}
		\includegraphics[width=\textwidth]{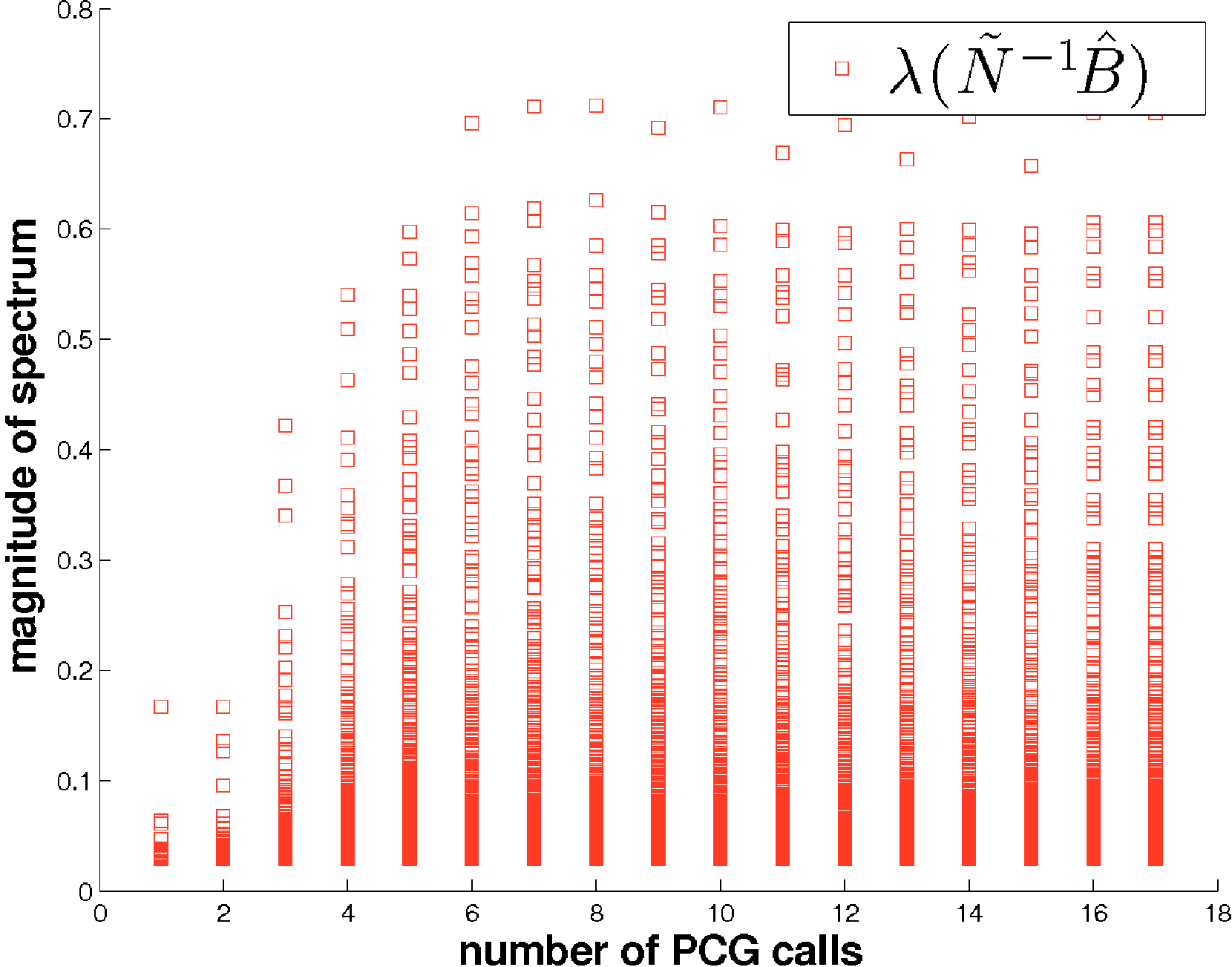}	
		\caption{Preconditioned}
		\label{fig5_1_b}%
         \end{subfigure}
         \quad
	\begin{subfigure}[b]{0.48\textwidth}
		\includegraphics[width=\textwidth]{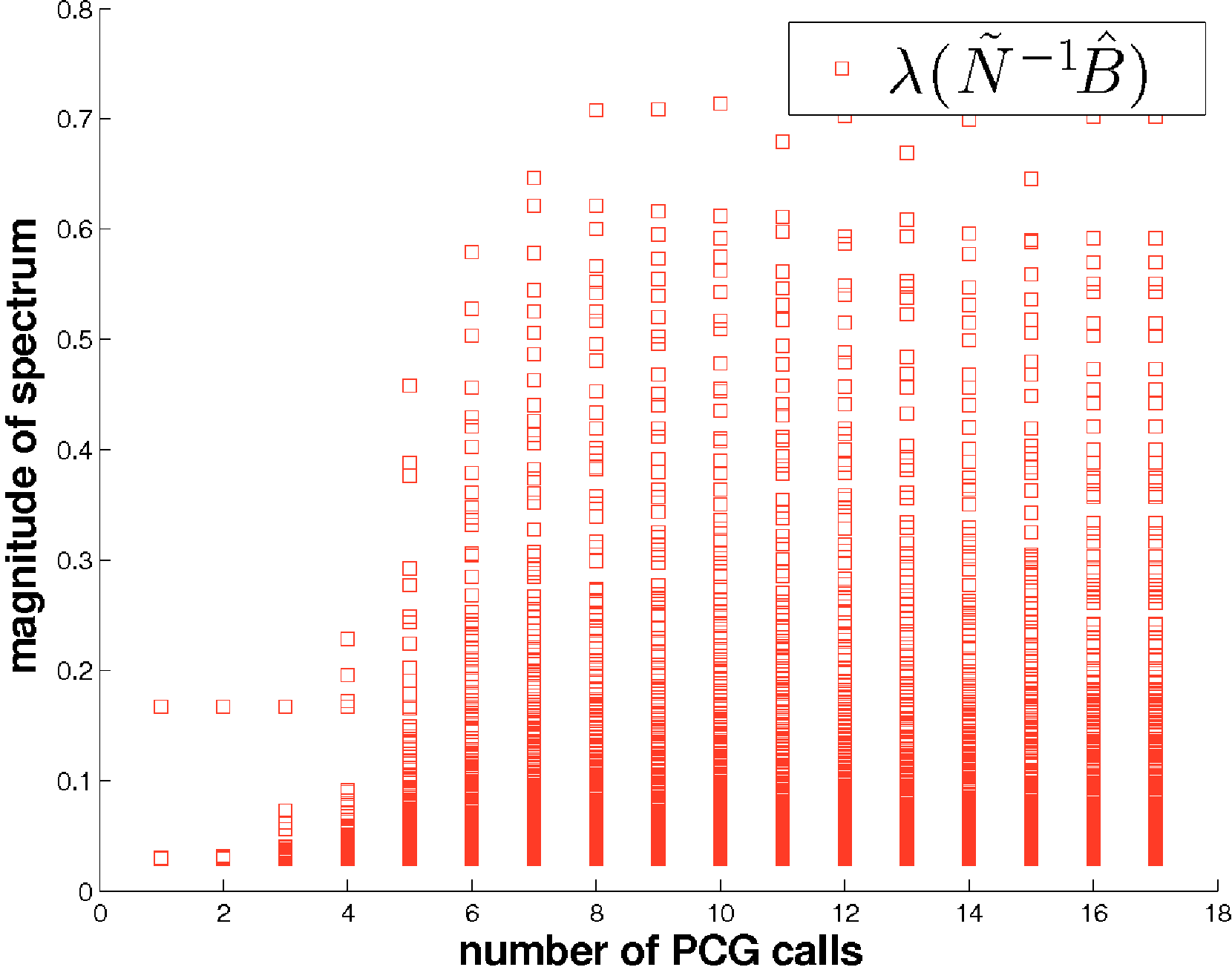}	
		\caption{Preconditioned}
		\label{fig5_1_d}%
         \end{subfigure}
	\caption{Spectra of $\lambda(\hat{B})$, $\lambda(\tilde{N}^{-1}\hat{B})$ when pdNCG is applied with smoothing parameter $\mu=1.0e$-$3$ (left column of sub-figures)
	and $\mu=1.0e$-$5$ (right column of sub-figures). 
	Matrix $A$ in $\hat{B}$ is a $2D$ DCT, $n=2^{10}$, $m=n/4$ and $c=2.29e$-$2$. Seventeen systems are solved in total for each experiment.
	}
	\label{figSpec}%
\end{figure}
  
\begin{figure}
\centering
	\begin{subfigure}[b]{0.48\textwidth}
		\includegraphics[width=\textwidth]{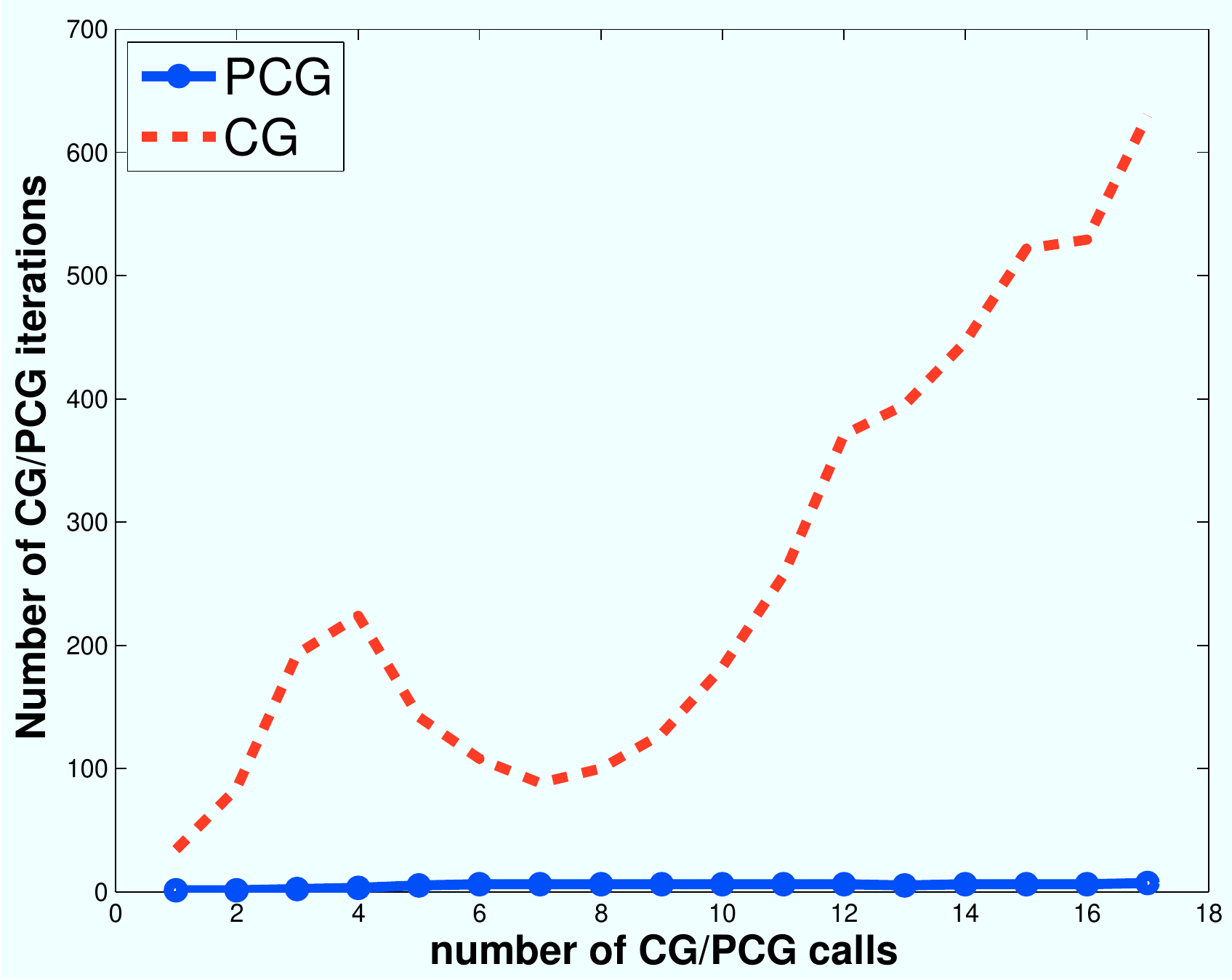}		
		\caption{CG/PCG iterations}
		\label{fig5_1_e}%
         \end{subfigure}
         \quad
	\begin{subfigure}[b]{0.48\textwidth}
		\includegraphics[width=\textwidth]{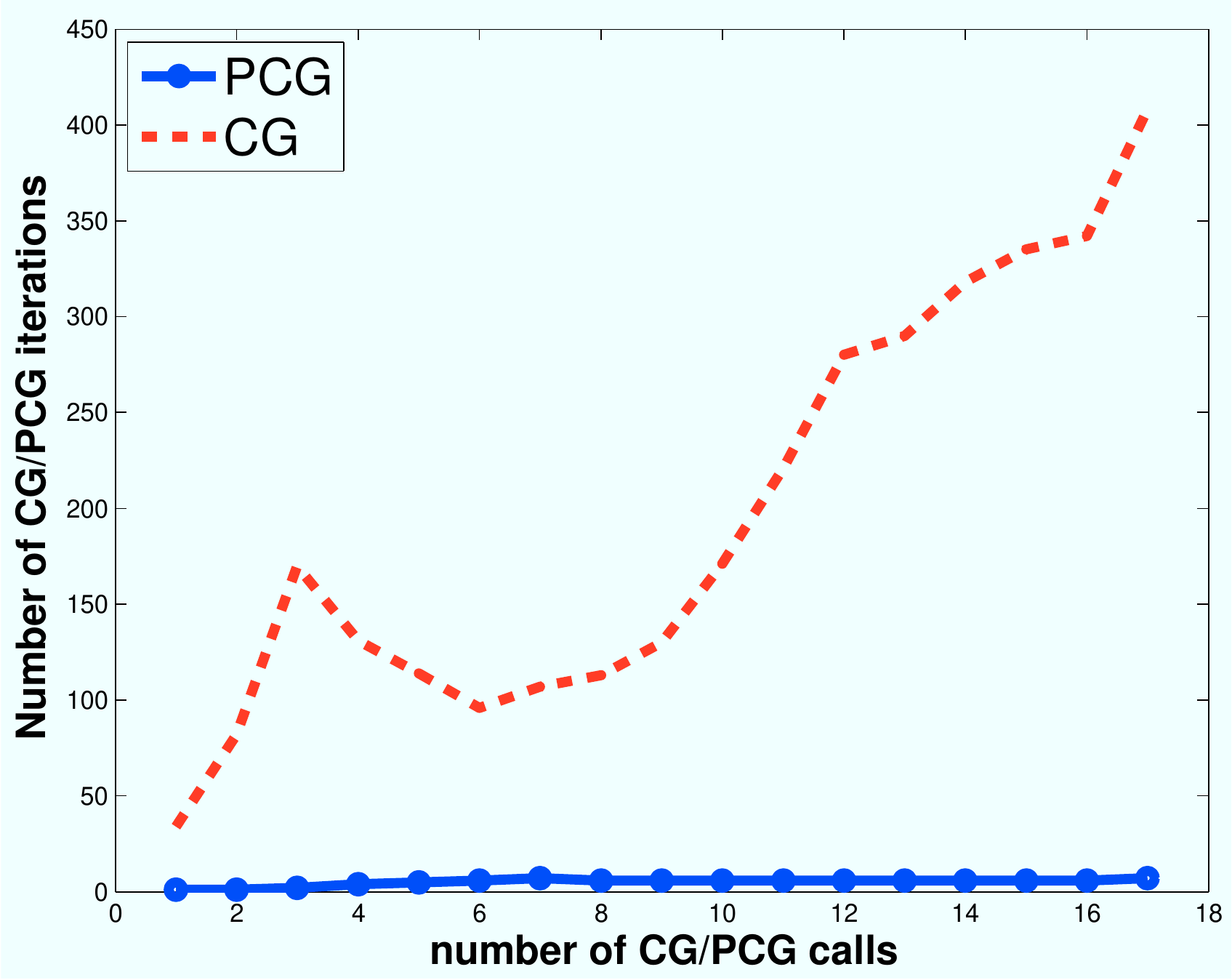}	
		\caption{CG/PCG iterations}
		\label{fig5_1_f}%
         \end{subfigure}
	\begin{subfigure}[b]{0.48\textwidth}
		\includegraphics[width=\textwidth]{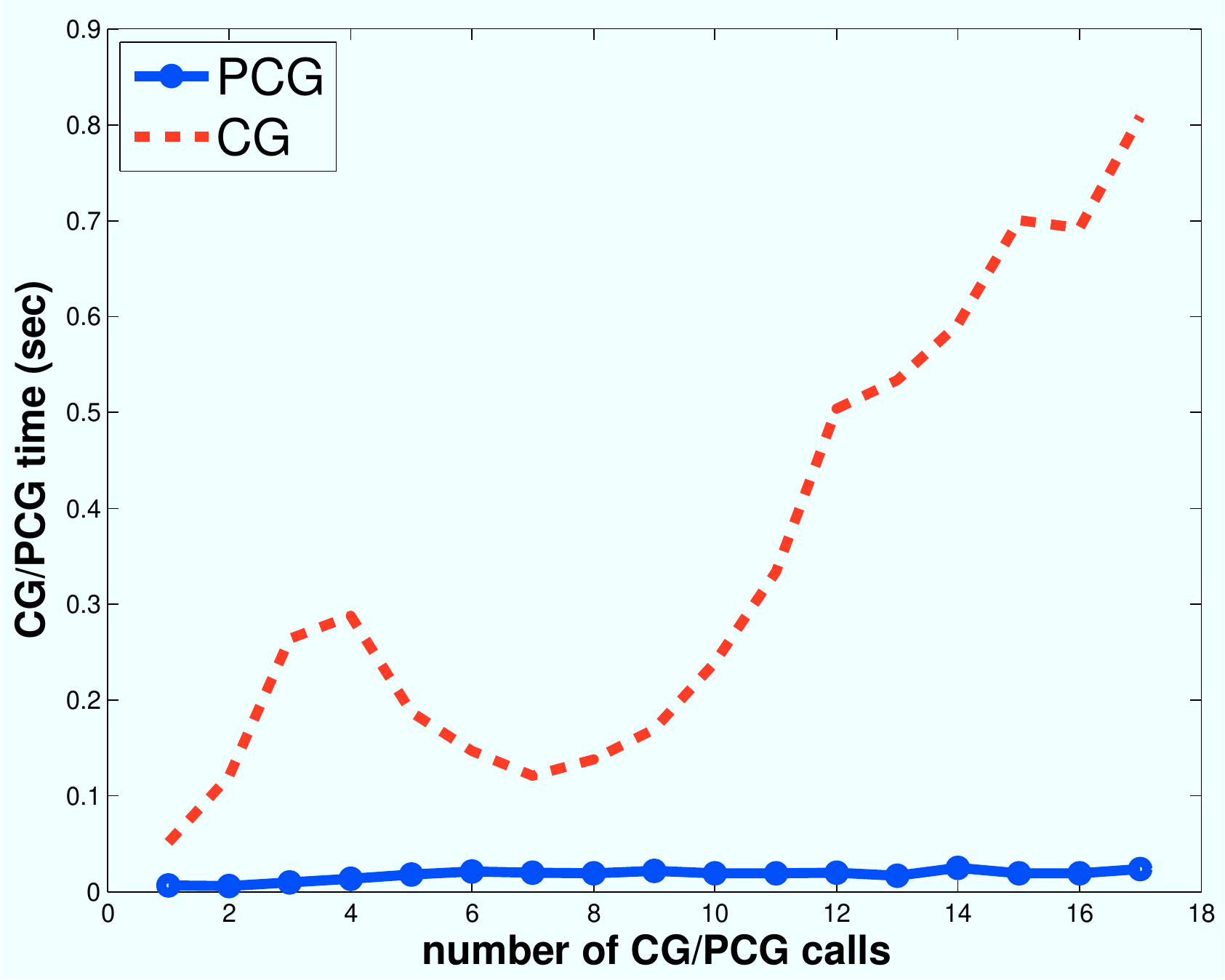}		
		\caption{CG/PCG time}
		\label{fig5_1_g}%
         \end{subfigure}
         \quad
	\begin{subfigure}[b]{0.48\textwidth}
		\includegraphics[width=\textwidth]{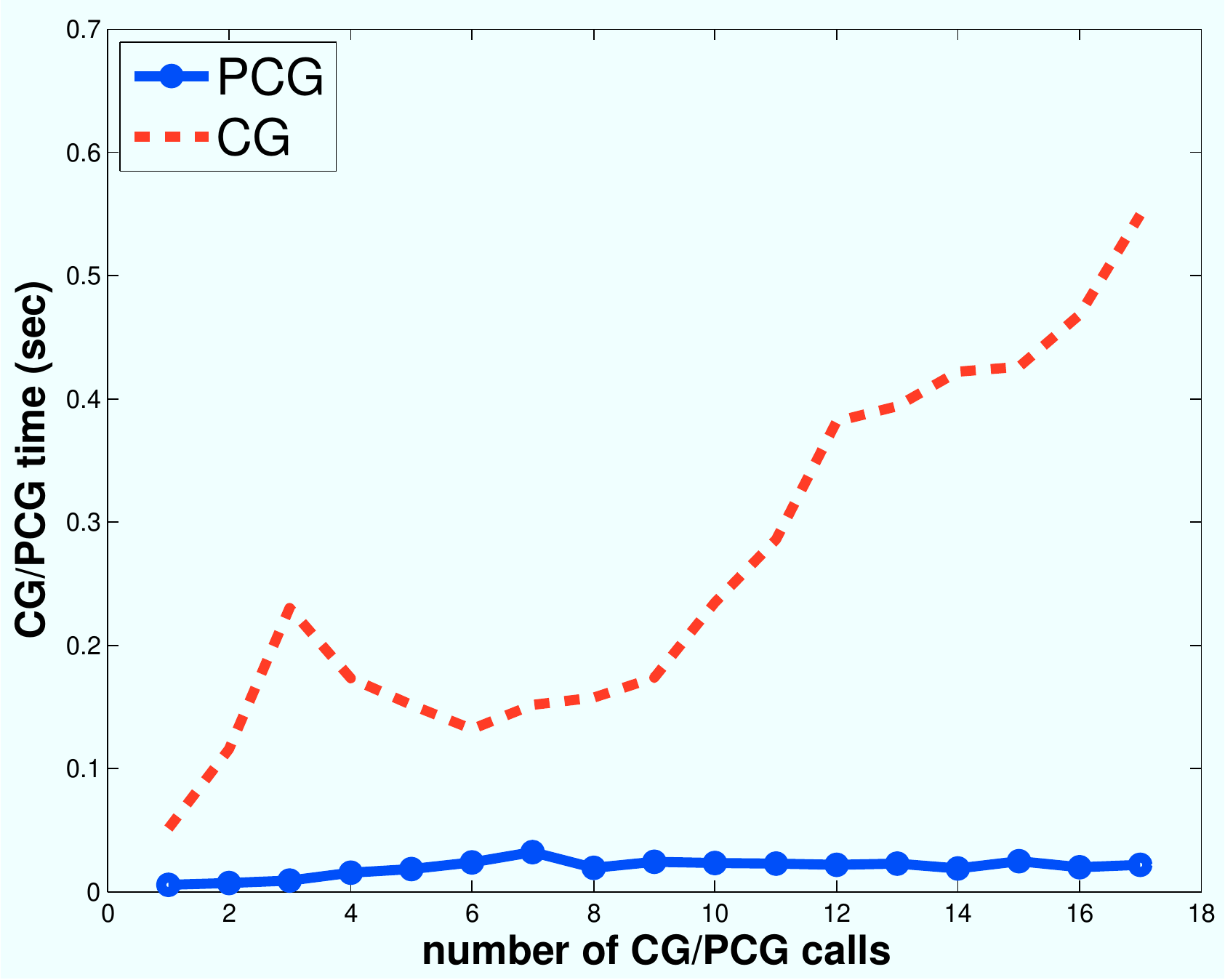}	
		\caption{CG/PCG time}
		\label{fig5_1_h}%
         \end{subfigure}
	\caption{Number of CG/PCG iterations and required time when pdNCG is applied with smoothing parameter $\mu=1.0e$-$3$ (left column of sub-figures)
	and $\mu=1.0e$-$5$ (right column of sub-figures). 
	Matrix $A$ in $\hat{B}$ is a $2D$ DCT, $n=2^{10}$, $m=n/4$ and $c=2.29e$-$2$. Seventeen systems are solved in total for each experiment.
	}
	\label{figSpec2}%
\end{figure}

\subsection{Solving systems with the preconditioner}\label{subsec:howsolvesystems}
In this subsection we discuss 
how we can solve systems with the proposed preconditioner $\tilde{N}$.

The simplest case is when $W$ is an orthogonal matrix. In this case it is readily verified that solving systems with matrix $\tilde{N}$ costs two matrix vector products with matrices $W$, $W^\intercal$
and the inversion of a diagonal matrix. Therefore, this operation is inexpensive. 
Especially, when
matrix $W$ is a DCT or a Wavelets transform for which matrix-vector products with $W$ and $W^\intercal$ can be calculated in $\mathcal{O}(n\log n)$ and $\mathcal{O}(n)$ time,
respectively.


Let us now consider iTV problems. Let $p_v$ be the vertical number of pixels of the image to be reconstructed and $p_h$ be the horizontal number of pixels.\ 
For simplicity we will assume that the image is square, hence, $p=p_v=p_h$.\ Additionally, we assume that the image is handled in a vectorized form, i.e., instead of an image of size $p\times p$ we have 
a vectorized image of size $p^2 \times 1$ where the columns of the image are stuck one after the other.\ In this case, for iTV the $W\in \mathbb{C}^{n \times n}$ matrix
in problem \eqref{prob1} is square with $n=p^2$, rank-deficient with $rank (W)=n-1$. Matrix $W$ corresponds to a discretization of the nabla operator and it measures local differences of pixels when applied on a vectorized
image.\
In particular,
$$
W = W_v + \sqrt{-1}W_h,
$$
where $W_v\in\mathbb{R}^{n\times n}$ and $W_h\in\mathbb{R}^{n\times n}$. Matrix $W_v$ measures vertical differences of pixels when applied on a vectorized image and it has the following non-zero components: 
$$
[W_v]_{p(j-1) + i, p(j-1)+i} = -1 \ \mbox{ and } \ [W_v]_{p(j-1) + i, p(j-1)+i+1} = 1
$$
$\forall j=1,2,\cdots,p$ and $\forall i=1,2,\cdots,p-1$.\ Matrix $W_h$ measures horizontal differences of pixels when applied on a vectorized image and it has the following non-zero pattern: 
 $$
[W_h]_{p(j-1) + i, p(j-1)+i} = -1 \ \mbox{ and } \ [W_h]_{p(j-1) + i, p(j-1)+i+p} = 1
$$
$\forall j=1,2,\cdots,p-1$ and $\forall i=1,2,\cdots,p$.\

Observe that matrix $\tilde{N}$ in this case is at most seven-diagonal and it has the following block tridiagonal form
\begin{equation}\label{blockprec}
\tilde{N}= 
\begin{bmatrix}   C_1             & K_1^\intercal   &    			  &                            &         		   &  				   \\
                       	  K_1		     & C_2                 & K_2^\intercal     & 			       &         		   &  				   \\
                           		     & 	K_2		      &    C_3                 & K_3^\intercal      & 		           &  				   \\
                      			     &   		      & K_3		          & \ddots		       &    \ddots      	  & 		  		   \\       
                      			     & 			      &  			  & 	 \ddots 	      &	C_{p-1}		           & K_{p-1}^\intercal 	  \\	
                      			     & 			      &  			  & 	 	 	      &K_{p-1}	           & C_p 			   \\			              
       \end{bmatrix},
\end{equation}
where $C_i\in\mathbb{R}^{p\times p}$ are tridiagonal matrices $\forall i=1,2,\cdots,p$ and $K_i\in\mathbb{R}^{p\times p}$ are upper bidiagonal matrices $\forall i=1,2,\dots,p-1$.
Solving systems with the symmetric positive definite block tridiagonal matrix $\tilde{N}$ can be done in $\mathcal{O}(p^4)$ time by calculating its Cholesky decomposition without re-ordering.
More precisely, the Cholesky factor $\tilde{L}$ of $\tilde{N}= \tilde{L} \tilde{L}^\intercal$ is of the form
\begin{equation}\label{eq:precfactor}
\tilde{L}= 
\begin{bmatrix}   L_1             & 			       &    			  &                            &         		   &  				   \\
                       	  U_1	     & L_2                 & 			  & 			       &         		   &  				   \\
                           		     & 	U_2		      &    L_3                 & 			      & 		           &  				   \\
                      			     &   		      & U_3		          & \ddots		       &    		      	  & 		  		   \\       
                      			     & 			      &  			  & 	 \ddots 	      &	L_{p-1}		           & 		 	  \\	
                      			     & 			      &  			  & 	 	 	      &U_{p-1}	           & L_p 			   \\			              
       \end{bmatrix},
\end{equation}
where $L_i\in\mathbb{R}^{p\times p}$ $\forall i=1,2,\cdots,p$ and $U_i\in\mathbb{R}^{p \times p}$ $\forall i=1,2,\cdots,p-1$. The factor $\tilde{L}$ can be calculated by 
using $\tilde{N} = \tilde{L}\tilde{L}^\intercal$, \eqref{blockprec} and \eqref{eq:precfactor} to get
\begin{align}
\label{eq:chol_1}
L_1 L_1^\intercal &=C_1 \\ \label{eq:chol_2}
U_iL_i^\intercal &= K_i \quad \forall i=1,2,\cdots,p-1 \\\label{eq:chol_4}
U_{i-1}U_{i-1}^\intercal + L_iL_i^\intercal & = C_i \quad \forall i=2,3,\cdots,p.
\end{align}
Notice that $C_1$ in \eqref{blockprec} is symmetric positive definite because $\tilde{N}$ is a symmetric positive definite matrix. 
Since $C_1$ is symmetric positive definite and tridiagonal we can obtain the lower bidiagonal matrix $L_1$ in \eqref{eq:chol_1} by calculating the Cholesky decomposition of matrix $C_1$. 
Calculation of $L_1$ can be done in $\mathcal{O}(p)$ time.
Matrix $U_1$ is upper diagonal and
can be calculated in $\mathcal{O}(p^3)$ time by solving $U_1 L_1^\intercal = K_1$.
The next step is the calculation of $L_2$ using \eqref{eq:chol_4}, which is the Cholesky factor of $C_2 - U_{1}U_{1}^\intercal$. 
The term $C_2 - U_{1}U_{1}^\intercal$ can be calculated in $\mathcal{O}(p^3)$.
Notice that 
$C_2 - U_{1}U_{1}^\intercal= C_2 - K_1 C_1^{-1}K_1^\intercal$ is the Schur complement of 
a two by two block matrix 
\begin{equation*}
S = 
\begin{bmatrix}   
C_1 & K_1^\intercal \\
K_1 & C_2	              
\end{bmatrix}.
\end{equation*}
Matrix $S$ is symmetric positive definite because matrix $\tilde{N}$ is positive definite, this can be readily seen from \eqref{blockprec}.
Hence, the Schur complement $C_2 - K_1 C_1^{-1}K_1^\intercal$ of $S$ is a symmetric positive definite matrix.
By repeating this process $p$ times using \eqref{eq:chol_2} and \eqref{eq:chol_4} we obtain the matrices $L_i$ and $U_i$.
Since each step requires $\mathcal{O}(p^3)$ time, the total calculation of $\tilde{L}$ requires $\mathcal{O}(p^4)$.

We note that this $\mathcal{O}(p^4)$ operation is expensive for the problems of our interest. 
Therefore, instead of naively calculating the factor $\tilde{L}$ we employ 
AMD (Approximate Minimum Degree) ordering \cite{amdalgo} by invoking MATLAB's backslash operator. 
AMD results in significant reduction in the FLOPS (FLoating-point Operations Per Second) rate and the filling in $\tilde{L}$. 
In Table \ref{tablescaleprec} we present how the running time of MATLAB's backslash operator scales as $p$
increases from to $2^3$ to $2^{11}$. The results are averaged over $20$ trials. 
Observe that the running time scales nearly $\mathcal{O}(p^2)$, which is a significant 
improvement compared to $\mathcal{O}(p^4)$. The matrices $\tilde{N}$ in this experiment were constructed using pdNCG.
Additionally, for this experiment we run MATLAB only in one thread in order to eliminate the effects of multithread implementations.


\begin{table}
\center
\caption{Scaling of running time of Cholesky factorization with ADM ordering, $p$ denotes the number of pixels of a $p \times p$ image.}
\begin{tabular}{|c|c|c|c|c|c|c|c|c|}
\hline
    $p$                   & $16$           & $32$ &  $64$ & $128$ & $256$ & $512$ & $1024$ & $2048$ \\ \hline \hline
    CPU (sec) &  $3.6e$-$4$   & $1.4e$-$3$     & $6.1e$-$3$  & $3.0e$-$2$   & $1.6e$-$1$   & $9.1e$-$1$ & $5.2$ & $30.6$\\
\hline
\end{tabular}
\label{tablescaleprec}
\end{table}

Unfortunately, solving systems with the proposed preconditioner is not always an inexpensive procedure. In particular, 
solving systems with matrix $\tilde{N}$ when $W$ has orthonormal rows is a non-trivial operation.
An example is radar and sonar systems \cite{apps4}, where $W\in\mathbb{R}^{n\times l}$ is a Gabor frame with $n\le l$.
In this case, $\tilde N$ does not have a structure which can be exploited in order to solve systems with it inexpensively. 
In similar cases in the literature, i.e., denoising of images \cite{approxprec}, attempts have been made to solve systems with the preconditioner using an iterative method. 
In our case $\tilde{N}$ is a symmetric positive definite matrix. Therefore, $\mbox{CG}_{\tilde{N}}$ can be used, where $\mbox{CG}_{\tilde{N}}$ denotes conjugate gradients method, but we add a subscript 
$\tilde{N}$ in order to distinguish from the unpreconditioned CG notation for the system in \eqref{eq108}.

Our personal numerical
experience regarding this strategy suggests that it is difficult to control the time required by $\mbox{CG}_{\tilde{N}}$ to solve systems with $\tilde{N}$ approximately, such that
the overall time required by PCG is reduced. Although, the number of PCG iterations is decreased, even with a small number of $\mbox{CG}_{\tilde{N}}$ iterations.
In Figure \ref{figapproxprec} we present an example where using the preconditioner $\tilde{N}$ in an iterative fashion does not decrease the time required by pdNCG.
The tested problem is a radar tone reconstruction instance \cite{convexTemplates}, where $W\in\mathbb{R}^{n\times l}$ is a Gabor frame with $n=2^{13}$, $l=228864$,  
$\mu=1.0e$-$5$, $c=2.41e$-$5$ and $\rho=5.0e$-$1$. Matrix $A$ is a block-diagonal with $\pm 1$ 
entries in the blocks and $m=648$ rows.
We make four experiments. For the first three experiments we vary the number of $\mbox{CG}_{\tilde{N}}$ iterations for solving systems with $\tilde{N}$.  
The number of $\mbox{CG}_{\tilde{N}}$ iterations is set to $5$, $10$ and $20$, respectively for each experiment. The last experiment is using unpreconditioned CG. Observe in Figure \ref{fig5_2_a}
that PCG requires significantly fewer iterations than unpreconditioned CG for all settings of $\mbox{CG}_{\tilde{N}}$.
However,
notice in Figure \ref{fig5_2_b} the time required by PCG is larger than the time required by CG.


\begin{figure}
\centering
	\begin{subfigure}[b]{0.48\textwidth}
		\includegraphics[width=\textwidth]{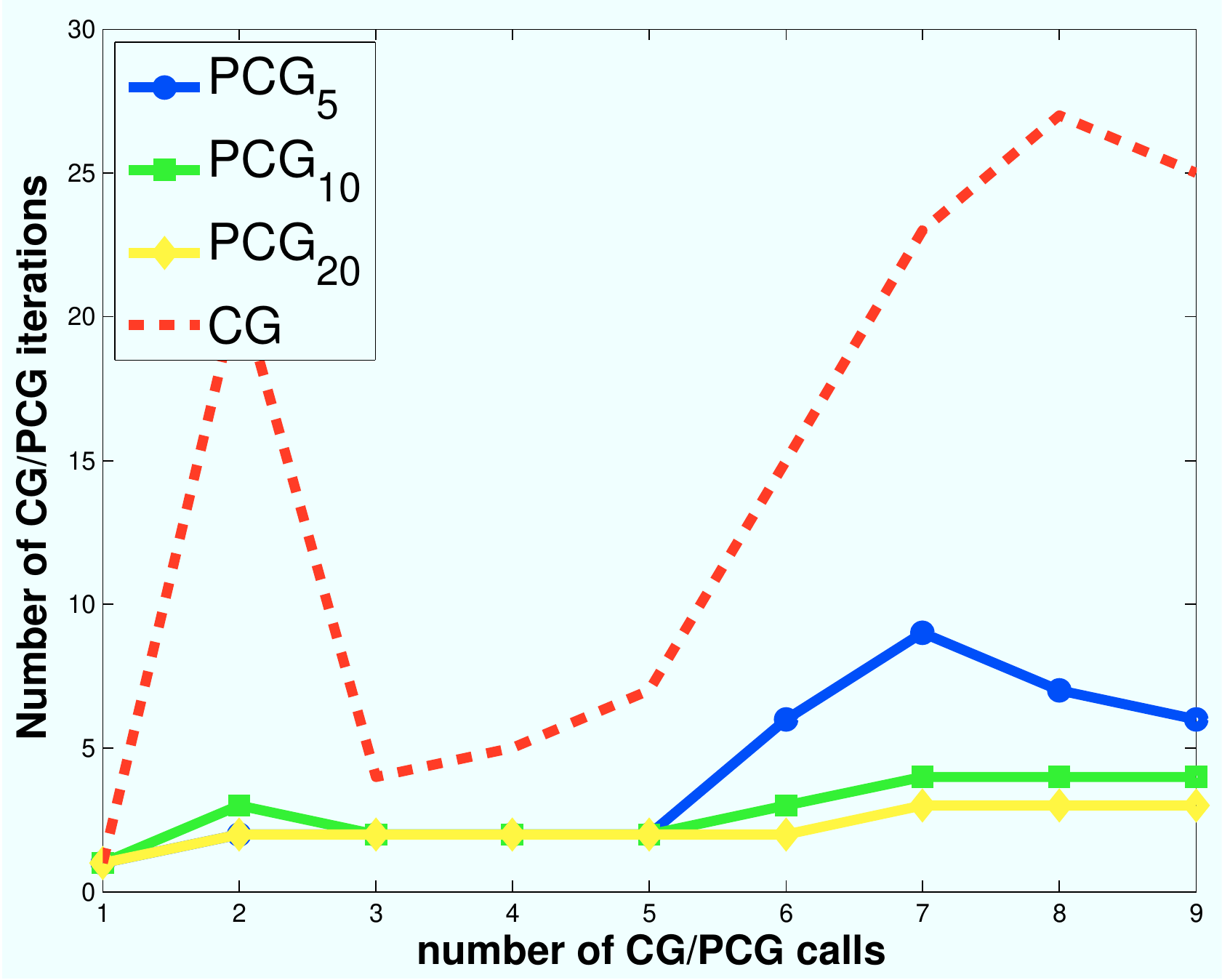}		
		\caption{CG/PCG iterations}
		\label{fig5_2_a}%
         \end{subfigure}
         \quad
	\begin{subfigure}[b]{0.48\textwidth}
		\includegraphics[width=\textwidth]{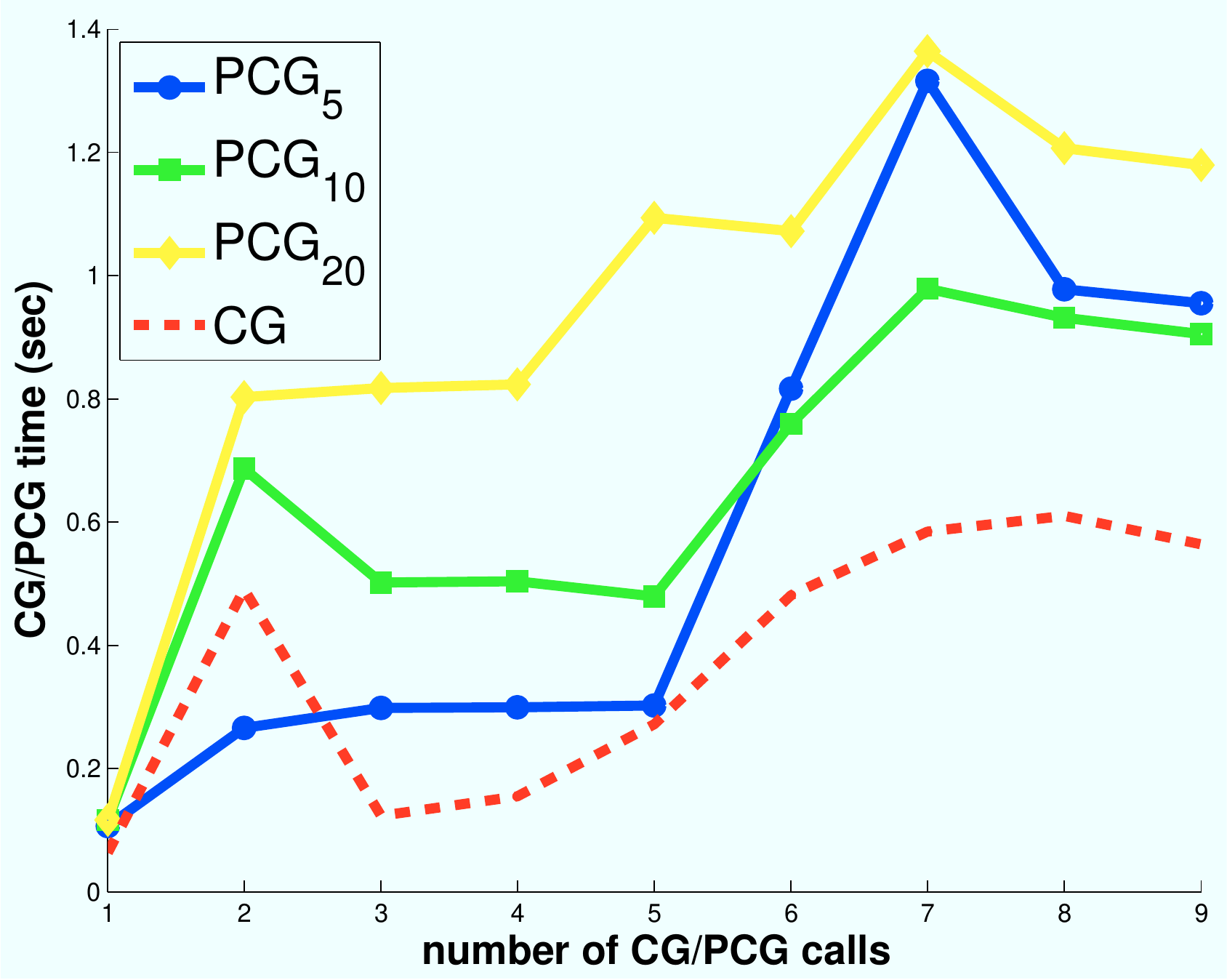}	
		\caption{CG/PCG time}
		\label{fig5_2_b}%
         \end{subfigure}
         \caption{Number of CG/PCG iterations and required time when $\mu=1.0e$-$5$ and systems with the preconditioner $\tilde{N}$ are solved 
         approximately using conjugate gradients. $\mbox{PCG}_{5}$, $\mbox{PCG}_{10}$ and $\mbox{PCG}_{20}$ correspond to PCG,
         where systems are solved with $\tilde{N}$ approximately using conjugate gradients which is terminated after $5$, $10$ and $20$ iterations, respectively.
         }
	\label{figapproxprec}%
\end{figure}

\section{Continuation}\label{sec:cont}
In the previous section we have shown that by using preconditioning, the spectral properties of systems which arise can be improved. However, for initial stages of pdNCG
a similar result can be achieved without the cost of having to apply preconditioning. In particular, at initial stages 
the spectrum of $\hat{B}$ can be controlled to some extent through inexpensive continuation. Whilst preconditioning is enabled only at later stages of the process. 
Briefly by continuation it is meant that a sequence of ``easier" subproblems is solved, instead of solving directly problem \eqref{prob2}. 
The reader is referred to Chapter $11$ in \cite{mybib:NocedalWright} for a survey on continuation methods in optimization.

 In this paper we use a similar continuation framework to \cite{IEEEhowto:Nesta,ctnewtonold,ctpdnewton,haleconti}. In particular, 
 a sequence of sub-problems \eqref{prob2} is solved, where each of them is parameterized by $c$ and $\mu$ simultaneously.
 Let $\tilde{c}$ and $\tilde{\mu}$ be the final parameters for which problem \eqref{prob2} must be solved. Then
 the number of continuation iterations $\vartheta$ is set to be the maximum order of magnitude between $1/\tilde{c}$ and $1/\tilde{\mu}$.
 For instance, if $\tilde{c}=1.0e$-$2$ and $\tilde{\mu}=1.0e$-$5$ then $\vartheta := \max(2,5) = 5$.
 If $\vartheta \ge 2$, then the initial parameters $c^0$ and $\mu^0$ are both always set to $1.0e$-$1$ and the intervals $[c^0,\tilde{c}]$ and $[\mu^0,\tilde{\mu}]$
 are divided in $\vartheta$ equal subintervals in logarithmic scale. For all experiments that we have reported in this paper we have found that 
 this setting leads to a generally acceptable improvement over pdNCG without continuation.
 The pseudo-code of the proposed continuation framework is shown in Figure \ref{fig:1}.
 
 \begin{figure}
\begin{algorithmic}[1]
\vspace{0.1cm}
\STATE \textbf{Outer loop:} For $j=0,1,2,\ldots ,\vartheta$, produce $(c^j,\mu^j)_{j=0}^\vartheta$.
\STATE \hspace{0.5cm}\textbf{Inner loop:} Approximately solve the subproblem
\begin{equation*}
 \mbox{minimize} \  f_{c^j}^{\mu^j}(x)
\end{equation*}
\hspace{0.5cm} using pdNCG and by initializing it with the  solution\\ 
\hspace{0.5cm} of the previous subproblem.
\end{algorithmic}
\caption{Continuation framework}
\label{fig:1}
\end{figure}

Figure \ref{fig3} shows the performance of pdNCG for four cases, no continuation and no preconditioning, no continuation with preconditioning, continuation with preconditioning through the whole process 
and continuation with preconditioning only at later stages. The vertical axis of Figure \ref{fig3} shows the relative error $\|x^k-x_{\tilde{c},\tilde{\mu}}\|_2/\|x_{\tilde{c},\tilde{\mu}}\|_2$.
The optimal $x_{\tilde{c},\tilde{\mu}}$ is obtained by using pdNCG with parameter tuning set to recover a highly accurate solution. The horizontal axis shows the CPU time.
The problem is an iTV problem where matrix $A$ is a partial $2D$ DCT, $n=2^{16}$, $m=n/4$, $c=5.39e$-$2$
and $\rho=5.0e$-$1$. The final smoothing parameter $\tilde{\mu}$ is set to $1.0e$-$5$. For the experiment that preconditioning is used only at later stages of continuation;
preconditioning is enabled when $\mu^j \le 1.0e$-$4$, where $j$ is the counter for continuation iterations. 
All experiments are terminated when the relative error $\|x^k-x_{\tilde{c},\tilde{\mu}}\|_2/\|x_{\tilde{c},\tilde{\mu}}\|_2 \le 1.0e$-$1$. Solving approximately the problem is an acceptable
practice since the problem is very noisy (i.e. signal-to-noise-ratio is $10$ decibel) and there is not much improvement of the reconstructed image if more accurate solutions are requested.
Finally, all other parameters of pdNCG were set to the same values for all four experiments.
Observe in Figure \ref{fig3} that continuation with preconditioning only at late stages was the best approach for this problem. 

\begin{figure}
\centering
	\begin{subfigure}[b]{0.7\textwidth}
		\includegraphics[width=\textwidth]{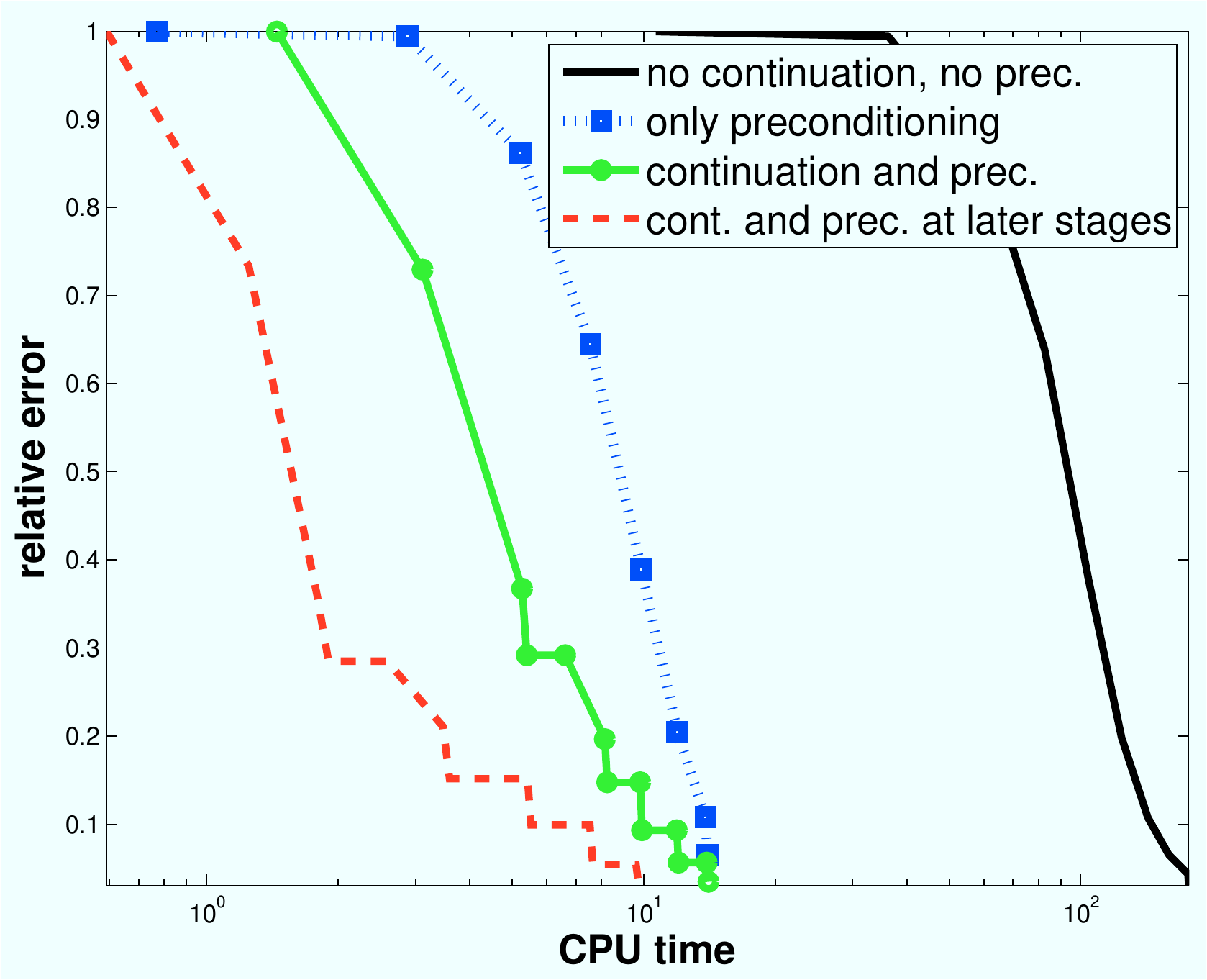}		
         \end{subfigure}
         \caption{Performance of pdNCG for four different settings, i) no continuation and no preconditioning, ii) no continuation with preconditioning, iii) continuation with preconditioning through all iterations and iv) continuation 
         with preconditioning only at later stages. The vertical axis presents the relative error $\|x^k-x_{\tilde{c},\tilde{\mu}}\|_2/\|x_{\tilde{c},\tilde{\mu}}\|_2$, where $x_{\tilde{c},\tilde{\mu}}$ is 
         the optimal solution for the parameter setting $\tilde{c}$, $\tilde{\mu}$ in problem \eqref{prob2}. The horizontal axis is in log-scale.}
	\label{fig3}%
\end{figure}

\section{Numerical experiments}\label{secNumExp}
In this section we demonstrate the efficiency of pdNCG against state-of-the-art methods for CS. 
We briefly discuss existing methods, describe the setting of the experiments and finally present numerical
results. All experiments that are demonstrated in this paper can be reproduced by downloading the software from 
\url{http://www.maths.ed.ac.uk/ERGO/pdNCG/}.

\subsection{Existing algorithms}\label{sec:algos}
We compare pdNCG with three state-of-the-art first-order methods, TFOCS \cite{convexTemplates}, TVAL3 \cite{tval3}
and TwIST \cite{twist}.

\begin{itemize}
\item[-] TFOCS (Templates for First-Order Conic Solvers) is a MATLAB software for the solution of signal reconstruction problems.
TFOCS solves the dual problem of
\begin{equation}\label{prob4}
         \mbox{minimize} \ c\|W^*x\|_1 + \frac{\mu_{T_1}}{2}\|x - x^0\|^2_2 + \frac{1}{2}\|Ax- b \|_2^2, 
\end{equation}
where $\mu_{T_1}$ is a positive constant. TFOCS
also solves the dual problem of
\begin{equation}
  \begin{array}{lll}\label{probqweqwe4}
        & \displaystyle\min_{x\in\mathbb{R}^{n}} & \|W^*x\|_1 + \frac{\mu_{T_2}}{2}\|x - x^0\|^2_2 \\
        &\mbox{subject to:}& \|Ax- b \|_2 \le \epsilon, \\
  \end{array}
\end{equation}
where $\epsilon>0$. 
Although problems \eqref{prob4} and \eqref{probqweqwe4} are 
non-smooth, the regularization terms ${\mu_{T_1}}/{2}\|x - x^0\|^2_2$ and ${\mu_{T_2}}/{2}\|x - x^0\|^2_2$ yield smooth convex dual problems, which can be solved by
standard first-order methods. 
In particular, the smooth dual problems are solved using the Auslender and Teboulle's 
accelerated first-order method \cite{austemb}. In our experiments we present results for TFOCS for both problems \eqref{prob4} and \eqref{probqweqwe4}. We denote by TFOCS\_unc the version that solves the unconstrained 
problem \eqref{prob4} and by TFOCS\_con the version that solves the constrained problem \eqref{probqweqwe4}. TFOCS can be downloaded from \cite{convexTemplates}. 
\item[-] TVAL3 (Total-Variation minimization by Augmented Lagrangian and ALternating direction ALgorithms) is a MATLAB software for the solution of signal reconstruction 
problems regularized with the total-variation semi-norm. TVAL3 reformulates problem \eqref{prob1} to the equivalent problem
\begin{equation}\label{realprob1}
\mbox{minimize} \ \sum_{i=1}^l \|\Omega_i^\intercal x \|_2 + \frac{1}{2c}\|Ax-b\|_2^2,
\end{equation}
where $\Omega_i=[ReW_i, ImW_i]\in\mathbb{R}^{n\times 2}$. Then it solves the augmented Lagrangian reformulation of problem \eqref{realprob1}, which is
\begin{equation}\label{auglagrealprob1}
\mbox{minimize} \ \sum_{i=1}^l (\|u_i\|_2 + \frac{\beta}{2}\|\Omega_i^\intercal x - u_i\|_2^2 - v_i^\intercal (\Omega_i^\intercal x - u_i)) + \frac{1}{2c}\|Ax-b\|_2^2,
\end{equation}
where $u_i,v_i\in\mathbb{R}^2$ and $\beta,c$ are positive constants. The augmented Lagrangian in \eqref{auglagrealprob1} is minimized for variables $x\in\mathbb{R}^n$
and $u_i$ $i=1,2,\cdots,l$. The parameters $v_i$ $i=1,2,\cdots,l$ are handled by the method.

\item[-] TwIST (Two-step Iterative Soft Thresholding): is also a MATLAB software for signal/image processing problems. TwIST solves problem \eqref{prob1}.
TwIST is a nonlinear two-step iterative version of IST, which according to its authors is more effective on ill-conditioned and ill-posed problems.

\end{itemize}

Another solver is
NestA \cite{IEEEhowto:Nesta} (by the same authors as TFOCS) which can also solve \eqref{prob1} but it is applicable only in the case that $(A A^\intercal )^{-1}$ is available.
Additionally, TFOCS is the sucessor of NestA, since both apply the similar techniques but TFOCS is a newer and allows of more control options.
Another method is the Primal-Dual Hybrid Gradient (PDHG) of \cite{pdhg}.
PDHG has been reported to be very efficient for imaging applications such as denoising and deblurring, for which
matrix $A$ is the identity or a square and full-rank matrix which is inexpensively diagonalizable. Unfortunately, this is not the case for the CS problems that we are interested in. 
However, for all earlier mentioned methods the matrix inversion can be replaced with a solution of a linear system at every iteration of the methods or a one-time cost of a factorization. 
To the best of our knowledge, there are no available implementations with such modifications for these methods.

There exists also a generic proximal algorithm for total-variation \cite{gentotalvar} and the Generalized Iterative Soft Thresholding (GISTA) in \cite{gista} for which we do not have generic implementations
for CS problems. 


\subsection{Equivalent problems}\label{subset:equiv}
Solvers pdNCG, TFOCS\_unc, TVAL3 and TwIST solve the penalized least squares problem \eqref{prob1}, while TFOCS\_con solves the constrained least squares problem \eqref{probqweqwe4}. 

In our experiments we put significant effort in 
calibrating the parameters $c$ and $\epsilon$ for the penalized and constrained least-squares problems, respectively, 
such that all methods solve similar problems. 
First, we set $\epsilon = \|b - \tilde b\|_2$ in \eqref{probqweqwe4}, where $\tilde b$ is the noiseless sampled signal. Hence problem \eqref{probqweqwe4}
is parameterized with the optimal $\epsilon$. Then we find an approximation of the optimal $c$. By optimal $c$ we mean the value of $c$ for which problems \eqref{prob1} and \eqref{probqweqwe4} 
are equivalent if $\epsilon = \|b - \tilde b\|_2$ and $\mu_{T_2}=0$. Let $\omega$ denote the optimal Lagrange multiplier of \eqref{probqweqwe4}.
If $\epsilon = \|b - \tilde b\|_2$ and $\mu_{T_2}=0$, then it is easy to show that 
for $c:= 2/\omega$ problems \eqref{prob1} and \eqref{probqweqwe4} are equivalent. 

The exact optimal Lagrange multiplier $\omega$ is not known a-priori.
However it can be calculated by solving to high accuracy the dual problem of \eqref{probqweqwe4} with TFOCS by setting $\mu_{T_2}\approx 0$ in \eqref{probqweqwe4}. 
Then we set $c:= 2/\omega$.
If $\tilde{b}$ is not available, then $\epsilon$ is set such that
a visually pleasant solution is obtained.

\subsection{Parameter tuning and hardware}\label{subsec:param}
The parameter $\mu$ of pdNCG is set to $\mu=1.00e$-$5$, which for the problems of our interest resulted 
in solutions with the similar or better accuracy than the compared methods.
The parameter $\eta$ in \eqref{bd59} is set to $1.0e$-$1$, the maximum number of backtracking line-search iterations
is fixed to $10$. Moreover, the backtracking line-search parameters $\tau_1$ and $\tau_2$ in step $5$ of pdNCG (Fig. \ref{fig:2}) are set to $9.0e$-$1$ and $1.0e$-$3$,
respectively.
The constant $\rho$ of the preconditioner in \eqref{bd10} is set 
to $5.0e$-$1$.

For TVAL3 parameter $\beta$ is set to $\beta=2^8$ based on suggestions of its authors in \cite{tval3} and our personal experience. 
Moreover, continuation was enabled in order to enhance the performance of the method.
Any other parameters that were not discussed are set to their default values.

We tune TwIST based on comments/suggestions of its authors and personal experience. In particular, parameter $\lambda$ is set to $\lambda=0.04$
and the maximum number of iterations for the iTV denoising procedure is set to $10$.

The version $1.3.1$ of TFOCS has been used. The termination criterion of TFOCS is by default the relative step-length. The tolerance for this criterion is set to the default 
value, except in cases that certain suggestions are made in TFOCS software package or the corresponding paper \cite{convexTemplates}. 
The default Auslender and Teboulle's single-projection method
is used as a solver for TFOCS. Moreover, as suggested by the authors of TFOCS, appropriate scaling is performed on matrices $A$ and $W$,
such that they have approximately the same Euclidean norms. All other parameters are set to their default values, except in cases
that specific suggestions are made by the authors. Generally, regarding tuning of TFOCS, substantial effort has been made to 
guarantee that problems are not over-solved.

All solvers are MATLAB implementations and all experiments are performed on a MacBook Air running OS X $10.10.1$ with $2$ GHz ($3$ GHz turbo boost) Intel Core Duo i7 processor
using MATLAB R2012a. The cores were working with frequency $2.7$ - $3$ GHz during the experiments and we did not observe any CPU throttling.

\subsection{Termination criteria}
For images we measure the quality of the reconstructed solutions by using the Peak-Signal-to-Noise-Ratio (PSNR) function 
$$
\mbox{PSNR} = 10 \log_{10} \left(\frac{\mbox{peakval}^2}{\mbox{MSE}}\right),
$$
where peakval is the range of the image datatype, in this case the range is one since we work with black and white images, and 
MSE is the mean squared error between the solution and the original noiseless image. For other types of signals we measure the quality
of the reconstructed solutions by measuring their Signal-to-Noise-Ratio (SNR).

We terminate pdNCG, TVAL3 and TwIST when the PSNR (for images) or the SNR (for other types of signals)
of their solution is equal or larger than the PSNR or SNR of the solution
obtained by TFOCS\_unc. This way we make sure that all methods which solve the penalized least-squares problem terminate when a solution of the same
quality as the one of TFOCS\_unc is obtained. As we mentioned in Subsection \ref{subsec:param} when we use TFOCS\_unc we do not over-solve the problem, otherwise, 
we would favour pdNCG, which is a second-order method.
Since pdNCG, TVAL3, TwIST and TFOCS\_unc solve the same problem with the same penalty parameter $c$ then we 
can make a fair comparison of their performance. The solution obtained by TFOCS\_con might differ slightly in terms of PSNR or SNR compared to the ones obtained by pdNCG, TVAL3, TwIST and 
TFOCS\_unc. However, this is because we 
set parameter $c$ to be approximately close to the optimal value that makes the penalized and the constrained problems equivalent.

\subsection{Problems sets}\label{subsec:problemsets}
We compare the solvers pdNCG, TFOCS, TVAL3 and TwIST on image reconstruction problems which are modelled using iTV.
We separate the images to be reconstructed into two sets, which are shown in Figures \ref{problemset1} and \ref{problemset2}. 
Figure \ref{problemset1} includes  some standard images from
the image processing community. There are seven images in total, the house and the peppers, which have $256\times 256$ pixels and Lena, the fingerprint, the boat and Barbara,
which have $512 \times 512$ pixels. Finally, the image Shepp-Logan has variable size depending on the experiment. 
Figure \ref{problemset2} includes images which have been sampled
using a single-pixel camera \cite{singlepixel}. Briefly a single-pixel camera samples random linear projections of pixels
of an image, instead of directly sampling pixels. The problem set can be downloaded from \url{http://dsp.rice.edu/cscamera}.
In this set there are in total five sampled images, the dice, the ball, the mug, the letter R and the logo. Each image has $64\times 64$ pixels. 

Moreover, we present the performance of pdNCG and TFOCS on the recovery of radio-frequency radar tones. This problem has been first demonstrated 
in Subsection $6.5$ of \cite{convexTemplates}. We describe again the setting of the experiment. 
The signal to be reconstructed consists of two radio-frequency radar tones which overlap in time.  The amplitude of the tones differs
by $60$ dB. The carrier frequencies and phases are chosen uniformly at random. Moreover, noise is added such that the larger tone has SNR
$60$ dB and the smaller tone has SNR $2.1e$-$2$ dB.
The signal is sampled at $2^{15}$ points, which corresponds to 
Nyquist sampling rate for bandwidth $2.5$ GHz and time period approximately $6.5e$$\mplus$$3$ ns.
The reconstruction is modelled as a CS problem where matrix $A\in\mathbb{R}^{m\times n}$ 
is block-diagonal with $\pm 1$ for entries, $n=2^{15}$ and $m=2616$, i.e. subsampling ratio $m/n \approx 7.9e$-$3$. Moreover,
$W\in\mathbb{R}^{n\times l}$ is a Gabor frame with $l=915456$. 

\begin{figure}%
\centering
	\begin{subfigure}[b]{0.23\textwidth}
		\includegraphics[width=\textwidth]{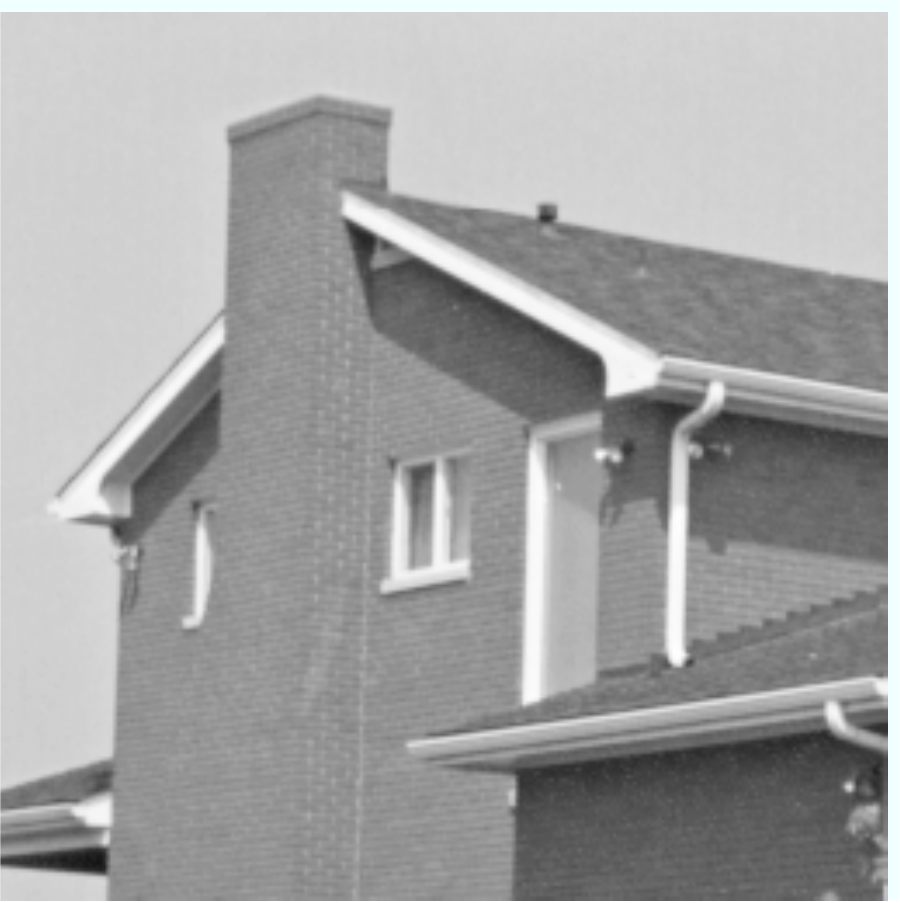}	
		\caption{House: $256^2$}
		\label{House}%
         \end{subfigure}
         \quad
	\begin{subfigure}[b]{0.23\textwidth}
		\includegraphics[width=\textwidth]{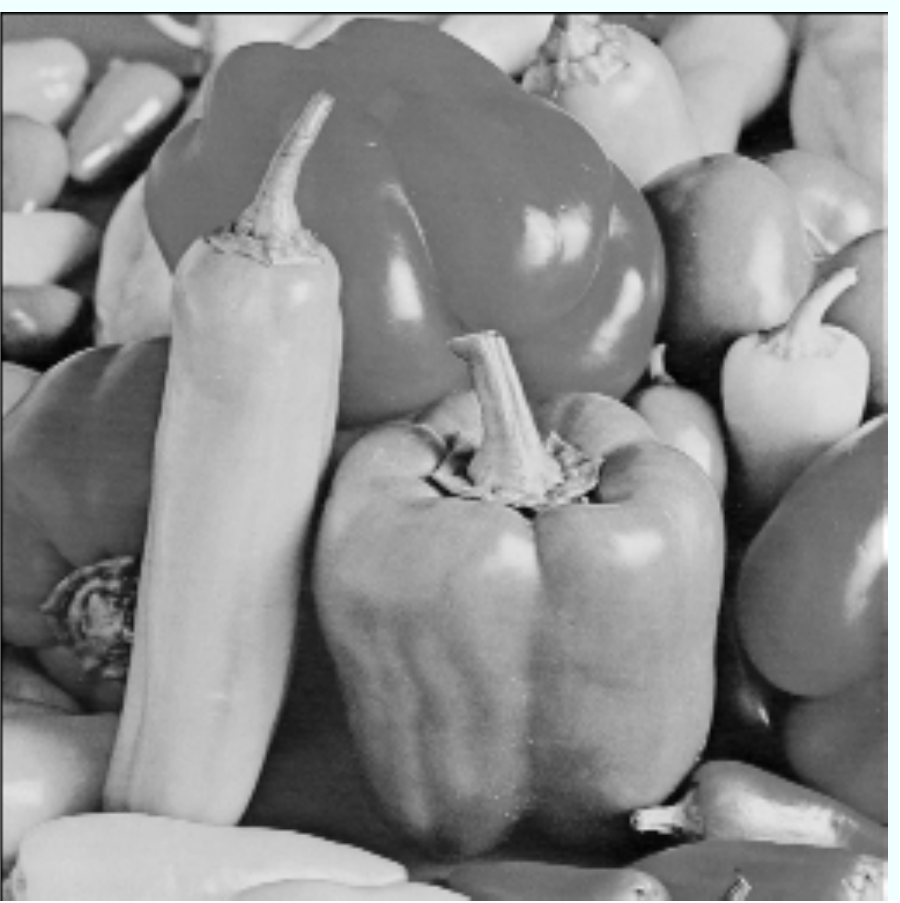}	
		\caption{Peppers: $256^2$}
		\label{Peppers}%
         \end{subfigure}
         \quad
	\begin{subfigure}[b]{0.23\textwidth}
		\includegraphics[width=\textwidth]{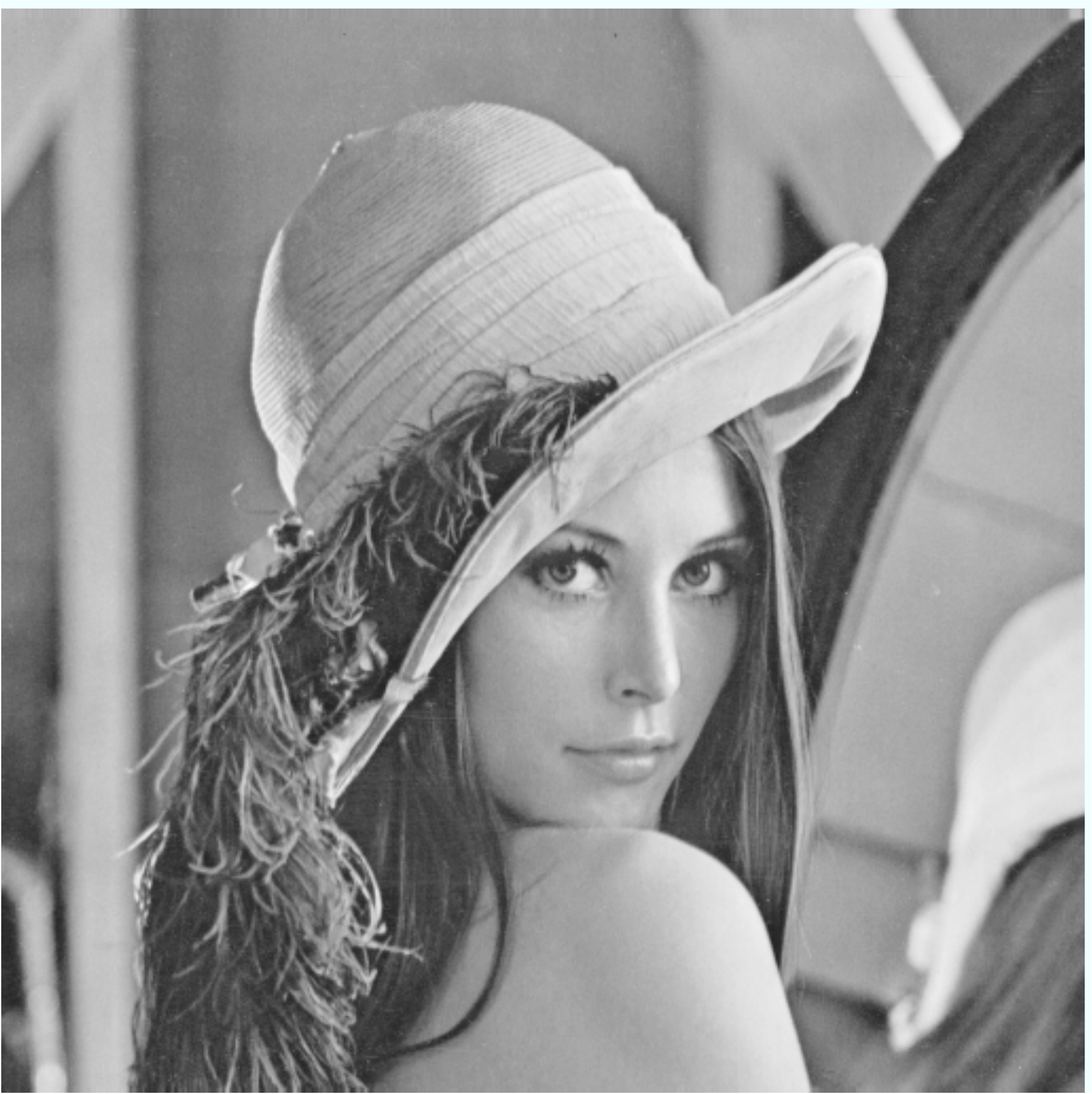}	
		\caption{Lena: $512^2$}
		\label{Lena}%
         \end{subfigure}
         \quad
	\begin{subfigure}[b]{0.23\textwidth}
		\includegraphics[width=\textwidth]{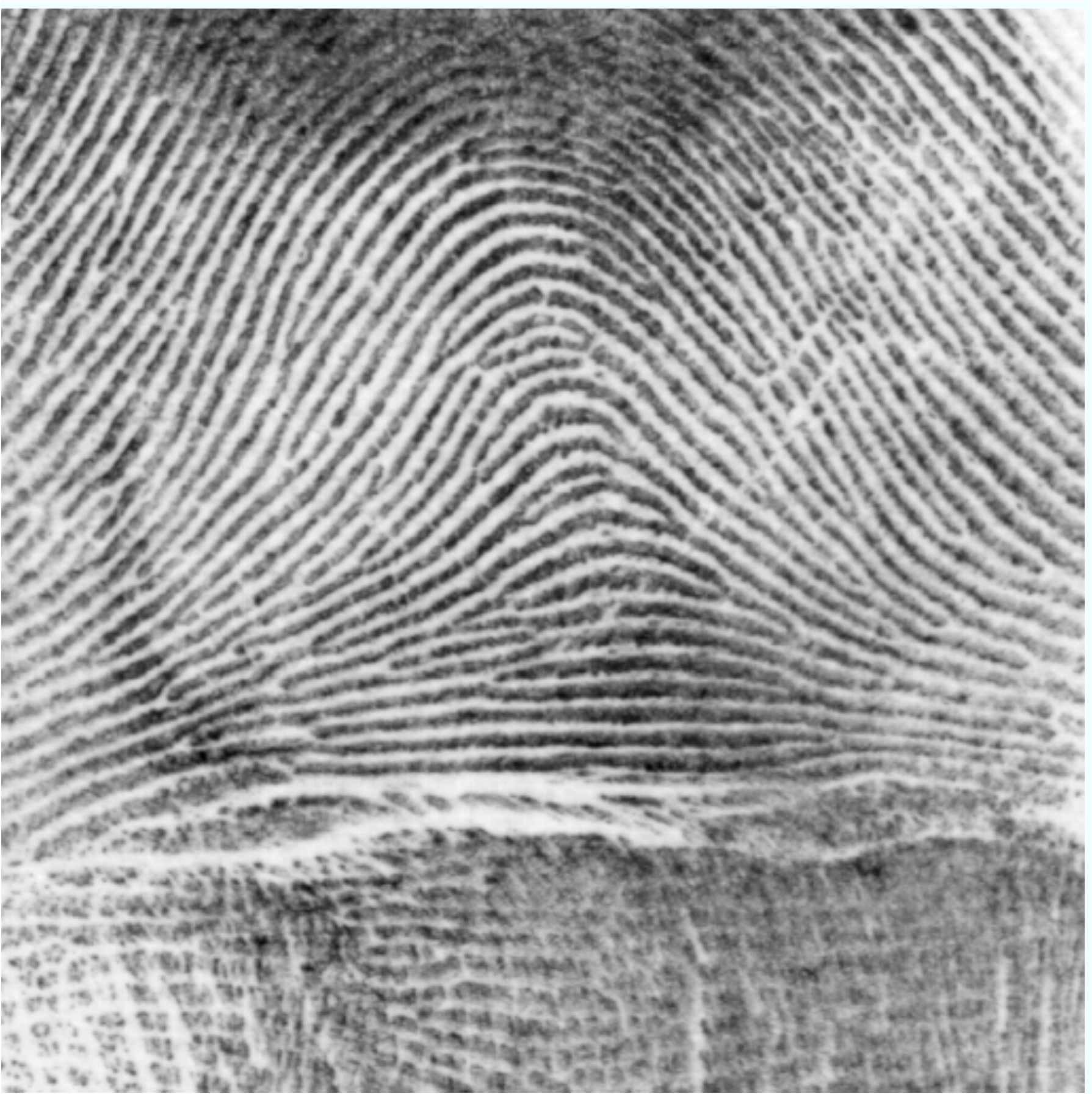}	
		\caption{Fingerprint: $512^2$}
		\label{fingerprint}%
         \end{subfigure}
         \quad
	\begin{subfigure}[b]{0.23\textwidth}
		\includegraphics[width=\textwidth]{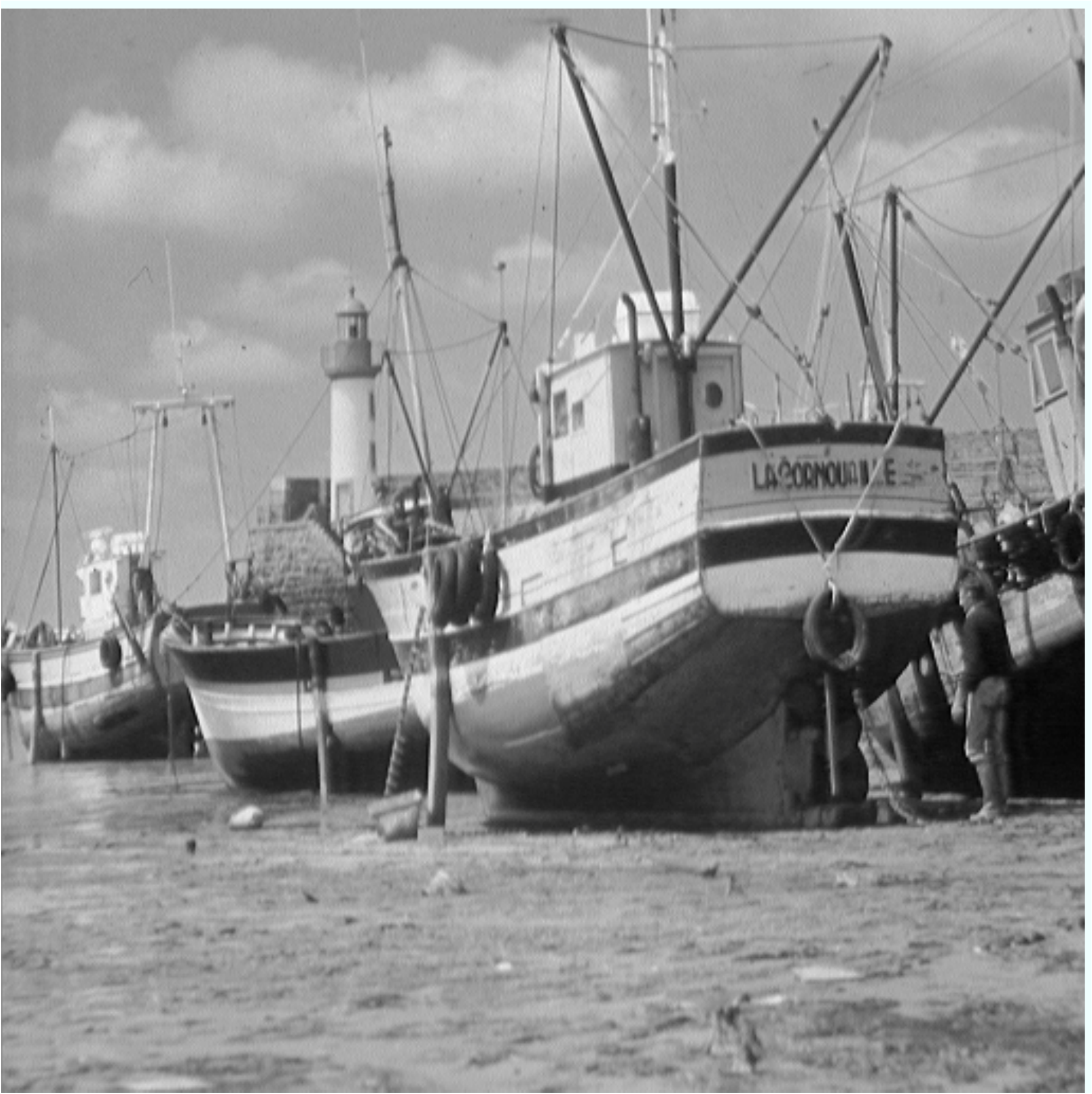}	
		\caption{Boat: $512^2$}
		\label{boat}%
         \end{subfigure}
         \quad
	\begin{subfigure}[b]{0.23\textwidth}
		\includegraphics[width=\textwidth]{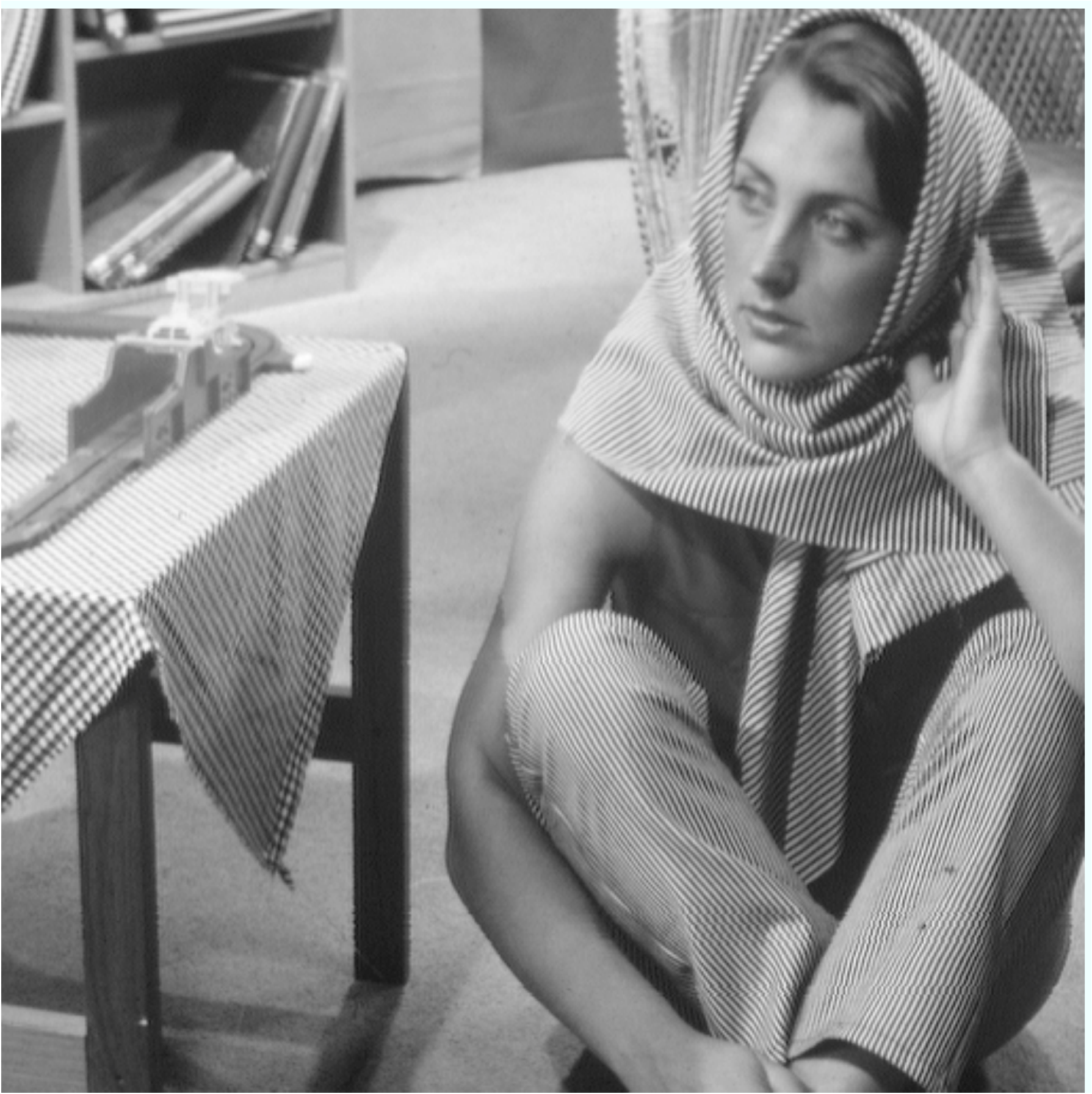}	
		\caption{Barbara: $512^2$}
		\label{barbara}%
         \end{subfigure}
         \quad
	\begin{subfigure}[b]{0.23\textwidth}
		\includegraphics[width=\textwidth]{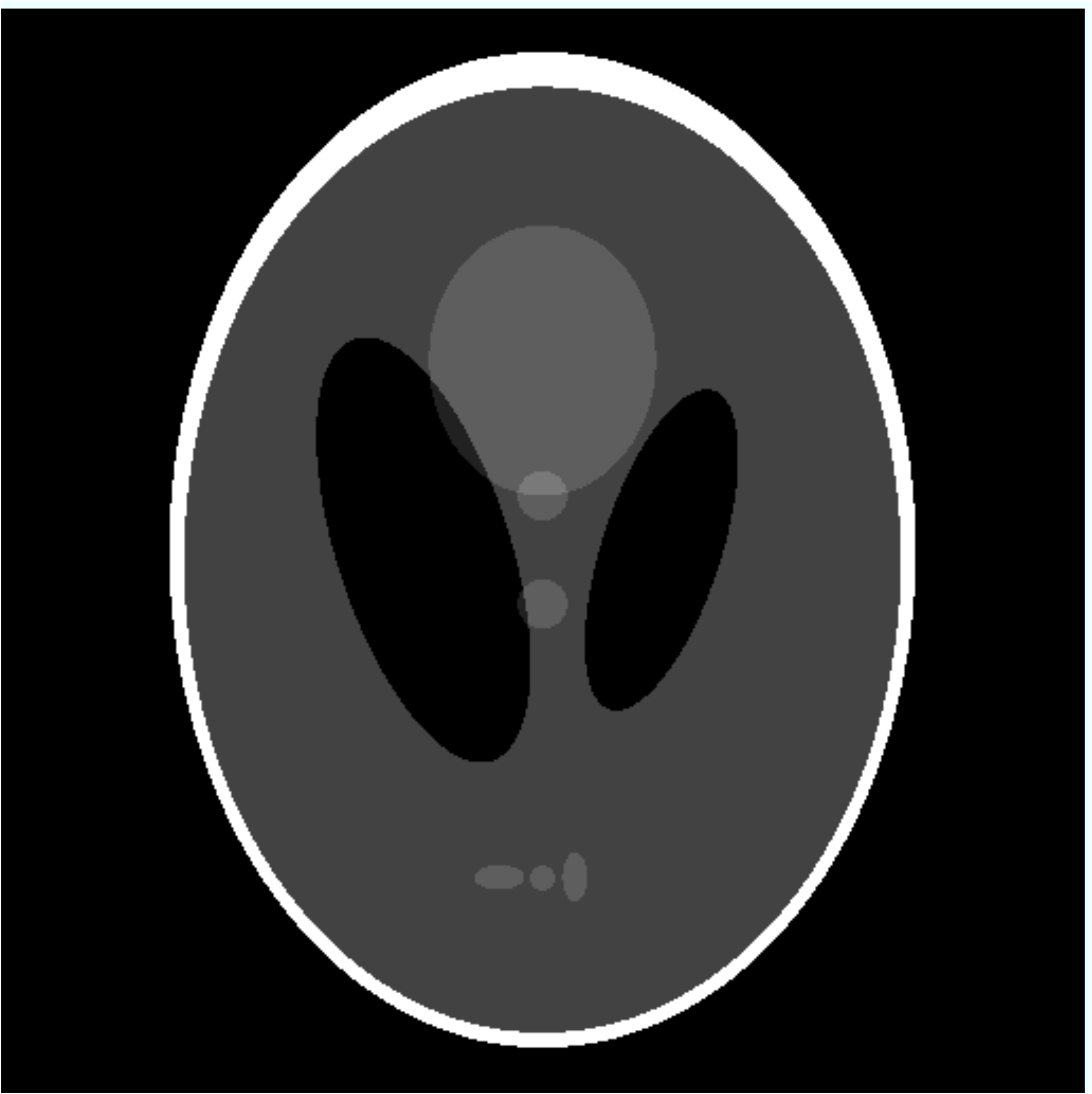}	
		\caption{Shepp-Logan: variable size}
		\label{shepplogan}%
         \end{subfigure}         
	\caption{Benchmark images, the number of pixels for each image is given in the sub-captions . For Figure \ref{shepplogan} the size varies depending on the experiment}
	\label{problemset1}%
\end{figure}

\begin{figure}%
\centering
	\begin{subfigure}[b]{0.23\textwidth}
		\includegraphics[width=\textwidth]{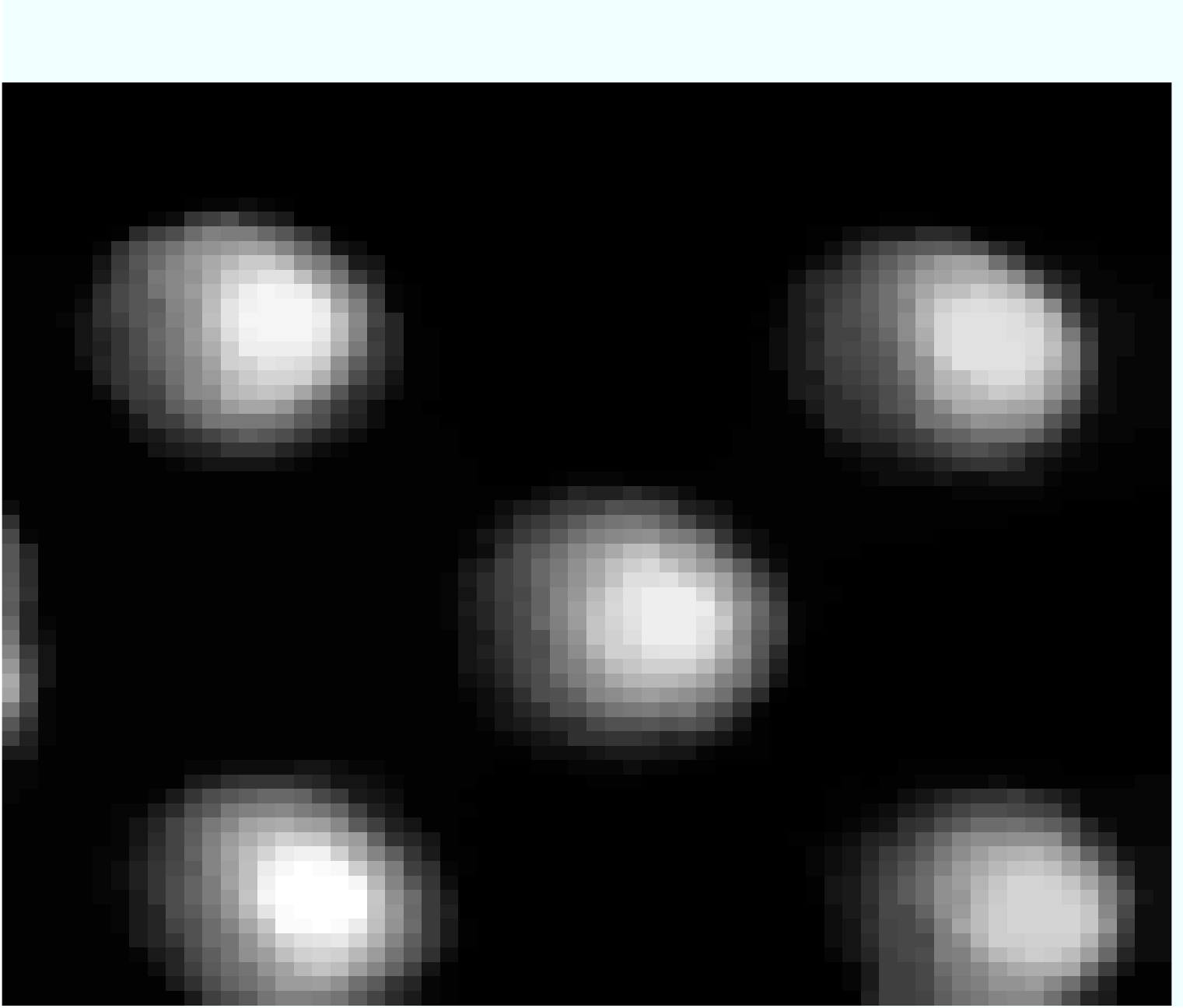}	
		\caption{Dice}
		\label{dice}%
         \end{subfigure}
         \quad
	\begin{subfigure}[b]{0.23\textwidth}
		\includegraphics[width=\textwidth]{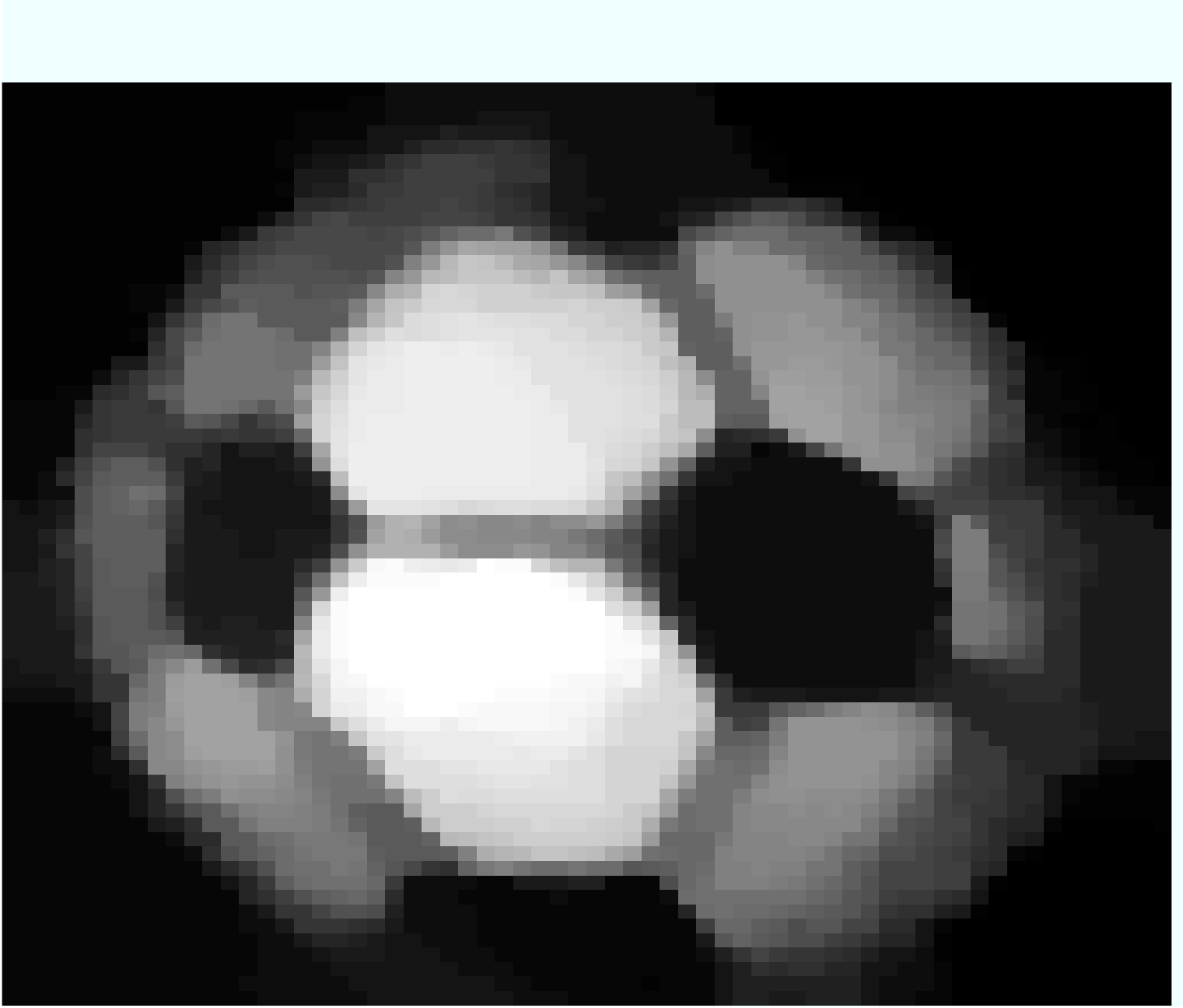}	
		\caption{Ball}
		\label{ball}%
         \end{subfigure}
         \quad
	\begin{subfigure}[b]{0.23\textwidth}
		\includegraphics[width=\textwidth]{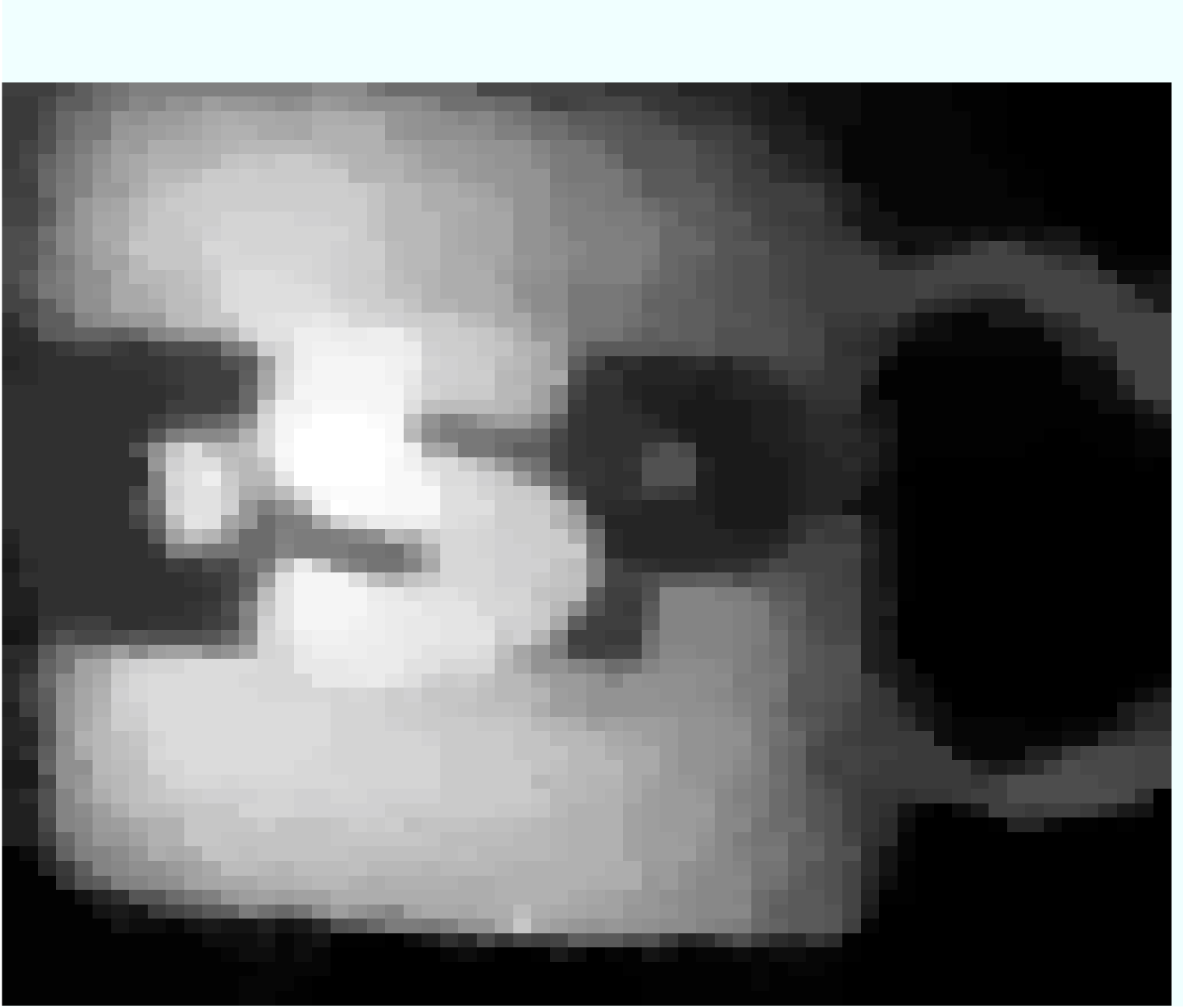}	
		\caption{Cup}
		\label{cup}%
         \end{subfigure}
         \quad
	\begin{subfigure}[b]{0.23\textwidth}
		\includegraphics[width=\textwidth]{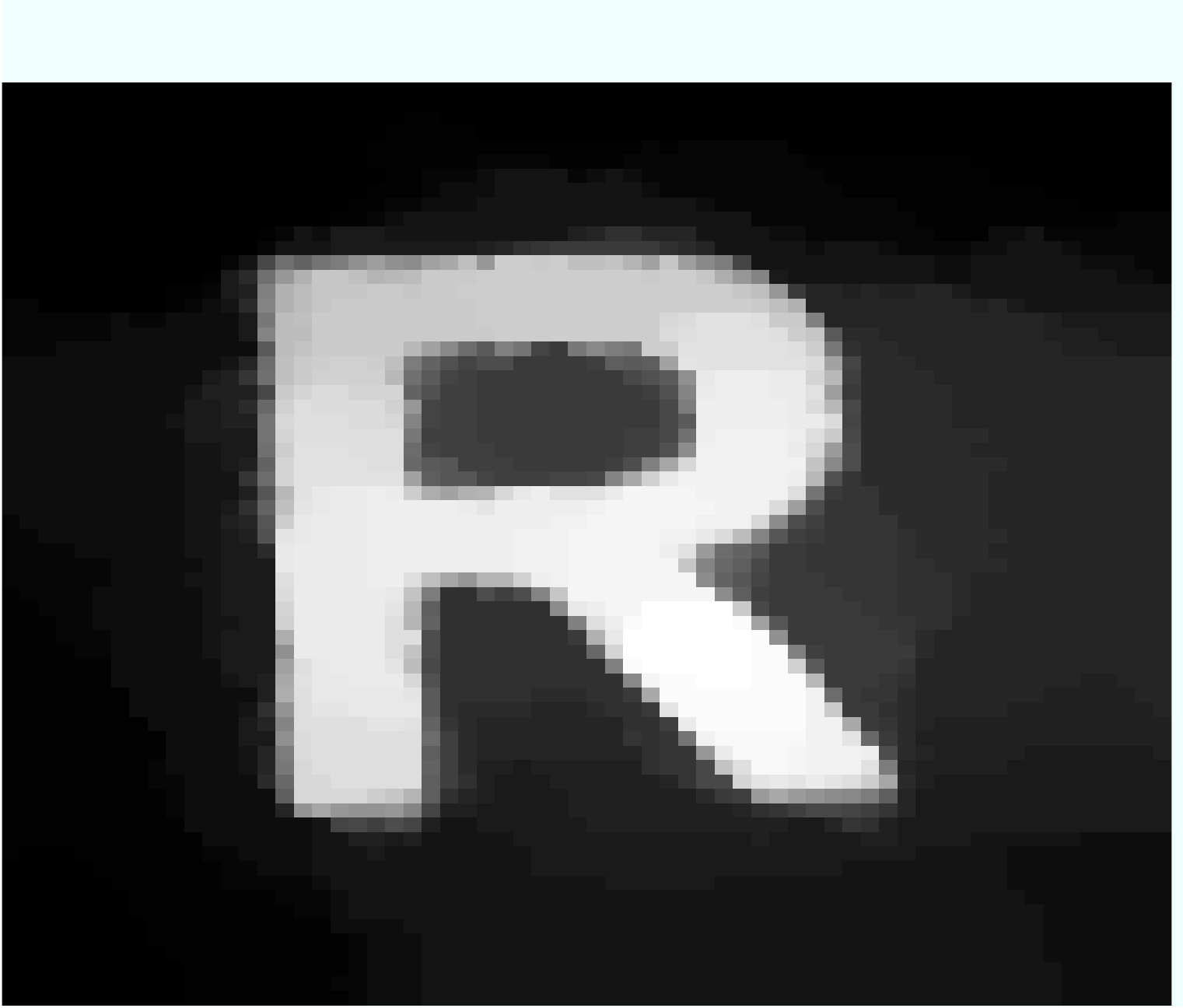}	
		\caption{Letter}
		\label{letter}%
         \end{subfigure}
         \quad
	\begin{subfigure}[b]{0.23\textwidth}
		\includegraphics[width=\textwidth]{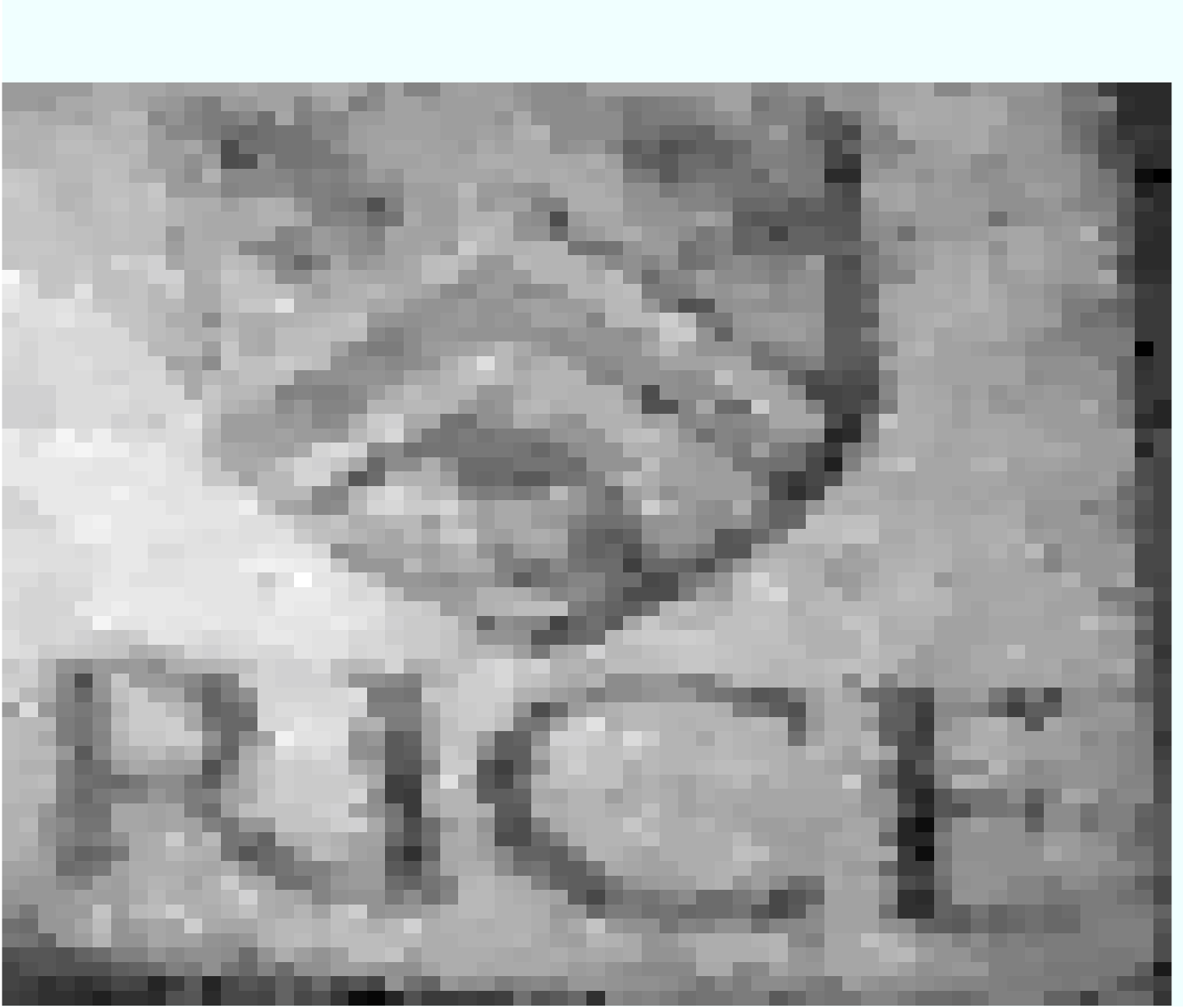}	
		\caption{Logo}
		\label{logo}%
         \end{subfigure}
	\caption{Benchmark images, which were sampled using the single-pixel camera \cite{singlepixel}}
	\label{problemset2}%
\end{figure}

\subsection{Dependence of pdNCG on smoothing parameter}
In this subsection we present the performance of pdNCG with and without preconditioning for decreasing values of the smoothing parameter $\mu$. 
For this experiment we use the images from Figures \ref{House}  to \ref{barbara}. The CS matrix for all experiments is a partial 
Discrete Cosine Transform (DCT) matrix with $m\approx n/4$ and $n$ is equal to the number of pixels of each image in Figure \ref{problemset1}.
For all experiments the sampled signals have PSNR equal to $20$ decibels (dB).

The results of the experiments are shown in Table \ref{table1}. In Table \ref{table1} notice that for the preconditioned case there is always a large increase in CPU time from $\mu=$ $1.0e$-$02$ to $1.0e$-$04$. This is because
for $\mu=$ $1.0e$-$02$ pdNCG relies only on continuation, while for values of $\mu$ equal or smaller than $1.0e$-$04$ preconditioning is necessary
and it is automatically activated using the technique described in Section \ref{sec:cont}. Overall pdNCG had a stable performance with respect to the smoothing parameter $\mu$. 
This is due to the good performance of the proposed preconditioner. Notice that without the preconditioner the performance of pdNCG for $\mu\le$ $1.0e$-$04$ 
worsens noticeably. In particular, for some experiments the unpreconditioned pdNCG required more than $3$ hours of CPU time. For these experiments we forced termination of the method
and we do not report any results. 

\begin{table}
\center
\caption{Performance of pdNCG for decreasing values of the smoothing parameter $\mu$. For this experiment the images from Figures \ref{House} to \ref{barbara} have been used. 
The table shows the CPU time in seconds required for preconditioned pdNCG and unpreconditioned pdNCG for each combination of $\mu$ and problem. PSNR corresponds to the reconstructed solution of pdNCG.}
\begin{tabular}{|c|c|c|c|c|c|c|c|}
\hline
    $\mu$        			     & CG/PCG & House &  Peppers & Lena & Fingerprint & Boat & Barbara  \\ \hline \hline
\multirow{ 3}{*}{$1.0e$-$02$} &PCG        & 4         & 4              & 22     & 19              & 23     & 19 \\
					     &CG           & 4         & 4      	       & 22     & 19              & 22     & 19 \\ 
					     &PSNR      & 19.2    & 24.4         & 25.6  & 18.1           & 24     & 22.6 \\ \hline \hline
\multirow{ 3}{*}{$1.0e$-$04$} &PCG        &  8        & 9              & 40     & 136            & 50     & 42  \\
					     &CG           & 10       & 11            & 85     & 155            & 85     & 57  \\ 
					     &PSNR      & 19.2    & 24.4         & 25.6  & 18.1           & 24     & 22.6 \\ \hline \hline
\multirow{ 3}{*}{$1.0e$-$07$} &PCG        &  10      & 10           &56       & 203            & 57     &85  \\
					     &CG           & 82       & 98           &1165   & 2839          & 1519 &852  \\ 
					     &PSNR      & 19.3    & 24.4        & 25.6   & 18.1           & 24     & 22.6 \\ \hline \hline
\multirow{ 3}{*}{$1.0e$-$10$} &PCG        & 11       & 12           &73       & 182            & 76     &86 \\
					     &CG           & 273     & 195         &3484   & -		       & 3208 &3508 \\ 
					     &PSNR      & 19.3    & 24.4        & 25.6   & 18.1           & 24     &22.6 \\ \hline \hline
\multirow{ 3}{*}{$1.0e$-$13$} &PCG        &  11      & 12           &66       & 232            & 68     &85 \\ 
					     &CG          & 265     & 242         &4834   & - 		       & 5356 & - \\ 
					     &PSNR      & 19.3    & 24.4        & 25.6   & 18.1           & 24     & 22.6 \\
\hline
\end{tabular}
\label{table1}
\end{table}

\subsection{Dependence on problem size}\label{subsec:scaling}
We now present the performance of methods pdNCG, TFOCS, TVAL3 and TwIST as the size of the problem $n$ increases. 
The image from Figure \ref{shepplogan} has been used for this experiment. 
Again, the CS matrix for all experiments is a partial 
Discrete Cosine Transform (DCT) matrix with $m\approx n/4$.
The sampled signals have PSNR equal to $20$ dB.

The results of this experiment are shown in Table \ref{table2}.
Observe that all methods exhibit a linear-like increase in CPU time as a function of the size of the problem. 
We denote with bold the problem for which pdNCG was the fastest method. 

In this table we present 
the PSNR of the recovered solutions for each solver. Notice that TVAL3 does not converge always to a solution of similar PSNR as the other solvers. Although, 
we put significant effort to tune its parameters. A similar performance of TVAL3 is observed for many of the experiments in the subsequent subsections. 
Moreover, observe that TwIST was much slower than the other methods and it did not converge to a solution of similar PSNR to TFOCS or pdNCG for all experiments.
Similar performance for TwIST has been observed in \cite{tval3}.

\begin{table}
\center
\caption{Performance of pdNCG, TFOCS, TVAL3 and TwIST for increasing problem size. The image Shepp-Logan from Figure \ref{shepplogan} has been used for this experiment.
The table shows the required CPU time and the PSNR of the recovered solutions for each solver.}
\begin{tabular}{|c|c|c|c|c|c|c|}
\hline
    Solver        			      & n=  & $64^2$ &  $128^2$ & $256^2$ & $512^2$ & $1024^2$ \\ \hline \hline
\multirow{ 2}{*}{TFOCS\_con} & CPU time (sec)   &16      & 23        & 55       & 264     & 1034 	\\  
		      			      &PSNR  		    &15.8  & 17.4      & 17.9    & 18       & 18 	\\ \hline \hline
\multirow{ 2}{*}{TFOCS\_unc} & CPU time (sec)   &19      & 30        & 79       & 385     & 1477 \\
					      & PSNR 		    &15.8  & 17.4      & 17.9    & 18       & 18 \\ \hline \hline
\multirow{ 2}{*}{TVAL3} 	      & CPU time (sec)   &1       & 2           & 250     & 1365   & 4843 \\
				              & PSNR 		    &15.8  & 17.4      & 17.2    & 17.2    & 17.3\\\hline \hline
\multirow{ 2}{*}{TwIST} 	      & CPU time (sec)   &28       & 135      & 259     & 2149   & 7223 \\
				              & PSNR 		    &13.5  & 15.9      & 16.8    & 16.9    & 16.9\\\hline \hline
\multirow{ 2}{*}{pdNCG}          &CPU time (sec)    &{2}     & {6 }       & \textbf{13}     & \textbf{62}     &  \textbf{237} \\ 
	          			      &PSNR  		    &15.8  & 17.4     & 17.9     & 18       &  18 \\ 
\hline
\end{tabular}
\label{table2}
\end{table}

\subsection{Dependence on the level of noise}
In this experiment we compare the solvers pdNCG, TFOCS and TVAL3 as the level of noise increases. 
We exclude TwIST from this and subsequent experiments due to its poor performance on the simple 
synthetic experiment reported in Subsection \ref{subsec:scaling}; similar performance has been observed in \cite{tval3}.
For this experiment we use the images from Figures \ref{House}  to \ref{barbara}. The CS matrix for all experiments is a partial 
Discrete Cosine Transform (DCT) matrix with $m\approx n/4$.

In Table \ref{table3} we present the performance of the methods for this experiments. In the second column of Table \ref{table3} the PSNR is shown, which is decreasing from $95$ dB to $20$ dB in six steps. The rest of the table shows
the CPU time, which was required by each solver. Overall pdNCG has good performance for problems with large level
of noise, i.e., PSNR equal to $20$ dB. We denote with bold the problems for which pdNCG was the fastest solver. 
In this table we use the star superscript to denote solvers, which solve the unconstrained problem \eqref{prob1} but do not converge 
to a solution of equal or larger PSNR than the solutions of TFOCS\_unc.

In Table \ref{table3_2} we show the PSNR for the solutions calculated by TFOCS\_con for the corresponding experiments in Table \ref{table3}. 
For experiments that are not denoted with a star superscript in Table \ref{table3} the solvers pdNCG, TFOCS\_unc and TVAL3 obtained solutions 
of similar PSNR due to our setting described in Subsection \ref{subset:equiv}.

\begin{table}
\center
\caption{Performance of pdNCG, TFOCS and TVAL3 for increasing level of noise (decreasing PSNR). PSNR is measured in dB. 
 For this experiment the images from Figures \ref{House} to \ref{barbara} have been used. 
The table shows the CPU time in seconds required by each solver. By the star superscript we denote methods that failed to converge to
a similar solution as the other solvers.}
\begin{tabular}{|c|c|c|c|c|c|c|c|}
\hline
    Solver        & PSNR 				   & House          &  Peppers  & Lena                & Fingerprint & Boat   & Barbara  \\ \hline \hline
\multirow{6}{*}{TFOCS\_con}    &    $95$     & $55$  		  & $55$            & $262$          & $261$        & $263$ & $263$	\\  
 		     				&    $80$     & $55$  		  & $55$            & $261$          & $263$  	   & $265$ &$262$ 	\\  
 		     				&    $65$     & $55$  		  & $55$            & $262$          & $264$  	   & $265$ & $262$ 	\\  
 		    				 &    $50$    & $55$  		  & $55$            & $262$          & $262$  	   & $262$ &$262$ 	\\  
 		     				&     $35$    & $55$  		  & $55$            & $261$          & $262$  	   & $262$ & $262$	\\  
 		     				&     $20$    & $54$  		  & $55$            & $263$          & $262$  	   & $262$ &$261$ 	\\ \hline \hline
\multirow{6}{*}{TFOCS\_unc}    &    $95$      &$76$  		  & $76$            & $382$          & $381$  	   & $390$ & $382$ \\ 
 		     				&    $80$     & $76$  		  & $76$            & $383$          & $385$  	   & $388$ & $381$	\\  
 		     				&    $65$     & $76$  		  & $76$            & $384$          & $399$  	   & $386$ & $381$	\\  
 		     				&    $50$     & $75$  		  & $76$            & $383$          & $383$  	   & $382$ & $381$	\\  
 		     				&    $35$     & $76$  		  & $76$            & $382$          & $382$  	   & $383$ & $382$	\\  
 		     				&    $20$     & $76$  		  & $76$            & $382$          & $379$  	   & $380$ &$379$ 	\\ \hline \hline																							
\multirow{6}{*}{TVAL3}              &    $95$     & $3$     		  & $3$              & $16$            & $10$    	   & $38$   & $55$\\
 		     				&    $80$     & $3$   		  & $3$              & $16$            & $10$    	   & $38$   & $55$ 	\\  
 		     				&    $65$     & $3$   		  & $4$              & $10$            & $10$    	   & $37$   & $56$	\\  
 		     				&    $50$     & $3$  		  & $3$              & $383$           &$10$    	   & $24$   & $44$	\\  
 		     				&    $35$     & $3$     	  & $3$              & $1104^*$      & $8$     	   & $36$   & $1067^*$	\\  
 		     				&    $20$     & $184^*$        & $185^*$        & $999^*$       & $8$     	   & $10$   & $ 987^*$	\\  \hline \hline	
\multirow{6}{*}{pdNCG}             &     $95$     &$20$  		  & $82$             & $145$          & $141$  	   & $182$ & $209$ \\ 
 		     				&    $80$     & $20$  		  & $82$             & $146$          & $142$  	   & $182$ &$208$ 	\\  
 		     				&    $65$     & $20$  	 	  & $243$           & $145$          & $142$  	   & $178$ & $209$ 	\\  
 		     				&    $50$     & $20$             & $55$             & \textbf{169}  & $142$  	   & $116$ & $209$	\\  
 		     				&    $35$     & $13$             & $38$             & \textbf{109}  & $138$  	   & $111$  & \textbf{202}	\\  
 		     				&    $20$     & \textbf{16}     & \textbf{16}     & \textbf{101} & $156$   	   & $103$ & \textbf{105} 	\\  
\hline
\end{tabular}
\label{table3}
\end{table}

\begin{table}
\center
\caption{This table shows the PSNR for the solutions calculated by TFOCS\_con for the corresponding experiments in Table \ref{table3}
 For this experiment the images from Figures \ref{House} to \ref{barbara} have been used.}
\begin{tabular}{|c|c|c|c|c|c|c|c|}
\hline
    Solver        				& PSNR 	   & House           &  Peppers        & Lena           & Fingerprint       & Boat    & Barbara  \\ \hline \hline
\multirow{6}{*}{TFOCS\_con}    &    $95$     & $20$  		  & $31.9$            & $29.5$          & $20.1$           & $27.6$ & $25$	\\  
 		     				&    $80$     & $20$  		  & $31.9$            & $29.5$          & $20.1$  	   & $27.6$ &$25$ 	\\  
 		     				&    $65$     & $20$  		  & $31.8$            & $29.5$          & $20.1$  	   & $27.6$ & $25$ 	\\  
 		    				 &    $50$    & $20$  		  & $31.6$            & $29.5$          & $20.1$  	   & $27.5$ &$24.9$ 	\\  
 		     				&     $35$    & $19.9$          & $29.7$            & $28.6$          & $19.7$  	   & $26.9$ & $24.7$	\\  
 		     				&     $20$    & $19.2$          & $24.3$            & $25.6$          & $17.9$  	   & $24$    &$22.6$ 	\\
\hline
\end{tabular}
\label{table3_2}
\end{table}

\subsection{Dependence on number of measurements}
In this experiment we compare the three methods for a decreasing number of measurements $m$. 
For this experiment we use the images from Figures \ref{House}  to \ref{barbara}. The CS matrix is a partial 
Discrete Cosine Transform (DCT) matrix. For all experiments the sampled signals have PSNR equal to $20$ dB.

The results of this experiment are shown in Table \ref{table4}. We denote with bold the problems for which pdNCG
was the fastest method. In Table \ref{table4_2} we show the PSNR for the solutions calculated by TFOCS\_con for the corresponding experiments in Table \ref{table4}. 
For experiments that are not denoted with a star superscript in Table \ref{table4} the solvers pdNCG, TFOCS\_unc and TVAL3 obtained solutions 
of similar PSNR due to our setting described in Subsection \ref{subset:equiv}.

\begin{table}
\center
\caption{Performance of pdNCG, TFOCS and TVAL3 for decreasing number of measurements $m$. 
In the second column the percentage of measurements is shown, for example $75\%$ means that $m\approx 3n/4$,
where $n$ is the number of pixels in the image to be reconstructed. 
 For this experiment the images from Figures \ref{House} to \ref{barbara} have been used. 
The table shows the CPU time in seconds required by each solver. By the star superscript we denote methods that failed to converge to
a similar solution as the other solvers.}
\begin{tabular}{|c|c|c|c|c|c|c|c|}
\hline
    Solver        & m 					     & House           &  Peppers  & Lena         & Fingerprint  & Boat 		& Barbara  \\ \hline \hline
\multirow{3}{*}{TFOCS\_con}    &    75\%       & 58        	    & 58             & 273      	   & 272  	         & 272 		& 273	\\  
 		     				&    50\%       & 56       	    & 56   	        & 268 	   & 266  	 	 & 268 		&270 	\\  
 		     				&    25\%       & 54       	    & 55   	       & 263 	   & 261  		 & 265 		&263 	\\  \hline \hline
\multirow{3}{*}{TFOCS\_unc}    &    75\%       &78                   & 78   	       & 392 	   & 388  	         & 393 		& 396 \\ 
 		     				&    50\%      & 77                  & 77   	       & 389            & 383 	  	 & 387 		& 389	\\  
 		     				&    25\%      & 76                  & 76   	       & 382            & 378  		  & 386 		 & 380	\\  \hline \hline																						
\multirow{3}{*}{TVAL3}              &    75\%       &$296^*$         & $301^*$    &$1584^*$    & $1587^*$     & $1497^*$    	 & $1250^*$   \\
 		     				&    50\%      & $262^*$        & $253^*$    & 8  		   & 7          	  & 7    		 &10 	\\  
 		     				&    25\%      & $184^*$        & $184^*$    & $998^*$      & 8    		  & 10    		 & $986^*$	\\  \hline \hline	
\multirow{3}{*}{pdNCG}             &     75\%     &\textbf{23}      & \textbf{23} & \textbf{132} & \textbf{134}  & \textbf{136}  & \textbf{140} \\ 
 		     				&    50\%      & \textbf{21}     & \textbf{21} & 131  	    & 162  		   & 135             &137 	\\  
 		     				&    25\%      & \textbf{16}     & \textbf{16} & \textbf{101} & 156  		   & 104             & \textbf{105} 	\\ 
\hline
\end{tabular}
\label{table4}
\end{table}

\begin{table}
\center
\caption{This table shows the PSNR for the solutions calculated by TFOCS\_con  for the corresponding experiments in Table \ref{table4}.
 For this experiment the images from Figures \ref{House} to \ref{barbara} have been used.}
\begin{tabular}{|c|c|c|c|c|c|c|c|}
\hline
    Solver        & m 					     & House           &  Peppers  & Lena         & Fingerprint  & Boat 		& Barbara  \\ \hline \hline
\multirow{3}{*}{TFOCS\_con}    &    75\%       & $19.6$           & $27.9$     & $27.2$       & $19.9$  	 & $25.9$ 		& $24.8$	\\  
 		     				&    50\%       & $19.5$          & $26.6$      & $26.7$ 	   & $19.2$  	 & $25.2$ 		&$23.8$ 	\\  
 		     				&    25\%       & $19.2$          & $24.3$      & $25.6$ 	   & $17.9$  	 & $24$ 		&$22.6$ 	\\ 
\hline
\end{tabular}
\label{table4_2}
\end{table}

\subsection{Single-pixel camera}
We now compare TFOCS with pdNCG on realistic image reconstruction problems where the data have been sampled
using a single-pixel camera \cite{singlepixel}. In this experiment we compare our solver only with TFOCS\_con. This is because 
in all previous experiments TFOCS\_con was faster than TFOCS\_unc. Additionally, we were not able to make TVAL3 
to converge to a solution which was as visually pleasant as the solutions obtained by TFOCS\_con and pdNCG. We believe
that this is due to the different CS matrix $A$ in these experiments. In particular, matrix $A\in\mathbb{R}^{m\times n}$, where $n=64^2$ and $m\approx 0.4 n$, 
is a partial Walsh basis which takes values $0/1$ instead of $\pm 1$. We noticed that this matrix $A$ does not satisfy the RIP property in Definition \ref{def:1}
with small $\delta_q$. Therefore, the least squares term in problem \eqref{prob1} might be ill-conditioned and this causes difficulties for TVAL3.

Moreover the optimal solutions are unknown and additionally the level of noise is unknown. 
Hence the reconstructed images can only be compared 
by visual inspection.
For all four experiments $40\%$ of measurements are selected uniformly at random. 

The reconstructed images by the solvers TFOCS\_con and pdNCG are presented in Figure \ref{figcscamera}. Solver pdNCG was faster on four out of five problems.
On problems that pdNCG was faster it required on average $1.5$ times less CPU time. Although it would be possible to tune pdNCG such that it is faster on all problems, we
preferred to use its (simple) default tuning in order to avoid a biased comparison. 
\begin{figure}%
\centering
	\begin{subfigure}[b]{0.26\textwidth}
		\includegraphics[width=\textwidth]{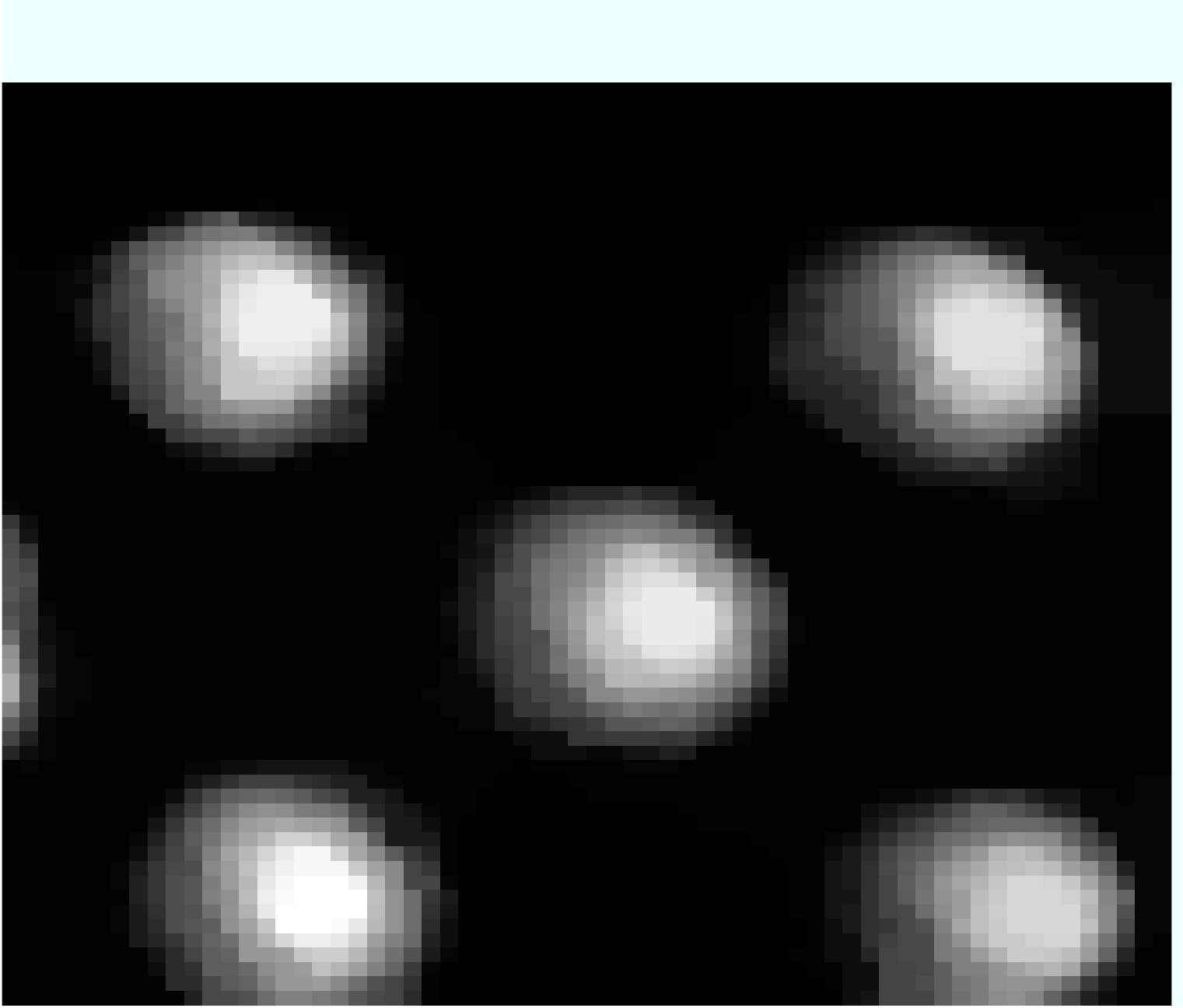}		
		\caption{TFOCS\_con, $25$ sec.}
		\label{figcscamera_a}%
         \end{subfigure}
         \quad
	\begin{subfigure}[b]{0.26\textwidth}
		\includegraphics[width=\textwidth]{fig7_3_b}	
		\caption{pdNCG, $7$ sec.}
		\label{figcscamera_b}%
         \end{subfigure}
         \quad
	\begin{subfigure}[b]{0.26\textwidth}
		\includegraphics[width=\textwidth]{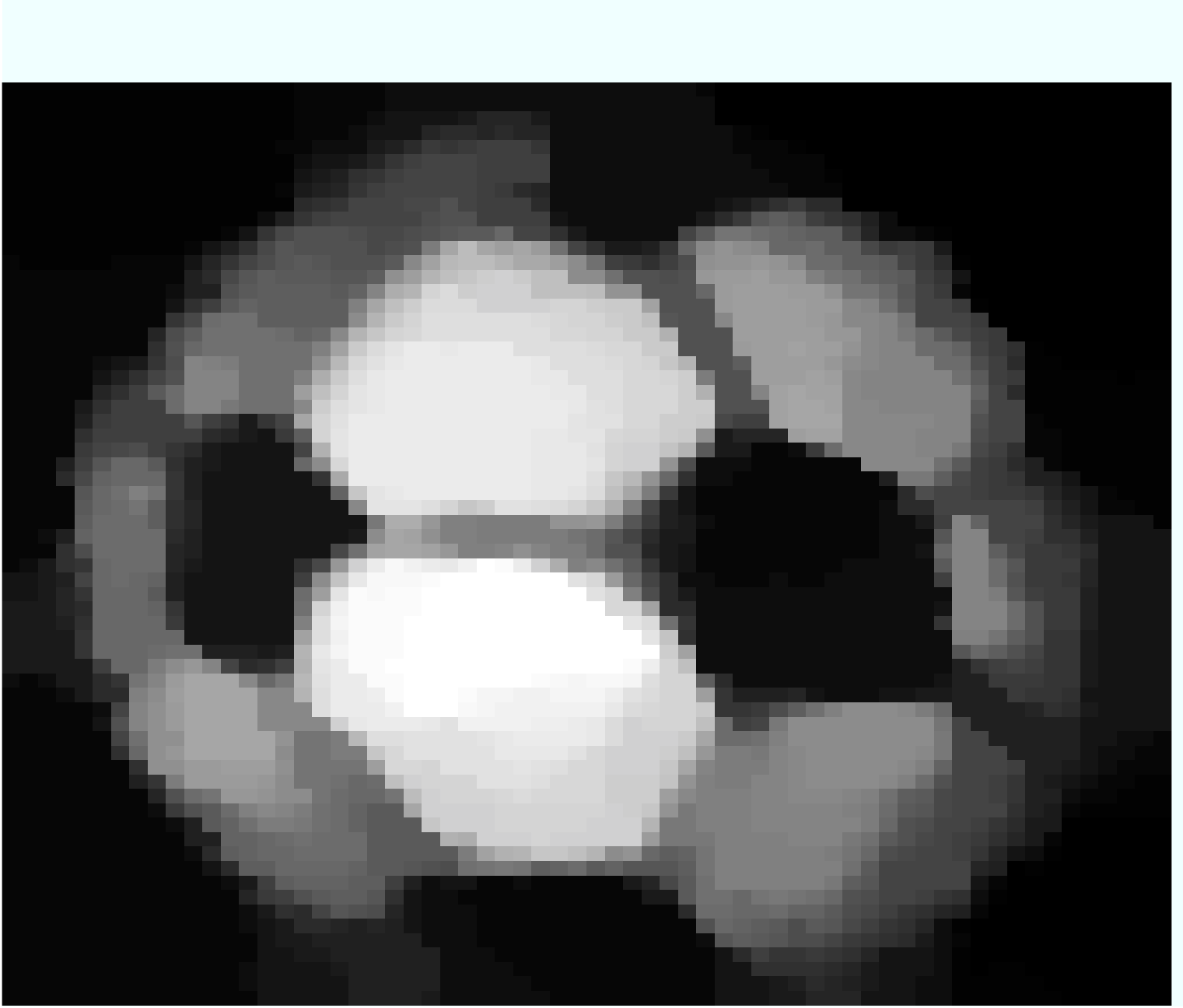}		
		\caption{TFOCS\_con, $24$ sec.}
		\label{figcscamera_c}%
         \end{subfigure}
         \\
	\begin{subfigure}[b]{0.26\textwidth}
		\includegraphics[width=\textwidth]{fig7_3_d}	
		\caption{pdNCG, $15$ sec.}
		\label{figcscamera_d}%
         \end{subfigure}
         \quad
	\begin{subfigure}[b]{0.26\textwidth}
		\includegraphics[width=\textwidth]{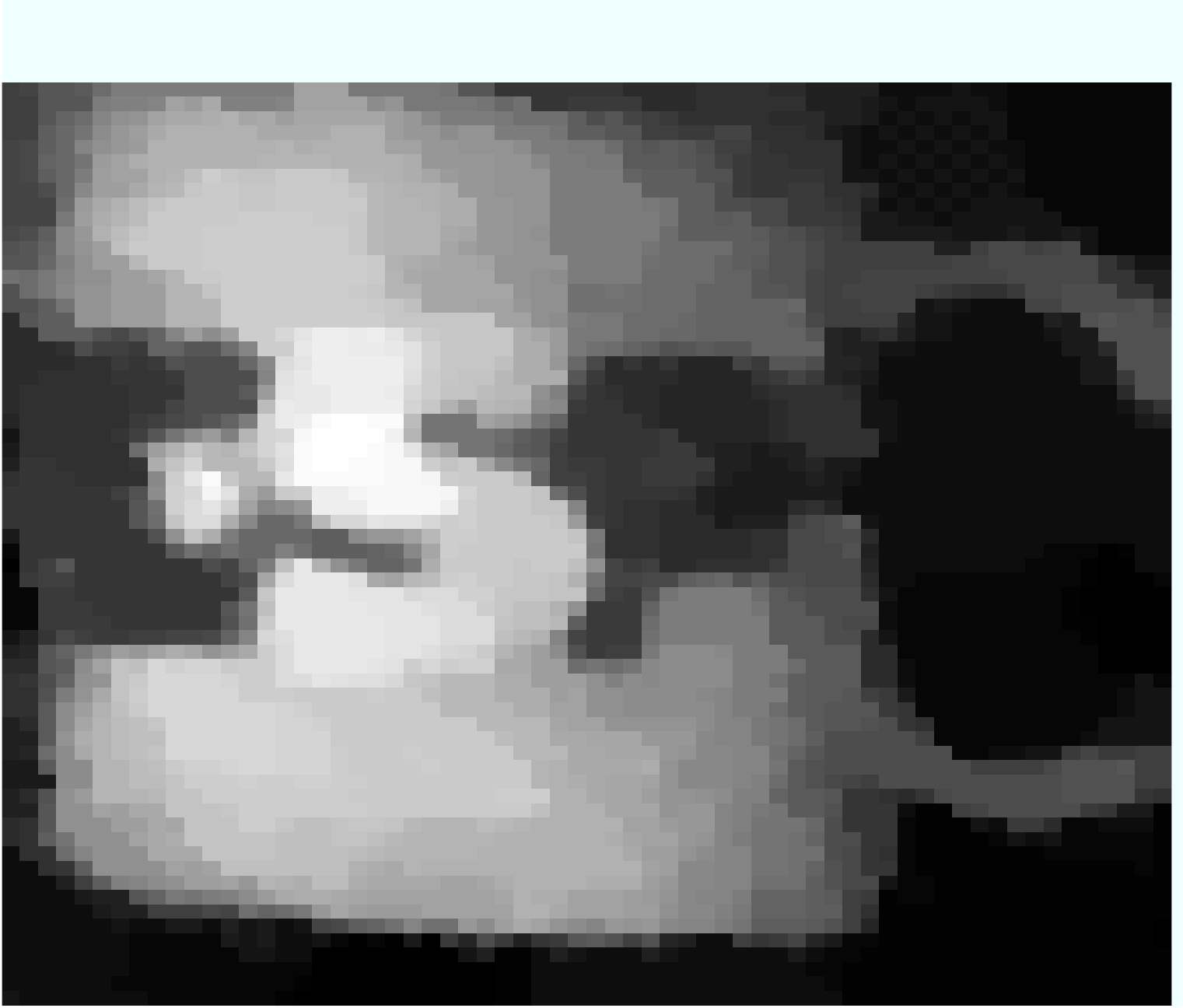}	
		\caption{TFOCS\_con, $37$ sec.}
		\label{figcscamera_e}%
         \end{subfigure}
         \quad
	\begin{subfigure}[b]{0.26\textwidth}
		\includegraphics[width=\textwidth]{fig7_3_f}	
		\caption{pdNCG, $15$ sec.}
		\label{figcscamera_f}%
         \end{subfigure}
         \\
	\begin{subfigure}[b]{0.26\textwidth}
		\includegraphics[width=\textwidth]{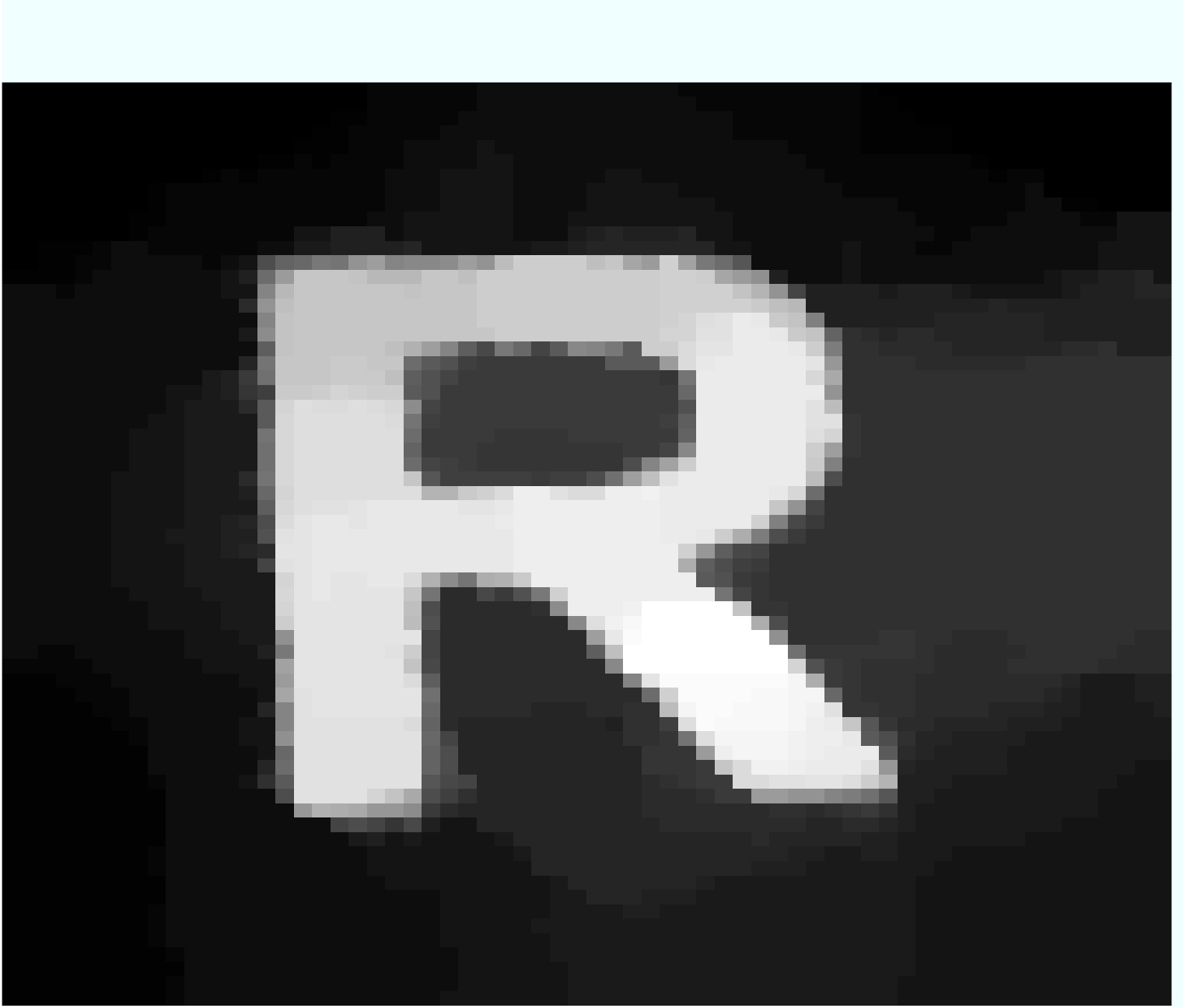}	
		\caption{TFOCS\_con, $26$ sec.}
		\label{figcscamera_g}%
         \end{subfigure}
         \quad
	\begin{subfigure}[b]{0.26\textwidth}
		\includegraphics[width=\textwidth]{fig7_3_h}	
		\caption{pdNCG, $27$ sec.}
		\label{figcscamera_h}%
         \end{subfigure}
         \quad
	\begin{subfigure}[b]{0.26\textwidth}
		\includegraphics[width=\textwidth]{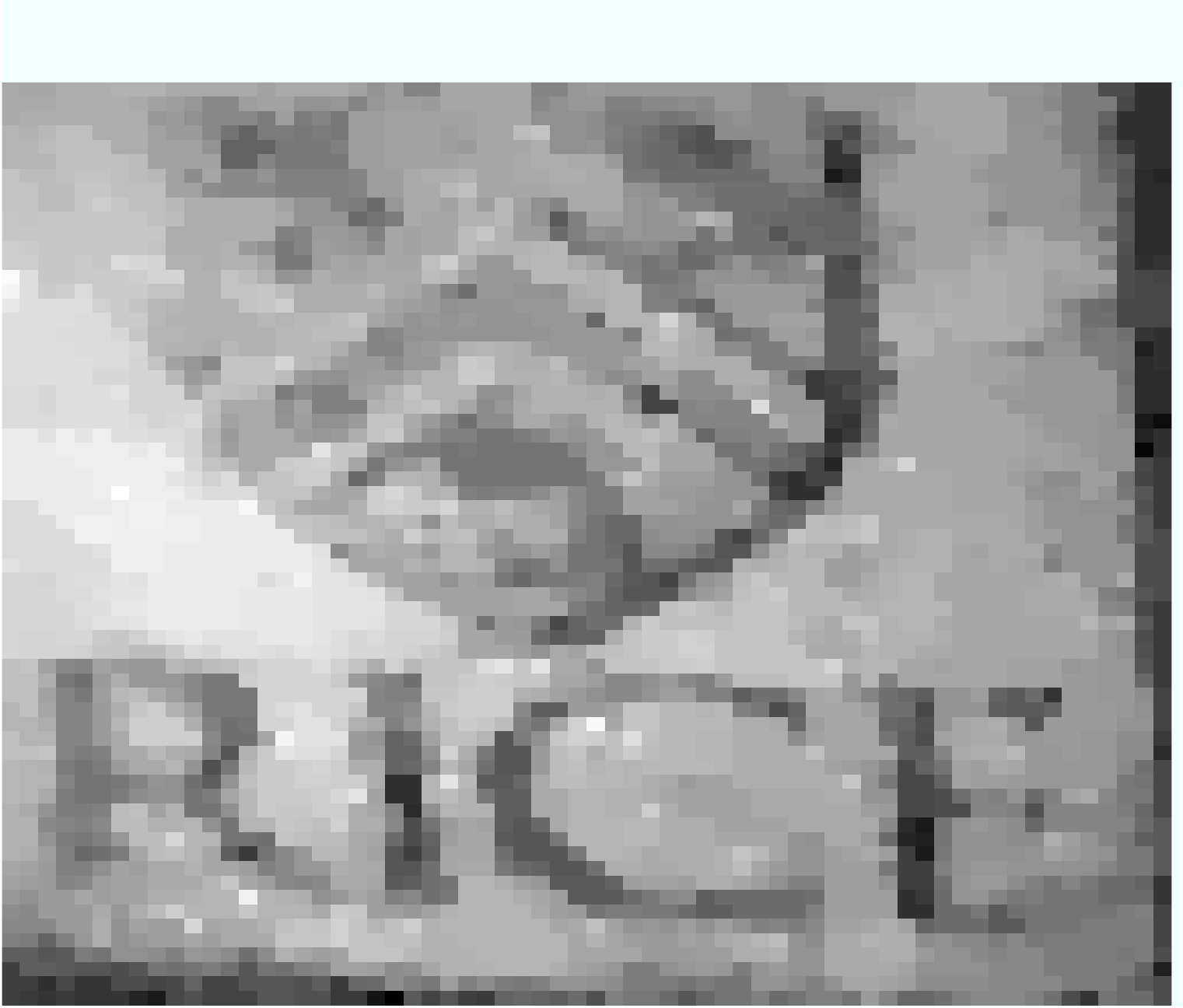}	
		\caption{TFOCS, $49$ sec.}
		\label{figcscamera_j}%
         \end{subfigure}
         \quad
	\begin{subfigure}[b]{0.26\textwidth}
		\includegraphics[width=\textwidth]{fig7_3_k}	
		\caption{pdNCG, $33$ sec.}
		\label{figcscamera_k}%
         \end{subfigure}
	\caption{Experiment on realistic image reconstruction where the samples are acquired using a single-pixel camera.
	The subcaptions of the figures show the required seconds of CPU time for the image to be reconstructed for each solver.
	}
	\label{figcscamera}%
\end{figure}

\newpage

\subsection{Radar tone reconstruction}
In this subsection we present the performance of TFOCS\_unc and pdNCG for the radar tone reconstruction problem, which
was described in Subsection \ref{subsec:problemsets}. We exclude TVAL3 from this experiment because it is implemented 
to solve only TV problems. We also exclude TFOCS\_unc since it is superseded in all previous experiments by TFOCS\_con. 

The results of the comparison are presented in Figure \ref{fig7_4}. 
Observe in Figures \ref{fig7_4_a} and \ref{fig7_4_b} that both solvers recovered a solution of similar accuracy but pdNCG was slightly faster. 
The solution of TFOCS\_con has SNR $62.3$ dB
and the solution of pdNCG has SNR $64.7$ dB.
It is important to mention that the problems 
were not over-solved. TFOCS\_con was tuned as suggested by its authors in a similar experiment which is shown in Subsection $6.5$ of \cite{convexTemplates}. 

In Figure \ref{fig7_4_c} we plot the SNR against CPU time for every iteration of pdNCG and TFOCS\_con. Observe that nearly for all iterations pdNCG the approximate 
solutions of pdNCG had larger SNR than the approximate solutions of TFOCS\_con.

For this experiment we enabled preconditioning for pdNCG. Although, in Subsection \ref{subsec:howsolvesystems} we mentioned that preconditioning might 
affect adversely the performance of pdNCG in terms of CPU time. However, we noticed that by enabling preconditioning pdNCG was more stable. Therefore, 
we believe that it is worth paying the cost of a slight increase of a CPU time to improve the overall robustness of pdNCG.

\begin{figure}%
\centering
	\begin{subfigure}[b]{0.48\textwidth}
		\includegraphics[width=\textwidth]{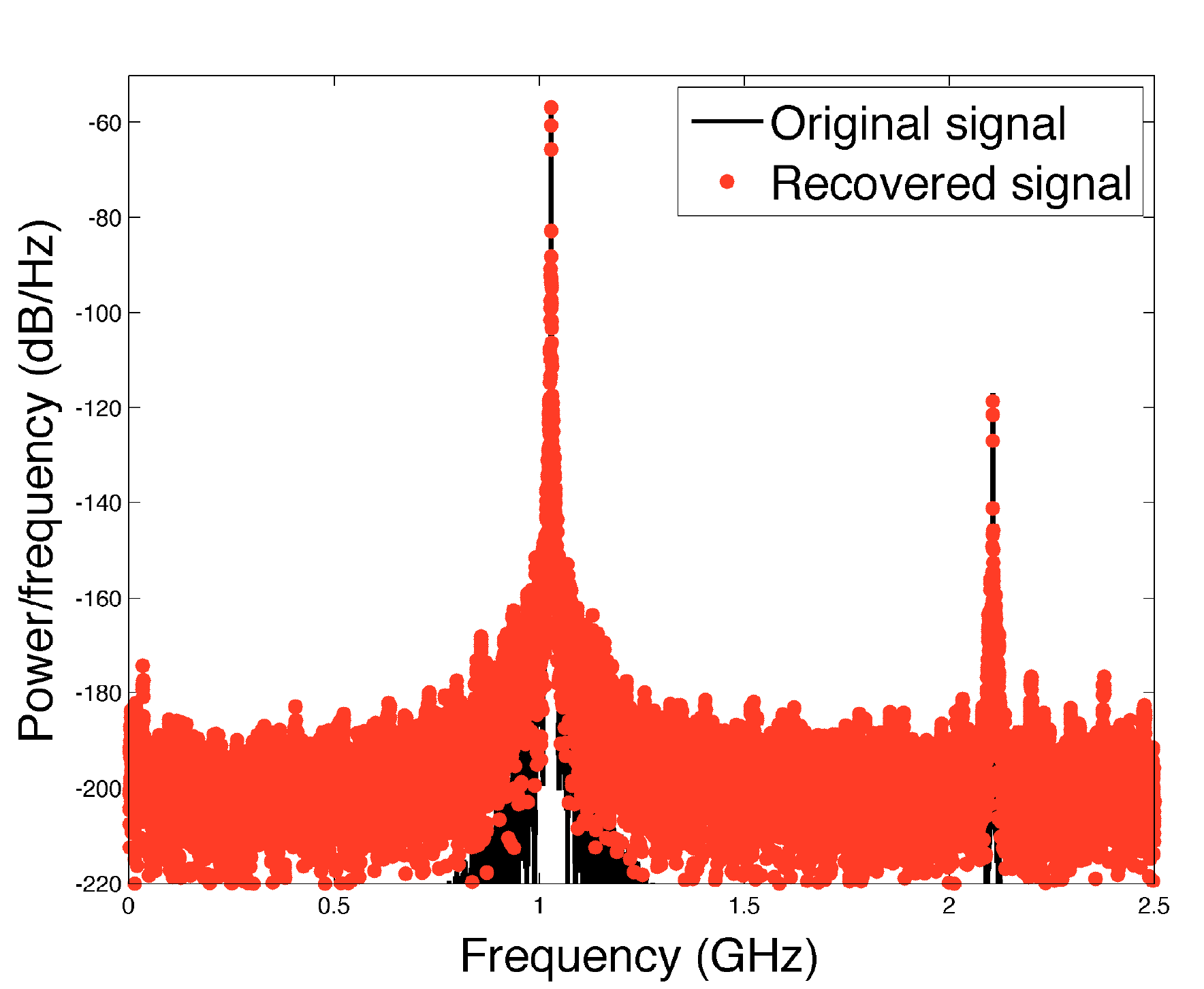}		
		\caption{pdNCG, SNR $64.7$, 525 sec}
		\label{fig7_4_a}%
         \end{subfigure}
         \quad
	\begin{subfigure}[b]{0.48\textwidth}
		\includegraphics[width=\textwidth]{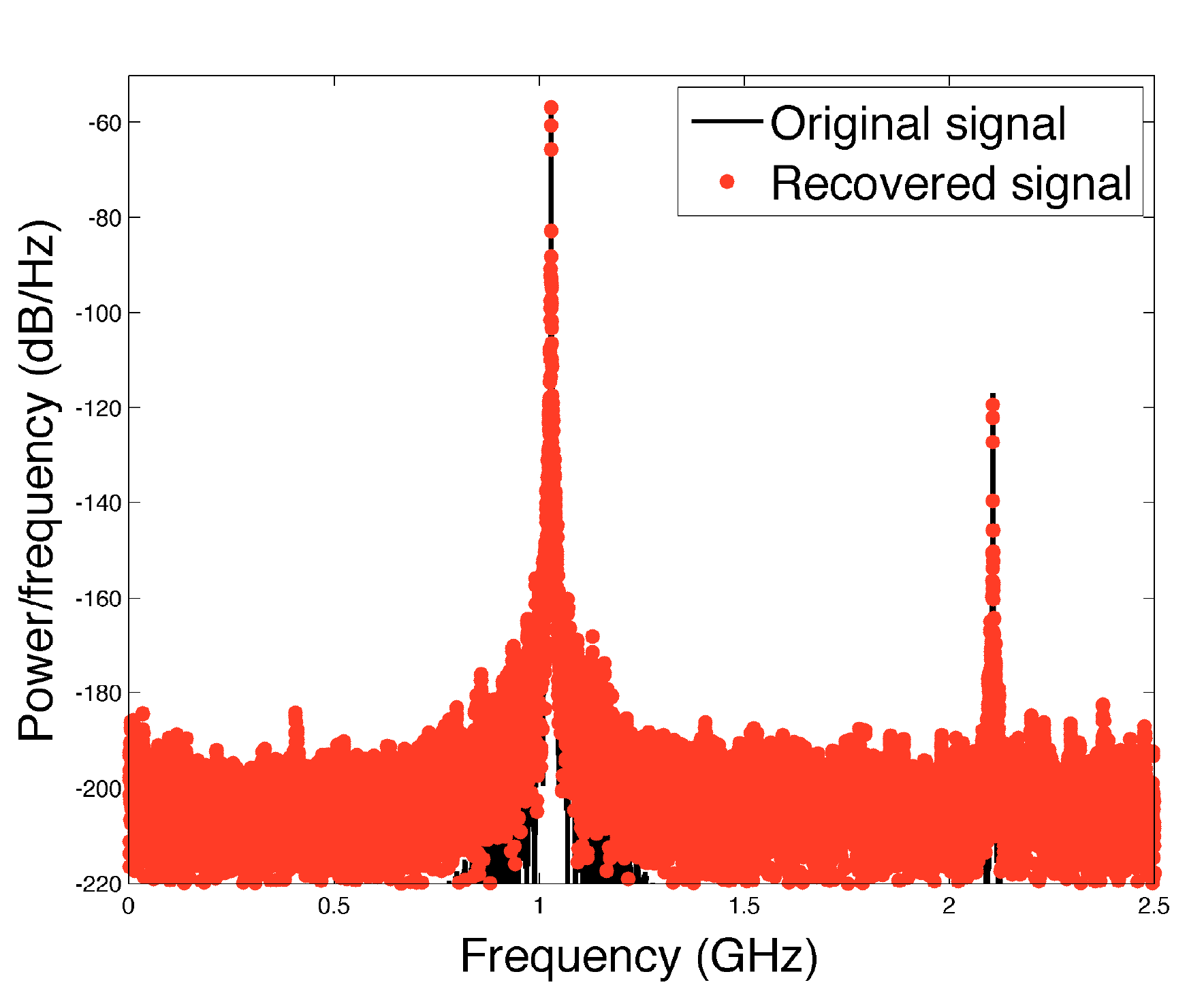}	
		\caption{TFOCS\_con, SNR $62.3$, 556 sec}
		\label{fig7_4_b}%
         \end{subfigure}
         \\
	\begin{subfigure}[b]{0.48\textwidth}
		\includegraphics[width=\textwidth]{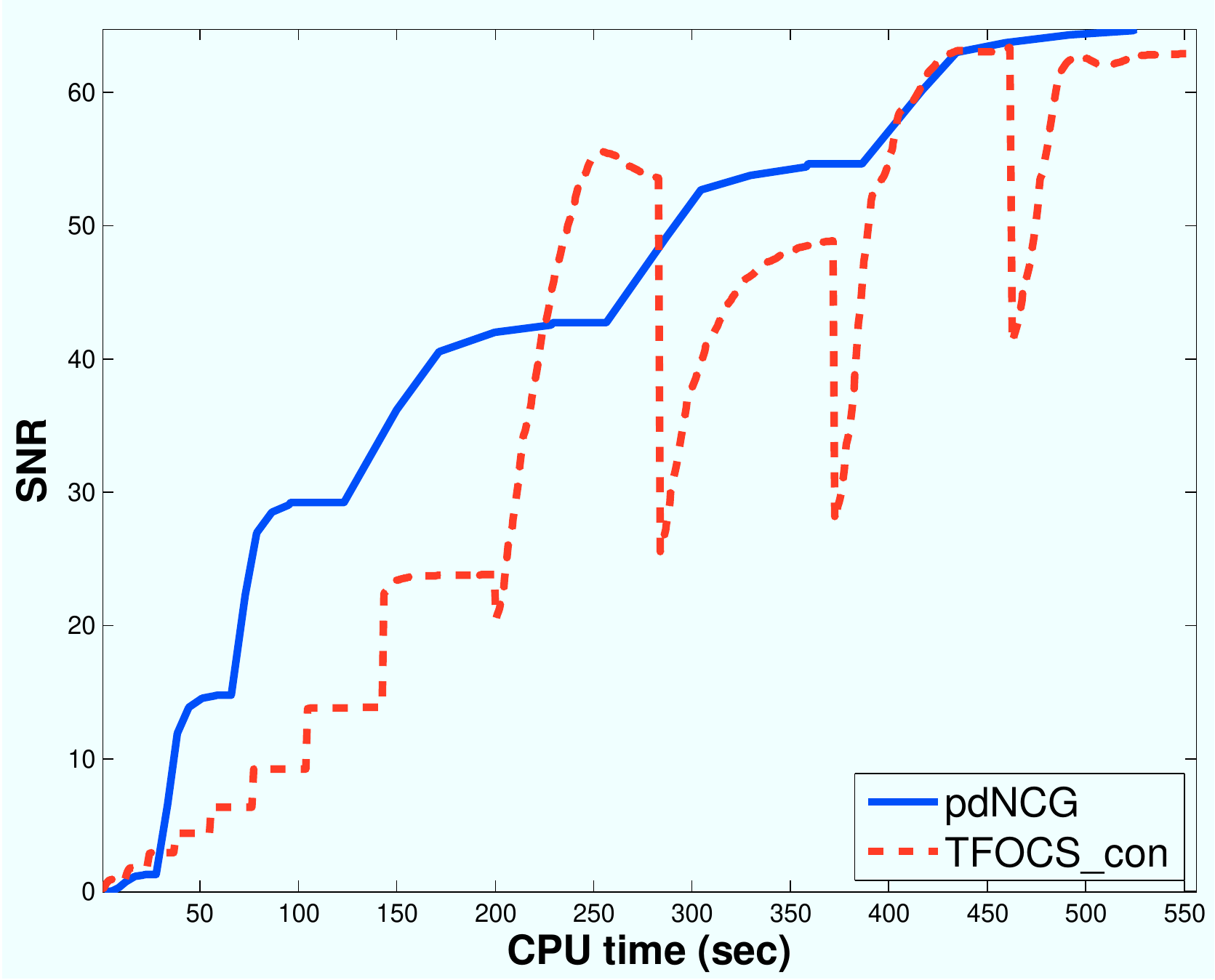}	
		\caption{SNR against CPU time}
		\label{fig7_4_c}%
         \end{subfigure}
	\caption{Reconstruction of two radio-frequency radar tones by TFOCS\_con and pdNCG. In Figure \ref{fig7_4_a} the reconstructed signal by pdNCG is shown. The signal has SNR $64.7$ dB
	and it required $525$ seconds of CPU time to be reconstructed. In Figure \ref{fig7_4_b}
	the reconstructed signal by TFOCS\_con is shown. The signal has PSNR $62.3$ dB and it required $556$ seconds of CPU time to be reconstructed. In Figure \ref{fig7_4_c} we present 
	the SNR for each iteration of pdNCG and TFOCS\_con against CPU time.
	}
	\label{fig7_4}%
\end{figure}

\section{Conclusions}\label{sec:concl}
Recently there has been great interest in the development of optimization methods for the solution of compressed sensing 
problems.
The methods that have been developed so far are mainly first-order methods.
This is because first-order methods have inexpensive iterations and frequently 
offer fast initial progress in the optimization process. On the contrary, second-order methods are considered to be rather expensive. 
The reason is that often access to second-order information requires the solution of linear systems. In this paper we 
develop a second-order method, a primal-dual Newton Preconditioned Conjugate Gradients. We show
that an \textit{approximate} solution of linear systems which arise is sufficient to speed up an iterative method and additionally
make it more robust. Moreover, we show that for compressed sensing problems an inexpensive preconditioner can be designed that
speeds up even further the approximate solution of linear systems. Extensive numerical experiments are presented which support our findings. 
Spectral analysis of the preconditioner is performed and shows its very good limiting behaviour.

\section*{Acknowledgement} We are grateful to the anonymous reviewers who made crucial recommendations which improved the clarity 
of contribution and the quality of this paper.


\bibliographystyle{plain}
\bibliography{KFandJG.bib}

\appendix

\section{Continuous path}
In the following lemma we show that 
$
x_{c,\mu} := \argmin f_c^\mu(x)
$ ($f_c^\mu$ is defined in \eqref{prob2}) 
for $c$ constant is a continuous and differentiable function of $\mu$.
\begin{lemma}\label{lem:7}
Let $c$ be constant and consider $x_{c,\mu}$ as a functional of $\mu$. If condition \eqref{bd54} is satisfied, then
$x_{c,\mu}$ is continuous and differentiable.
\end{lemma}
\begin{proof}
The optimality conditions of problem \eqref{prob2} are 
$$
c\nabla \psi_\mu(W^*x) +A^\intercal (Ax-b) = 0.
$$ 
According to definition of $x_{c,\mu}$, we have 
\begin{align*}
c\nabla \psi_\mu(W^*x_{c,\mu}) +A^\intercal (Ax_{c,\mu}-b) & = 0 & \Longrightarrow \\
c \frac{d \nabla \psi_\mu(W^*x_{c,\mu}) }{d \mu} + A^\intercal A \frac{d x_{c,\mu}}{d\mu} & =0 & \Longrightarrow \\
c \Big(\nabla^2 \psi_\mu(W^*x_{c,\mu}) \frac{d x_{c,\mu}}{d\mu} + \frac{d \nabla \psi_\mu(W^*x)}{d\mu} \Big|_{x_{c,\mu}}\Big) + A^\intercal A \frac{dx_{c,\mu}}{d\mu} & =0 & \Longleftrightarrow \\ 
\Big(c\nabla^2 \psi_\mu(W^*x_{c,\mu}) + A^\intercal A\Big)\frac{d x_{c,\mu}}{d\mu} + c\frac{d \nabla \psi_\mu(W^*x)}{d\mu} \Big|_{x_{c,\mu}} & =0 &\Longleftrightarrow \\
\nabla^2 f_c^\mu(W^*x_{c,\mu})\frac{d x_{c,\mu}}{d\mu} + c\frac{d \nabla \psi_\mu(W^*x)}{d\mu} \Big|_{x_{c,\mu}} & =0, & 
\end{align*}
where ${d \nabla \psi_\mu(W^*x)}/{d\mu} |_{x_{c,\mu}}$ is the first-order derivative of $\nabla \psi_\mu(W^*x)$ as a functional of $\mu$, measured at $x_{c,\mu}$. Notice that due to condition 
$
\mbox{Ker}(W^*)\cap\mbox{Ker}(A)=\{0\}
$ 
we have that $\nabla^2 f_c^\mu(x)$ is positive definite $\forall x$, hence $x_{c,\mu}$ is unique. Therefore, the previous system has a unique solution, which means that
$x_{c,\mu}$ is uniquely differentiable as a functional of $\mu$ with $c$ being constant. Therefore, $x_{c,\mu}$ is continuous as a functional of $\mu$. 
\end{proof}

\begin{remark}\label{rem:3}
Lemma \ref{lem:7} and continuity imply that there exists sufficiently small smoothing parameter $\mu$ such that $\|x_{c,\mu}-x_c\|_2< \omega$ for any arbitrarily small $\omega>0$.
\end{remark}

\end{document}